\documentclass[absolute]{mymathart}
\usepackage{mymathmacros}
\usepackage[shortlabels]{enumitem}
\usepackage{gensymb}

\usepackage{caption,subcaption}
\usepackage[dvipsnames]{xcolor}
\usepackage[colorlinks=true,linkcolor=blue!60!black,citecolor=teal!80!black,urlcolor=RoyalBlue]{hyperref}

\def\CB{{\mathcal{CB}}}
\newcommand{\distH}{\operatorname{H-dist}}
 \newtheoremstyle{numberedstyle}% name
   {9pt}%      Space above, empty = `usual value'
   {9pt}%      Space below
   {\normalfont}% Body font
   {}%         Indent amount (empty = no indent, \parindent = para indent)
   {\bfseries}% Thm head font
   {.}%        Punctuation after thm head
   {\newline}% Space after thm head: \newline = linebreak
   {}%         Thm head spec

\newcommand{\T}{\mathcal{T}}

\newcommand{\sep}{\operatorname{sep}}

\numberwithin{equation}{section}

\newtheorem{thm}{Theorem}[section]%
\newtheorem{lem}[thm]{Lemma}%
\newtheorem{cor}[thm]{Corollary}%
\newtheorem{prop}[thm]{Proposition}%

\newcommand{\HH}{\mathbb{H}}

\newcommand{\id}{\operatorname{id}}

\theoremstyle{definition}
\newtheorem{rmk}[thm]{Remark}%
\newtheorem{question}[thm]{Question}%
\newtheorem{defn}[thm]{Definition}%
\newtheorem{obs}[thm]{Observation}%

\newcommand{\B}{\mathcal{B}}
\renewcommand{\H}{\mathbb{H}}

\newcommand{\Ima}{\operatorname{Im}}

\newcommand{\Rea}{\operatorname{Re}}
\newcommand{\real}{\operatorname{real}}

\newcommand{\Blog}{\mathcal{B}_{\log}}
\newcommand{\BlogP}{\Blog^{\operatorname{p}}}
\newcommand{\s}{\underline{s}}

\title[Cantor bouquet Julia sets]{Entire functions with Cantor bouquet Julia sets}

\begin{document} 

\author{Leticia Pardo-Sim\'on}
\address{\noindent Dep. de Matemàtiques i Informàtica\\ Universitat de Barcelona\\ Catalonia\\ Spain\\
		\newline  Centre de Recerca Matemàtica\\ Bellaterra\\ Catalonia\\ Spain.
	\textsc{\newline \indent 
		\href{https://orcid.org/0000-0003-4039-5556%
		}{\includegraphics[width=1em,height=1em]{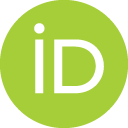} {\normalfont https://orcid.org/0000-0003-4039-5556}}%%
}}
\email{lpardosimon@ub.edu}

\author{Lasse Rempe} 
\address{\noindent Dept. of Mathematics \\ The University of Manchester \\ Manchester \\ M13 9PL \\ UK \\
	 \textsc{\newline \indent 
	   \href{https://orcid.org/0000-0001-8032-8580%
	     }{\includegraphics[width=1em,height=1em]{orcid2.png} {\normalfont https://orcid.org/0000-0001-8032-8580}}%%
	       }}
%	ORCiD: 0000-0001-8032-8580}
\email{lasse.rempe@manchester.ac.uk}
\thanks{This work was partially supported by the grant PID2023-147252NB from the Spanish State Research Agency.}

\begin{abstract}
A hyperbolic transcendental entire function with connected Fatou set is said to be of \textit{disjoint type}. It is known that 
 the Julia set of a disjoint-type function of finite order 
  is a \textit{Cantor bouquet}; in particular, it is a collection of arcs (``hairs''), each connecting a finite endpoint to infinity. 
 We show that the latter 
  property is equivalent to 
  the function being \emph{criniferous} in the sense of
  \cite{lasse_dreadlocks} (a necessary condition for having a Cantor bouquet 
  Julia set). On the other hand, we show that there is a criniferous
  disjoint-type entire function whose Julia set is \emph{not} a Cantor bouquet.
   We also provide a new characterisation of Cantor bouquet Julia sets in terms of the existence of certain absorbing sets for the set of escaping points, and use this to give a new intrinsic description of a class of entire functions previously introduced by the first author. Finally, the main known sufficient condition for Cantor bouquet Julia sets is the so-called \emph{head-start condition} of Rottenfu{\ss}er et al. Under a mild geometric assumption, we prove that this condition is also necessary. 
\end{abstract}

\maketitle

\section{Introduction}

In 1926, Fatou~\cite{Fatou} initiated the study of the iteration of transcendental entire functions $f\colon\C\to\C$. The locus of stable behaviour of such a function~-- more precisely, the set of points $z\in\C$ at which the family $\{f^n\}_{n=0}^{\infty}$ is equicontinuous with respect to the spherical metric~-- is today called the \emph{Fatou set} $F(f)$. Its complement $J(f) \defeq \C\setminus F(f)$ is known as the \emph{Julia set}.

 In~\cite[p.\ 369]{Fatou}, Fatou observes that the Julia sets of certain explicit transcendental entire functions, such as the family
\begin{equation}\label{eq_sin}
 f_a\colon z\mapsto a\sin(z), \qquad 0<a<1,
\end{equation}
contain infinitely many curves to infinity, obtained here as iterated preimages of a piece of  the imaginary axis.  In the 1980s, Devaney with a number of co-authors studied this phenomenon further, showing the existence of many more curves in the Julia set for a variety of transcendental entire functions. In particular, Devaney and Krych studied the family 
\begin{equation}\label{eq_exp}
	E_a\colon z\mapsto a\exp(z), \qquad 0<a<1/e,
\end{equation}
 and observed that $J(E_a)$ consists entirely of arcs, each of which
 connects a finite endpoint to infinity
   \cite[p. 50]{devaney_Krych}. These arcs 
%%%    lie entirely within the escaping set $I(E_a)$, 
 %%%  except possibly for their finite endpoints and 
 are called \textit{(Devaney) hairs}. 
%  It was observed already in the 1980s \cite{dgh} that they may be thought of as
%   analogues
%  of \emph{external rays} in polynomial dynamics (see \cite[???]{milnor_book} and 
%\cite{milnorrays}). Subsequently, they have indeed
% found significant applications in the
%study of the dynamics of transcendental entire functions; see for example~\cite{bifurcations}. 

Devaney and Goldberg~\cite[Proposition~3.6]{devaneygoldberg} show that these hairs are embedded in the plane in such a way that each finite endpoint is 
 accessible from the Fatou set $F(E_a)$. This was strengthened by Aarts and Oversteegen~\cite{aartsoversteegen},
 who proved that 
  the Julia set of any map $f$ of the form~\eqref{eq_sin} or~\eqref{eq_exp} is what is now called a
  \emph{Cantor bouquet}.\footnote{%
  In the case of the family~\eqref{eq_sin}, Aarts and Oversteegen prove this only for $0<a<\nu\approx 0.85$. However, all functions
  $f_a$ for $0<a<1$ are known to be topologically (and even quasiconformally) conjugate in the plane, and hence the result holds for $\nu\leq a < 1$ also.} 
  We refer to Definition~\ref{def_brush} for the definition of Cantor bouquets, and note here only that the key is to prove that $J(f)\cup \{\infty\}$ is 
    an \emph{arc-smooth fan}, meaning that there is a homeomorphism of the plane to itself that transforms each hair into a horizontal ray. Aarts and Oversteegen also show that any two Cantor bouquets are ambiently homeomorphic (i.e., related to each other
    by a self-homeomorphism of the plane), and hence their
   results provide a complete
   description of the embedding of the hairs in the plane, for these families.
  
 A function $f$ of the form~\eqref{eq_sin} or~\eqref{eq_exp} has the following properties: 
  \begin{enumerate}[(1)]
    \item $F(f)$ is connected and contains an attracting fixed point;\label{item:connectedfatou}
    \item the set $S(f)$ of \emph{singular values} is bounded and contained in the Fatou set.\label{item:singularvaluesinfatou}
  \end{enumerate}
  Here $S(f)$ is the closure of the
  set of critical and asymptotic values of $f$, or equivalently the smallest closed set
   such that $f\colon \C\setminus f^{-1}(S(f))\to \C\setminus S(f)$ is a covering map.
   (We have $S(f_a) = \{a,-a\}$ and $S(E_a)=\{0\}$.) 
     The class of transcendental entire functions
   with properties~\ref{item:connectedfatou} and~\ref{item:singularvaluesinfatou} 
   was first studied by Karpi\'nska in 2003~\cite{karpinskaaccessible}. Today
   these are called functions of \emph{disjoint type}.
    Karpi\'nska raised the question~\cite[p.~91]{karpinskaaccessible}
   under which conditions the properties of disjoint-type trigonometric and exponential functions carry over to this more general setting. In particular, one may ask
   whether the Julia set of a disjoint-type function is always a union of hairs, or a Cantor bouquet. 
   
  In \cite[Theorem~1.1]{rrrs} it is shown that this is not the case: there exists a disjoint-type function $f$ such that $J(f)$ contains no arcs. This also gives a negative
   answer to a well-known question of Eremenko from 1989~\cite{eremenko1989} and shows that the Julia sets of disjoint-type functions can be 
   far more complicated than suggested by the families~\eqref{eq_sin} and~\eqref{eq_exp}. (Further disjoint-type functions with pathological
   Julia sets were subsequently constructed in~\cite{lasse_arclike} and~\cite{pseudoarcs}.) On the other hand, if $f$ is a disjoint-type
   entire function of \emph{finite order}, then it was shown in~\cite[Theorem~C]{Baranski_Trees} and, independently,
    in~\cite[Theorem~5.10]{rrrs} that 
    every connected component of $J(f)$ is indeed an arc to infinity. 
    (The results of~\cite{rrrs} also
     apply to a larger class of disjoint-type entire functions; see the discussion of head-start conditions below.) 
    Bara\'nski proves additionally that
    the endpoints of these hairs are accessible from the Fatou set, 
    and in~\cite{lasseBrushing} it is shown that the Julia sets in question
    are indeed
    Cantor bouquets. 
    
  This suggests the following natural question: if $f$ is a disjoint-type function for which $J(f)$ is a union of hairs, does it follow that
    the endpoints of these hairs are accessible, or that $J(f)$ is a Cantor bouquet? We give a negative answer.
    \begin{thm}\label{thm_intro_crinNoCantorB}
      There is a disjoint-type entire function $f$ such that
       every connected component of $J(f)$ is an arc connecting a finite endpoint to infinity, but one of these connected components contains no point that is accessible from $F(f)$. 
      In particular, $J(f)$ is not a Cantor bouquet. 
    \end{thm}
 We remark that the existence of an entire function $f$ for which some connected component of $J(f)$ contains no accessible point follows 
  from~\cite[Theorems~2.3 and~5.10]{lasse_arclike}; however, in this example the component is not an arc but rather a much more complicated continuum. 
     
 It has long been known that Devaney hairs~-- i.e., curves in the Julia set of entire functions most of whose points converge to infinity~-- exist not only for disjoint-type functions.
   For exponential maps $E_a$ as in~\eqref{eq_exp}, but now allowing all 
   $a\in\C\setminus \{0\}$, 
   it was observed already in the 1980s \cite{dgh} that these hairs may be thought of as
   analogues
   of external rays 
   in polynomial dynamics (see \cite[\S~18]{milnor_book} and \cite{milnorrays}). For this reason, they
   are also called \emph{dynamic rays}; they have indeed
   found significant applications in the
   study of the dynamics of transcendental entire functions; see for example~\cite{bifurcations}. 

 Let us make precise what we mean when we say that the
   escaping set of a transcendental entire function $f$ consists of 
   hairs. A \emph{ray tail} 
  of $f$ is an injective curve $\gamma\colon [0,\infty)\to \C$ with the following properties: 
\begin{itemize}
  \item for each $n\geq 1$, $t \mapsto f^{n}(\gamma(t))$ is injective with $\lim_{t \rightarrow \infty} \lvert f^{n}(\gamma(t))\rvert =+\infty$, and
  \item $f^{n}(\gamma(t))\rightarrow \infty$ uniformly in $t$ as $n\rightarrow \infty$.
\end{itemize}
Following~\cite{lasse_dreadlocks}, we then say that $f$ is \emph{criniferous} if every point of the \emph{escaping set}
\[I(f)\defeq \{z\in\C : f^n(z)\to \infty \text{ as } n\to \infty\},\]
 is eventually mapped into a ray tail of $f$;  see also Definition~\ref{def_criniferous}. Criniferous functions most commonly arise in the \emph{Eremenko-Lyubich class} $\B$ of 
 transcendental entire functions with bounded set of singular values (see \cite{Dave_survey}), and we shall restrict to this class in the following. 
 
 Observe that disjoint-type functions belong to the class $\B$ by definition. Conversely, if $f\in\B$, then the function $\lambda f$ is of disjoint type 
  for sufficiently small $\lambda\in\C\setminus\{0\}$; see~\cite[p.~392]{BarKarp_codingtrees}. Moreover, it is shown in~\cite{lasseRigidity} that the functions
  $f$ and $\lambda f$ are topologically conjugate on the sets of points that remain sufficiently large under iteration. (See~\cite[Theorem~1.1]{lasseRigidity} for the precise 
  statement.) In particular, $f$ is criniferous if and only if this holds for $\lambda f$; see~\cite[Appendix~A]{lasseRigidity}. We prove that, for disjoint-type functions, the property of being criniferous is
  reflected in the structure of the entire Julia set.
  \begin{thm}[Characterisations of criniferous disjoint-type functions]\label{thm:criniferousdisjointtype}
   Let $f$ be a disjoint-type entire function. The following are equivalent. 
   \begin{enumerate}[(a)]
     \item $I(f)\cup \{\infty\}$ is path-connected;\label{item:Ipathconnected}
     \item every connected component of $J(f)$ is an arc connecting a finite endpoint to infinity;\label{item:juliaarcs}
     \item every $z\in I(f)$ can be connected to $\infty$ by a curve on which $f^n\to\infty$ uniformly;\label{item:uniformstrongeremenko}
     \item $f$ is criniferous.\label{item:fcriniferous}
  \end{enumerate}
  \end{thm}

 Here the key implication is \ref{item:juliaarcs} $\Rightarrow$ \ref{item:uniformstrongeremenko} (the other parts of the theorem 
    follow from previous work in~\cite{lasse_dreadlocks} and~\cite{lasse_arclike}; see Proposition~\ref{prop:disjoint_crin}). This implication
    is remarkable since it does not hold for individual components of the Julia set: by~\cite[Theorem~2.10]{lasse_arclike}, 
   there exists a disjoint-type entire function $f$ and a connected component $\gamma$ of its Julia 
   set such that $\gamma$ is an arc connecting a finite endpoint to infinity, $\gamma$ is entirely contained in the escaping set, but the iterates of $f$ do not converge
   \emph{uniformly} to infinity on $\gamma$. In particular, the finite endpoint of $\gamma$ is never mapped into a ray tail. So 
   Theorem~\ref{thm:criniferousdisjointtype} implies that, for such $f$, there must be a \emph{different} component of the Julia set that is not an arc. 
   
 An entire function satisfying~\ref{item:Ipathconnected} is said to have the \emph{strong Eremenko property}, in reference
  to the question from~\cite{eremenko1989} mentioned above.  It 
  satisfies the \emph{uniform Eremenko property} if $I(f)$ is a union of unbounded connected sets on each of which 
  $f^n\to\infty$ uniformly 
   (compare~\cite[Appendix~A]{lasseRigidity}). Thus Theorem~\ref{thm:criniferousdisjointtype} shows that, for disjoint-type functions, the 
   strong Eremenko property implies the uniform Eremenko property. It is interesting to ask whether the converse holds:
    \begin{question}
      Is there a non-criniferous disjoint-type entire function $f$ such that, for  every $z\in I(f)$, there is an unbounded connected
        set $A\ni z$ such that $f^n|_A\to\infty$ uniformly? 
    \end{question}
   
      By the comments before Theorem \ref{thm:criniferousdisjointtype}, the property of $f\in\B$ being criniferous can be expressed purely in terms of the Julia set of the disjoint-type function $\lambda f$.
    In circumstances such as in the dynamical study of functions $f\in\B$ with escaping singular values~\cite{mio_splitting}, in addition to the existence of hairs one also 
    requires control on how a hair is approximated on by limiting hairs, similarly to the
    structure of a Cantor bouquet. For this reason, the first author~\cite{mio_newCB}
    defined the class $\CB$ of those functions $f\in\B$ for which $g\defeq \lambda f$ is a Cantor bouquet for sufficiently small $\lambda$. (By~\cite{lasseBrushing}, all finite-order
    functions in $\B$ as well as their compositions belong to $\CB$.) One may ask whether, similarly to Theorem~\ref{thm:criniferousdisjointtype}, the property that $J(g)$ is a Cantor bouquet can also be described purely in terms of the dynamics of $g$ near infinity. 
    
  We show that this is indeed the case. Let us say that a set $A\subset J(f)$ is \emph{absorbing} if it is closed and forward-invariant, the image of any arc to infinity in $A$ is an arc to infinity,
  %\lr{what exactly do we mean by an unbounded arc? And also, is this a good definition for an absorbing set?} in $A$ is also unbounded,
   and every point of the escaping set $I(f)$ is eventually mapped into $A$. 
   \begin{thm}[Absorbing Cantor bouquets]\label{thm:cantor}
     Let $f$ be a disjoint-type entire function. The following are equivalent.
      \begin{enumerate}[(a)] 
        \item \label{item:Jcantor} $J(f)$ is a Cantor bouquet.
        \item \label{item:absorbing} For every $R>0$, $J(f)\cap \{\lvert z\rvert \geq R\}$ contains an absorbing Cantor bouquet.
        \item \label{item:abssmooth} $J(f)\cup\{\infty\}$ contains an absorbing arc-smooth fan. f
      \end{enumerate}
   \end{thm}
%%%LR: I have removed the notion of strongly absorbing here; we should now define it later in the paper. 

As mentioned above, the results of~\cite{lasseRigidity} imply that $f\in\B$ and $g = \lambda f$ have the same dynamics near infinity. Therefore Theorem~\ref{thm:cantor} leads to an intrinsic definition of the class $\CB$, without the need of passing to a disjoint-type function $g$.  

 \begin{thm}[The class $\CB$]\label{thm_CB}
Let $f\in\B$. The following are equivalent.
   \begin{enumerate}[(a)]
    \item $f\in\CB$;
    \item $J(f)$ contains an absorbing 
      set $X$ such that $X\cup\{\infty\}$ is an arc-smooth fan;\label{item:absorbingfan}
    \item  Let $E\subset I(f)$ be a compact set such that $f^n|_E\to\infty$ uniformly. Then there is $n_0\geq 0$ such that, for 
       every $z\in f^{n_0}(E)$, there is a ray tail $\gamma_z$ that connects $z$ to $\infty$. Moreover, 
       $\gamma_z\cup\{\infty\}$ depends continuously on $z$ in the Hausdorff metric on compact subsets of $\Ch$, and 
       $f^n\to\infty$ uniformly on $\bigcup_{z\in f^{n_0}(E)} \gamma_z$.\label{item:newCB}
 \end{enumerate}
 \end{thm}
 We can interpret condition~\ref{item:newCB} as stating that $f$ is criniferous, and moreover the ray tails in the definition of a criniferous function can be chosen
  in a ``uniform'' way when $z$ ranges over a set on which $f^n\to\infty$ uniformly. 

The main method known for proving that a given Julia set is a Cantor bouquet is via showing that the function $f$ satisfies a ``head-start condition''~\cite{rrrs}. We use the following version, which is slightly more general
  than the original definition from~\cite{rrrs}. 

\begin{defn}[Uniform head-start condition] \label{UHSCentire} Let $f$ be a disjoint-type entire function. 
We say that $f$ satisfies
a \emph{uniform head-start condition (with respect to $\lvert \cdot \rvert$) 
on its Julia set} if there is an
upper semicontinuous function $\phi \colon [0,\infty)\to [0,\infty)$ with $\phi(t)>t$ for all $t$ such that 
the following holds. Whenever $z$ and $w$ belong to the same component of $J(f)$,
\begin{enumerate}[(i)]
\item if $\lvert w \rvert > \phi(\lvert z\rvert)$, then
$\lvert f(w) \rvert > \phi( \lvert f(z)\rvert)$, and
\item if $z\neq w$, then there is $n \geq 0$ such that either 
$\lvert f^n(w) \rvert > \phi( \lvert f^n(z)\rvert)$ or
$\lvert f^n(z) \rvert > \phi( \lvert f^n(w)\rvert)$. 
\end{enumerate}
\end{defn}
We show that for many functions of disjoint type, having a Cantor bouquet Julia set is equivalent to satisfying a uniform head-start condition.
  Following~\cite{rrrs}, we say that a function $f$ has \emph{bounded slope} if there is a curve $\gamma:[0,\infty)\to\C\setminus\{0\}$ such that
    \begin{align*}
     \limsup_{t\to\infty} \lvert f(\gamma(t))\rvert &< \infty \quad\text{and}\\
    \limsup_{t\to\infty} \frac{\arg \gamma(t)}{\log\lvert \gamma(t)\rvert} &=\infty. \end{align*}
%\[ \limsup_{t\to\infty} \lvert f(\gamma(t))\rvert  < \infty \quad\text{and} \quad
%\limsup_{t\to\infty} \frac{\arg \gamma(t)}{\log\lvert \gamma(t)\rvert} \infty. \]

\begin{thm}\label{thm:headstart}
Let $f$ be of disjoint type. If $f$ satisfies a uniform head-start condition
on its Julia set, then $J(f)$ is a Cantor bouquet.

Conversely, if $J(f)$ is a Cantor bouquet, and additionally
$f$ has bounded slope, then $f$ satisfies
a uniform head-start condition on its Julia set. 
\end{thm}
\begin{remark}
The first conclusion is \cite[Corollary 6.3]{lasseBrushing}\footnote{%
The result in \cite{lasseBrushing} uses the definition of head-start condition from~\cite{rrrs}, which is slightly more restrictive. However,
  the same proof goes through with our definition; compare Proposition~\ref{prop_UHSC_CB}.}
 and it is included here for completeness. In Section \ref{sec_CB_then_UHSC}, we prove a more general version of the second for a notion of \emph{uniformly anguine} tracts, see Definition \ref{defn_anguine}.
\end{remark} 

We do not currently know whether the condition of bounded slope can be removed. 

\begin{question}[Necessity of head-start conditions]
 Is there a disjoint-type entire function such that $J(f)$ is a Cantor bouquet, but $f$ does not satisfy a uniform head-start condition on $J(f)$?
\end{question}

\subsection*{Structure of the article and comments on the proofs} 
In Section \ref{sec_prelim}, we review the logarithmic change of variable, which is used throughout the remainder of the paper. We also review results and definitions
  for functions expressed in logarithmic coordinates, including the definition of criniferous functions and head-start conditions.  Theorem \ref{thm:criniferousdisjointtype} is proved in Section \ref{sec:crin_disj}. 
   Section \ref{sec_dendroids} is of a topological nature and includes the definition and fundamental properties of \textit{Cantor bouquets}. We 
    provide some characterisations of when the Julia set is a Cantor bouquet, and introduce the notion of
    a ``bad pair'', which prevents the Julia set of a disjoint-type entire function from being a Cantor bouquet.
   
With these preparations, we are ready to prove our main results. Section \ref{sec_abs} is dedicated to the proof of Theorem \ref{thm:cantor}. The key to the proof
 of this theorem is to  show the following. Suppose that $J(f)$ is not a Cantor bouquet; this means that there are some hairs that ``fold'' over themselves as they approach some
 limiting hair; see Proposition~\ref{prop_properties_bad}. Our goal is to show that this behaviour is then ubiquitous in a certain subset of the Julia set, at all scales.
 (See Theorem~\ref{thm:bad}.) The proof of this fact turns out to be rather involved, and requires a detailed understanding of the properties of disjoint-type functions. 
  We then deduce Theorem~\ref{thm_CB} in Section~\ref{sec_CB}.
  
 In Section \ref{sec_noCb}, we prove Theorem \ref{thm_intro_crinNoCantorB}. We first provide an example of a criniferous entire function whose Julia set is not a Cantor bouquet,
  and then indicate how to modify this construction to also ensure that some endpoint is not accessible. Constructing a disjoint-type function whose Julia set is not a Cantor bouquet 
  is achieved by introducing a sequence of tracts with sufficiently large ``hooks'' whose preimages will accumulate on an invariant ray in the Julia set in 
  a manner that is not compatible with the Julia set being a Cantor bouquet. The main
   difficulty lies in ensuring that this is nonetheless criniferous; i.e., that the 
  hooks themselves do not lead to the existence of non-arc components of the Julia set, as was the case in Theorem~\ref{thm:criniferousdisjointtype}. To this end, we use a 
  non-uniform version of the head-start condition, which we verify using careful choice of the tracts and estimates from geometric function theory that allow us to control the
  function we construct. Finally, in Section \ref{sec_CB_then_UHSC}, we prove (a generalisation of) Theorem \ref{thm:headstart}. 

\subsection*{Basic notation}
Recall from earlier in this section that the Julia and escaping set of an entire function $f$ are denoted by $J(f)$ and $I(f)$ respectively. The set of singular values of $f$ is $S(f)$. We denote the complex plane by $\C$, the Riemann sphere by $\widehat{\C}$, and the unit disc by $\D$. We will indicate the closure of a domain $U$ by $\overline{U}$, that must be understood to be taken in $\C$. For each $R\geq 0$, we let $\H_R\defeq \{z\in \C \colon \Rea z>R\}$ and abbreviate $\H\defeq \H_0$.
The Euclidean disc of radius $r>0$ with centre $z_0\in\C$ is denoted $D_r(z)$. Euclidean distance and diameter are respectively denoted by dist and diam. If $A,B\subset\C$ are compact, $\distH(A,B)$ denotes the Hausdorff distance between $A$ and~$B$:
$$\distH(A,B)\defeq \max \left\{\,\sup _{a\in A} \dist(a,B),\ \sup _{b\in B}\dist(A,b)\,\right\}. $$  
\subsection*{Acknowledgements}
We thank David Mart\'i-Pete and James Waterman for helpful discussions and the 
	referee for thoughtful and valuable feedback that helped us to improve
	the article. This work was partially conducted while the  authors worked at the University of Manchester and at the University of Liverpool, respectively. The second author gratefully acknowledges funds provided by the University of Liverpool for esearch visits to Manchester and Barcelona in support of this research. The first author is a Serra Húnter fellow.
\section{Preliminaries}\label{sec_prelim}

\subsection*{The logarithmic change of variable}  
Recall from the introduction that the \emph{Eremenko-Lyubich class} $\B$ consists of those transcendental entire functions whose set $S(f)$ of singular values is bounded. 
 Let us review the \emph{logarithmic change of coordinates} for $f\in\B$, a commonly used tool for studying the Eremenko-Lyubich class
 first introduced to holomorphic dynamics in \cite[\S2]{eremenkoclassB}; for more details, see \cite[\S5]{Dave_survey}, \cite[\S2]{rrrs}~or~\cite[~\S3]{lasse_arclike}. 

Let $f\in\B$ and fix a bounded Jordan domain $D\Supset S(f)\cup \{0,f(0)\}$. Let $H \defeq \exp^{-1}(\C\setminus \overline{D})$; note that $H$ contains a right half-plane. 
Let $\T_F\defeq \exp^{-1}(f^{-1}(\C\setminus \overline{D}))$. 
Since $f\colon \T_F\to \C\setminus\overline{D}$ is a covering map, and $f$ is transcendental, 
each connected component $T$ of $\T_F$ is a simply connected domain whose boundary is homeomorphic to $\R$. Moreover, by the action of the exponential map, both ``ends'' of the boundaries of $T$ have real parts converging to $+\infty$, and both $\T_F$ and $H$ are invariant under translation by $2\pi i$. Consequently, we can lift $f$ to a map $F\colon \T_F\to \H_{\log L}$ satisfying 
\begin{equation}\label{eq_commute_log}
\exp\circ F = f\circ\exp
\end{equation}
and such that $F$ is $2\pi i$-periodic. We call $F$ a \emph{logarithmic transform} of $f$. Moreover, we call each connected component of $\T_F$ a \emph{(logarithmic) tract} of $F$. Note that $F\colon T\to H$ is a conformal isomorphism for every tract $T$. 

For most part of the article, we will study the logarithmic transform of $f$. In fact, the map $F$ needs not have arisen from an entire function $f$, and we can work in this
 greater generality from now on.
\begin{defn}[The classes $\Blog$ and $\BlogP$ \cite{lasseRigidity,lasse_arclike}]\label{defn:Blog}
The class $\Blog$ consists of all holomorphic functions 
\[ F\colon \T\to H, \]
where $F$, $\T$ and $H$ have the following properties:
\begin{enumerate}[(a)]
\item $H$ is a $2\pi i$-periodic unbounded Jordan domain that contains
a right half-plane.   \label{item:H}
\item $\T\neq\emptyset$ 
is $2\pi i$-periodic and $\re z$ is bounded from below in
$\T$, but unbounded from above.
\item Every connected component $T$ of $\T$ is an 
unbounded Jordan domain that is disjoint
from all its $2\pi i\Z$-translates. 
For each such $T$, the restriction
$F\colon T\to H$ is a conformal 
isomorphism whose continuous extension to the closure
of $T$ in $\widehat{\C}$, which we also denote by $F$, satisfies $F(\infty)=\infty$, and whose
inverse we denote by $F_T^{-1} \defeq  (F|_T)^{-1}$.\label{item:tracts}%
\item The tracts of $F$ have pairwise disjoint closures and
accumulate only at $\infty$; i.e.,
if $z_n\in\T$ is a sequence of points all belonging to different
tracts of $F$, then $z_n\to\infty$.\label{item:accumulatingatinfty}%
\end{enumerate}

If, furthermore, $F$ is $2\pi i$-periodic, then we say that $F$ belongs to the class
$\BlogP$. We are mainly interested in the class $\BlogP$~-- which contains
all logarithmic transforms of functions $f\in\B$ as defined above~-- but some of
our results hold more generally for $F\in\Blog$. The \emph{Julia set} and \emph{escaping set} of $F\in \Blog$ are respectively defined by
\begin{equation}
\begin{split}
J(F) \defeq \ &\{z\in \overline{\T}\colon \colon F^n(z)\in \overline{\T}  \text{ for all $n\geq 0$} \} \quad \text{ and } \\
I(F) \defeq\  &\{z\in J(F)\colon \re F^n(z)\to\infty \text{ as $n\to\infty$}\}. 
\end{split} 
\end{equation}
\end{defn}

For $Q>0$, we also define
\begin{equation}\label{eq_JRF}
J_Q(F) \defeq \{z\in J(F) \colon \Rea F^n(z)\geq Q \text{ for all } n\geq 1\}.
\end{equation}
Note that if $F$ is a logarithmic transform of $f$, then $\exp(J_Q(F))=J_{e^Q}(f)$, where for any $R>0$, we define
\begin{equation*}
J_R(f)\defeq  \{z\in J(f) \colon \vert f^n(z)\vert >R \text{ for all } n\geq 1\}.
\end{equation*}

If $\cl{\T}\subset H$, then we say that $F$ is of \emph{disjoint type}, and we have
\begin{equation}\label{eq_JF_disj}
J(F)=\bigcap_{n\geq 0} F^{-n}(\overline{\T}).
\end{equation}
In particular, $J(F)$ is a union of uncountably many connected components, each of which is closed and unbounded, see \cite[Corollary 3.1]{lasse_arclike}. 
%%If $C$ is such a component, we call the continuum $\hat{C}\defeq C\cup \{\infty\}$ a \textit{Julia continuum} of $F$. 

The following well-known fact (see e.g.\ \cite[Proposition 3.2]{lasse_arclike}) 
shows that disjoint-type entire functions have disjoint-type logarithmic transforms (possibly after applying an affine conjugacy that ensures that
$0\in F(f)$). 

\begin{prop}\label{prop_disjoint_type}
If $f\in\B$ is of disjoint type and $0\in F(f)$, then there exists a bounded Jordan domain $D \supset S(f)\cup\{0\}$ such that $f(\overline{D}) \subset D$. In particular,
if $F$ is a logarithmic transform of $f$ as above, defined using $D$, then $F$ is of disjoint type and $\exp(J(F))=J(f)$. 
\end{prop}

The following result shows that, in a certain sense, the converse also holds: For every disjoint-type $G\in\BlogP$, there is a 
disjoint-type entire function $f\in\B$ with a logarithmic transform that has the same dynamics as $G$. 
\begin{thm}[Realisation of disjoint-type models] \label{thm_Bishop} Let $G\in \BlogP$ be of disjoint type and let $g$ be defined by $g(\exp(z)) = \exp(G(z))$. Then there is a disjoint-type function $f\in \B$ such that $f\vert_{J(f)}$ is topologically (and, in fact, quasiconformally) conjugate to $g \vert_{\exp(J(G))}$.
\end{thm}
\begin{proof}
The statement is \cite[Theorem 1.2]{bishopClassB}, which follows from Theorem 1.1 in the same paper together with \cite[Theorem~3.1]{lasseRigidity}.
\end{proof}

In view of Theorem \ref{thm_Bishop}, to construct disjoint-type entire functions,  as e.g.\ in Theorem~\ref{thm_intro_crinNoCantorB}, we need only construct
 a disjoint-type function in $\BlogP$ with analogous properties.

Given $F\in \Blog$, we can assign symbolic dynamics to points whose orbit stays in $H$ using the partition that logarithmic tracts provide: 
\begin{defn}[External addresses] Let $F\in \Blog$. An \emph{(infinite) external address} is a sequence $\s = T_0T_1T_2\ldots$ of logarithmic tracts of $F$. If $\s$ is such an external address, then we denote
\[J_{\s} \defeq \{ z \in J(F)\colon F^n(z)\in \overline{T_n} \text{ for all $n\geq 0$}\}. \]
If $z\in J_{\s}$ for some $\s$, then we say that $z$ has external address $\s$. If $J_{\s}\neq \emptyset$, then we say that $\s$ is \textit{admissible}. %An address $\s$ is \textit{bounded} if it contains only finitely many different tracts, and \emph{periodic} if there is $p \geq 1$ such that $T_j = T_{j+p}$ for all $j \geq 0$. 
For each $n\geq 0$, we denote
\[F_{\s}^n\defeq F\vert_{T_n}\circ F\vert_{T_{n-1}}\circ \cdots \circ F\vert_{T_0} \quad \text{ and } \quad F_{\s}^{-n}\defeq \left(F_{\s}^n\right)^{-1}.\]
\end{defn}
When $F$ is of disjoint type and $\s$ is admissible, $J_{\s}\cup \{\infty\}$ is called a \textit{Julia continuum} of $F$. The following result gives conditions for
 a Julia continuum to be an arc. It is a consequence of work of the second author in~\cite{lasse_arclike}, but is not explicitly stated there. (Compare also
 the discussion in~\cite[Remark 4.15]{lasse_dreadlocks}.) 

\begin{prop}\label{prop:disjoint_crin} 
  Let $F\in \Blog$ be of disjoint type, let $J_{\s} = J_{\s}(F)$ be a component of $J(F)$, and set $I_{\s} \defeq J_{\s}\cap I(F)$. The following are equivalent.
  \begin{enumerate}[(a)]
     \item $I_{\s}\cup \{\infty\}$ is path-connected;\label{item:BlogIpathconnected}
     \item $J_{\s}\cup \{\infty\}$ is path-connected;\label{item:BlogJpathconnected}
     \item $J_{\s}$ is an arc connecting a finite endpoint to infinity.\label{item:Blogjuliaarcs}
  \end{enumerate}
   Moreover, suppose that~\ref{item:Blogjuliaarcs} (and hence all three conditions) holds. Then for every $z\in J_{\s}$ with the possible exception of its
   finite endpoint, the iterates $F^n$ tend to infinity uniformly on the sub-arc of $J_{\s}$ connecting $z$ to $\infty$. 
 \end{prop}
\begin{proof}
 Clearly~\ref{item:Blogjuliaarcs} $\Rightarrow$~\ref{item:BlogJpathconnected}.
   To prove the remaining implications, we recall several facts:
     \begin{enumerate}[(1)]
       \item $J_{\s}$ does not separate the plane, since $J_{\s}\cup\{\infty\}$ is a full
           compact set. Moreover,
           $J_{\s}$ has empty interior; see e.g.~\cite[Lemma 2.3]{lasseRigidity}. 
           (We caution that, in \cite{lasseRigidity}, the result is stated only for what
           we call $\BlogP$. However, the proof applies equally to the class
           $\Blog$. Both facts mentioned 
           also follow immediately from~\cite[Theorem~5.6]{lasse_arclike}, since
           planar continua of span zero are one-dimensional and non-separating.)\label{item:Jnonseparating} 
       \item If $A\subsetneq J_{\s}$ is unbounded, closed and connected, then $F^n|_A\to\infty$ uniformly by~\cite[Proposition~6.4]{lasse_arclike}.\label{item:composantinI}
       \item The set of $A$ as in~\ref{item:composantinI} is linearly ordered by inclusion by \cite[Theorem~5.7]{lasse_arclike}.\label{item:linearlyordered}
       \item $J_{\s}\setminus I_{\s}$ is totally disconnected by~\cite[Proposition~5.9]{lasse_arclike}.\label{item:hausdorffzero}
     \end{enumerate}
     The final claim of the proposition follows directly from~\ref{item:composantinI}. 
In the following, let
      $A_{\s}$ be the union of all unbounded closed proper subsets of $J_{\s}$. 

  Now suppose that $J_{\s}\cup\{\infty\}$ is path-connected. Let 
      $z\in I_{\s}$ and let $\gamma\subset J_{\s}$ be an arc connecting $z$ to infinity. Then
      $\gamma\setminus\{z\}\subset A_{\s}$, and thus $\gamma\subset I_{\s}$
      by~\ref{item:composantinI}. So $I_{\s}\cup\{\infty\}$ is path-connected, and we have shown that \ref{item:BlogJpathconnected} $\Rightarrow$~\ref{item:BlogIpathconnected}. 
   
   Finally, assume that $I_{\s}\cup\{\infty\}$ is path-connected; we must prove~\ref{item:Blogjuliaarcs}. If 
     $J_{\s}\setminus A_{\s}$ contains an escaping point $z$, let $\gamma\subset J_{\s}$ be an arc connecting $z$ to infinity. 
     By definition of $A_{\s}$, we must have $J_{\s} = \gamma$, and $J_{\s}$ is an arc as desired. 
     
   Otherwise, we have $A_{\s} = I_{\s}$. The proof in this case
    is analogous to that of \cite[Theorem 2.12]{mio_signed_addr}, but we provide the 
    details for completeness. 
     By~\ref{item:BlogIpathconnected},~\ref{item:composantinI} and~\ref{item:linearlyordered}, $I_{\s}$ is an increasing union of arcs to infinity. 
     Thus $I_{\s}$ is the image of a continuous injective function $\gamma\colon (0,\infty)\to I_{\s}$ with $\gamma(t)\to\infty$ as $t\to\infty$. 
    Let $\Lambda$ be the accumulation set of $\gamma(t)$ as $t\to 0$, in $\C$. Then $\Lambda$ is closed and its closure in the sphere is
     connected; by~\ref{item:linearlyordered}, it follows that $\Lambda$ itself is connected. Also, $\Lambda$ is non-empty (otherwise $\gamma((0,1])$ would be a closed and unbounded 
     set, contradicting~\ref{item:linearlyordered}). 
     
     We claim that $\Lambda\cap I_{\s}=\emptyset$. 
    Otherwise, $\gamma(t)\in \Lambda$ for some $t>0$. For $\eps>0$, the set
\[        \Lambda \cup \gamma\left((0,\eps]\cup [t,\infty)\right)\]
      is closed and connected. So by~\ref{item:linearlyordered}, $\gamma((0,t])\subset \Lambda$. By~\ref{item:Jnonseparating} and a theorem
      of Curry~\cite[Theorem~8]{curry_continua}, $\Lambda$ is an indecomposable continuum. All points of $I_{\s}$ belong to the same composant of $\Lambda$,
         and hence by~\cite[Theorem~11.15]{nadler_continuum}, $\Lambda\setminus I_{\s} \subset J_{\s}\setminus I_{\s}$ contains continua. This contradicts~\ref{item:hausdorffzero}. 

    We conclude that $\Lambda$ is a connected subset of $J_{\s}\setminus I_{\s}$, and therefore must consist of a 
      single point by~\ref{item:hausdorffzero}. In other words, $\gamma(t)$ has a limit for $t\to 0$, and the proof is complete. 
%% This follows from \cite[Corollary 6.6]{lasse_arclike} and \cite[Remark 4.15]{lasse_dreadlocks}; compare to \cite[Theorem 2.12]{mio_signed_addr}.
\end{proof}

\subsection*{Criniferous functions in $\Blog$}
Recall that we defined \emph{ray tails} of entire functions in the introduction. For maps in $\Blog$, the definition simplifies.
\begin{defn}[Ray tails and dynamic rays]\label{def_ray}
Let $F\in \Blog$. A \emph{ray tail} of $F$ is an injective curve $\gamma \colon[t_0,\infty)\rightarrow I(F)$, with $t_0>0$, such that
   $F^{n} \circ \gamma \to \infty$ uniformly as $n\rightarrow \infty$.

 If $F$ is of disjoint type, then a \emph{dynamic ray} of $F$ is a maximal injective curve $\gamma \colon (0,\infty)\rightarrow I(F)$ such that the restriction $\gamma_{|[t,\infty)}$ is a ray tail for all $t > 0$. We say that such $\gamma$ \emph{lands} at $z\in \C$ if $\lim_{t \rightarrow 0^+} \gamma(t)=z$, and we call $z$ the \emph{endpoint} of $\gamma$.
\end{defn}

%In \cite{lasse_dreadlocks}, Benini and the second author introduce the following terminology: an entire function is defined to be \textit{criniferous} if some iterate of every point in their escaping set belongs to a ray tail. We adapt this definition to functions in $\Blog$.
Similarly, we adapt the definition of criniferous functions to the class $\Blog$. 
\begin{defn}[Criniferous functions]\label{def_criniferous} A function $F\in \Blog$ is \textit{criniferous} if every $z\in I(F)$ is eventually mapped to a ray tail. That is, for every $z \in I(F)$, there exists $N \defeq N(z)\geq 0$ such that $F^n(z)$ belongs to a ray tail for all $n\geq N$.
\end{defn}

\subsection*{Hyperbolic geometry}
We frequently use the \textit{standard estimate} on the hyperbolic metric in a simply-connected domain $U \subsetneq \C $, \cite[Theorems 8.2 and 8.6]{beardon_minda}:  for $z\in U$,
\begin{equation}\label{eq_standard_est}
\frac{1}{2\dist(z, \partial U)}\leq \rho_U(z)\leq \frac{2}{\dist(z, \partial U)},
\end{equation}
where $\rho_U \colon U \to (0, \infty)$ denotes the density of the hyperbolic metric in $U$.

\begin{lem}\label{lem_distpreim}
There exists a universal constant $D$ with the following property. 
   Let $F\colon \T\to H$ be a function in  $\BlogP$ with $H=\H$, let $T$ be a tract of $F$ and let $z\in T$.
   Then for every $w\in \H$ with $1 \leq \Rea w < \Rea F(z)$, there is $m\in \Z$ such that
$$\dist(F_T^{-1}(w + 2\pi i m),z) \leq D.$$
\end{lem}
\begin{proof}
Connect $z$ and $\partial T$ by a vertical line segment $\gamma$. 
Then $F(\gamma)$ connects $F(z)$ to $\partial \H$. In particular, for any $w$ as in the statement, $F(\gamma)$ must intersect the vertical line at $\Rea w$. Recall that the hyperbolic distance in $\H$ between two points lying in the same vertical line is bounded above by the quotient of their Euclidean distance by their real part. 
Since we assumed that $\Rea w \geq 1$, it follows that there is $m\in \Z$ such that for $w_m = 2\pi i m + w$, $\dist_\H(w_m, F(\gamma))<\pi$. The 
restriction $F\colon T\rightarrow \H$ is a conformal isomorphism. So by Pick's theorem~\cite[Theorem 6.4]{beardon_minda}, $\dist_T(F_T^{-1}(w_m),\gamma)<\pi$. Then, by \eqref{eq_standard_est}, the Euclidean distance between $F_T^{-1}(w_m)$ and $\gamma$ is at most $2\pi^2$, and 
the length of $\gamma$ (which contains $z$) is at most $\pi$. So the statement follows taking $D\defeq 2\pi^2+\pi$.
\end{proof}

\section{Characterisation of criniferous disjoint-type functions}\label{sec:crin_disj}

To prove Theorem \ref{thm:criniferousdisjointtype}, we require the following result concerning plane topology. 
Suppose that $\alpha\subset\Ch$ is a curve, parameterised 
   as $\alpha\colon I\to\Ch$, where 
  $I$ is an interval. If $J\subset I$ is connected and non-empty, then
  we refer to $\alpha_0\defeq \alpha(J)$ 
  as a \emph{sub-piece} of $\alpha$. Observe that, when 
  $\alpha$ is not an arc, the notion of a sub-piece depends, strictly speaking,
  not only on the point-set $\alpha(I)$, but also on its parameterisation. 
  In all cases we consider, $\alpha$ is either an arc or comes implicitly with a
  parameterisation, so we our notation will suppress this subtlety. Observe that
  we allow the case where the sub-piece is degenerate; i.e.\ where $J$ and
  hence $\alpha_0$ consist of a single point.

\begin{lem}\label{lem:tracts}
 Suppose that $V$ is an unbounded Jordan domain, with real parts bounded from below
 but unbounded from above, and such that $\overline{V}$ is disjoint from its
  $2\pi i \Z$-translates. 
  
  Let $K\subset V$ be a compact connected set, and let $\alpha\subset V$ be a curve
  that connects a left-most point $z_0$ of $\partial V$ to $\infty$. Then there is a compact sub-piece
  $\alpha_0$ of $\alpha$ such that, for every $z\in K$, $\alpha_0$ intersects
  the vertical segment $I_z \defeq z + [-2\pi i,2\pi i]$, and conversely,
   for every $z\in \alpha_0$, $I_z$ intersects $K$. 
\end{lem}
\begin{remark}
 The claim is trivial if the imaginary parts of any two points in $V$ that belong to the
  same vertical line differ by less than $2\pi$. We note that there exist
  domains $V$ that do not have this property and yet
  satisfy the hypotheses of the lemma; compare~\cite[Figure~1]{lasse_questionErem}.
\end{remark}
\begin{proof}
  We use the following fact~\cite[Corollary~5.4]{lasse_arclike}: There is an odd number $n$
    such that every curve connecting $z_0$ to $\infty$ in $V$ intersects $I_z$ at least $n$ times. In particular,
  $\alpha\cap I_z\neq \emptyset$. 

 Now let $a$ and $b$ be a left-most and a right-most point of $K$, respectively.
   If $\re a = \re b$, then $K$ is a vertical line segment (of length less than $2\pi$),
    and we may choose $\alpha_0$ to be a degenerate sub-piece
    consisting of a single
    point of $\alpha\cap I_z$, where $z\in K$. 

   Now suppose $\re a < \re b$. Consider the set 
    \[ A \defeq I_a \cup I_b \cup (K+2\pi i)\cup (K-2\pi i). \]
  Clearly, if $z$ belongs to a bounded connected component of $\C\setminus A$, then
    $I_z\cap A\neq \emptyset$. We claim that the points of $\C\setminus A$ that
    are close and to the right of $I_a$ belong to such a bounded connected component.
    More precisely, let $z\in V\cap I_a$ and let $\eps<\re b - \re a$ be small enough
    such that 
    \[ H_{\eps}(z) \defeq \{z + \zeta\colon \lvert \zeta\rvert<\eps\text{ and }
       \re \zeta > 0 \}\subset V\setminus I_b. \] Then we claim that
       $A$ separates $H_{\eps}(z)$ from $a-1$, and hence from
    infinity. Indeed,
    consider the compact connected subset of the Riemann sphere given by 
      \[ \tilde{A} \defeq I_b \cup \{\infty\} \cup \bigcup_{\sigma\in \{1,-1\}} ( K + \sigma 2\pi i ) \cup    
         (a + \sigma 2\pi i + (-\infty,a] ). \]
    Clearly $\tilde{A}$ does not separate any $w\in H_{\eps}(z)$ from $a-1$. On the other
     hand, $\tilde{A}\cup A$ clearly does separate $w$ from $a-1$, and 
     $\tilde{A}\cap A = I_b \cup \bigcup_{\sigma \in \{+,-\}} (K+\sigma 2\pi i)$ is
     connected. Hence, by Janiszewski's Theorem \cite[p. 110]{Newman},
     $A$ must also separate these points. The same argument applies to
     points just to the left of $I_b$. 

  It follows that, each time $\alpha$ crosses $I_a$ or $I_b$, it moves from
   a bounded component of $\C\setminus A$ to the unbounded 
   component or vice versa. Since $\alpha$ begins and ends in the unbounded component and 
   crosses both $I_a$ and $I_b$ an odd number of times, it follows that there is
   a sub-arc $\alpha_0$ that connects $I_a$ and $I_b$ and is contained (apart from its endpoints) in a bounded complementary component of $A$. 
   
   It remains to show that $\alpha_0$ intersects $I_z$ for every $z\in A$. To do so,
    we use the separation number $\sep_V(c,d;z)$, where $c$ and $d$ are the endpoints
    of $\alpha_0$, and $z\in A$. Here $\sep_V(c,d;z)$ denotes the smallest
    number of times that a curve connecting $c$ and $d$ in $V$ intersects the segment
    $I_z$. By~\cite[Proposition~5.3]{lasse_arclike}, the parity of
    $\sep_V(c,d;z)$ varies continuously for $z\notin I_c\cup I_d$, and it
    equals $1$ for $z\in A$ sufficiently close to either $I_c$ or $I_d$. Indeed,
    since $\alpha_0$ does not contain any points of real part less than
    $\re a = \re c$, we have $\sep_V(c,d;z)=0$ for $\re z < \re c$, and 
    by~\cite[Proposition~5.3]{lasse_arclike}, the separation number increases by $1$
    as $z$ crosses $I_c$. Since $A$ is connected, 
    it follows that $\sep_V(c,d;z)$ is odd for all $z\in A$. So $\alpha_0$ must
    intersect $I_z$, as claimed. 
\end{proof}

To show that disjoint-type functions whose Julia set is a collection of arcs are criniferous, we prove the following stronger result.

\begin{thm}\label{thm_arcs_infinity}Suppose that $F\in \BlogP$ is of disjoint type, and every connected component of $J(F)$ is an arc to infinity. If $z_n$ is a sequence in $J(F)$ with $z_n\to\infty$, and $\gamma_n$ is the arc  in $J(F)$ connecting $z_n$ to infinity, then $\gamma_n\to\infty$ uniformly. 
\end{thm}

\begin{proof}
For convenience, we may assume without loss of generality that $F$ is \textit{normalized}, that is, $F\colon \T\to \H$ such that 
\begin{equation}\label{eq_normalized2}
	\vert F'(z)\vert>2 \quad  \text{ for all }z\in \T.
\end{equation}
Otherwise, we can pre-compose $F$ with a translation and restrict to a smaller domain. The resulting map is affine equivalent to $F$, and so they are conjugate on their Julia sets; see \cite[\S 2]{lasseRigidity} for details. 
 
We will call the components of $J(F)$ \textit{hairs}. Assume by contradiction that there exists $R>0$ and  a sequence of points $z_n \to \infty$ in $J(F)$ such that if $\alpha^n$ is the arc in $J(F)$ connecting $z_n$ to infinity, then $\alpha^n\not \subset \H_R$ for all $n>0$. In particular, for each $Q>R$ and all $n$ large enough, there exist at least two disjoint arcs in $\alpha^n$ connecting $\partial \H_R$ to $\partial \H_Q$.

\begin{claim}
 Let $\gamma_1,\dots,\gamma_m$ be pairwise disjoint compact sub-arcs of
 some hair $\gamma$ 
 of $F$, and let $\eps>0$. Then there exists a hair $\beta$ of $F$ and 
 $2m$ pairwise disjoint compact
 sub-arcs $\beta_1^1,\beta_1^2,\beta_2^1,\beta_2^2,\dots,\beta_m^1,\beta_m^2$
 such that the Hausdorff distance $\distH (\beta_k^j, \gamma_k)$ is
 less than $\eps$, for $k=1,\dots,m$ and $j=1,2$. 
\end{claim}
\begin{subproof}
  We may find an unbounded Jordan domain $V\supset \gamma$, 
  disjoint from its $2\pi i\Z$-translates,
  such that the closure of every tract of $F$ is contained in some $2\pi i\Z$-translate
  of $V$. Let $C>0$ be sufficiently large such that every point of $V$ at real part
  at most $R$ can be connected to a left-most point of $V$ by a curve of diameter
  at most $C$. 
  Let $\delta$ be smaller than $\eps$ and  the smallest Euclidean 
   distance between two of the arcs $\gamma_k$, and let $n$ be so large that $(C+2\pi)/2^n<\delta/2$. 
  
  Set $\tilde{\gamma}_k\defeq F^n(\gamma_k)$. Let $Q>R$ be so large that any point of $V$
  at
   real part $Q$ can be connected to infinity within $V$ without passing
   within distance $2\pi$ of any $\tilde{\gamma}_k$. 
   
   Recall that by assumption,
   there exists a hair of $F$ that contains two disjoint sub-arcs $\alpha^1$, $\alpha^2$,
   each connecting a point at real part $R$ to a point at real part $Q$. 
   Recall that $J(F)$ is periodic of period $2\pi i$, since $F$ is $2\pi i$-periodic. 
   Since $J(F)$ is 
   contained in the union of tracts of $F$, we may assume that the hair in question
   is contained in $V$, as one of its $2\pi i\Z$-translates must be.  
   We may extend each $\alpha^j$ to a curve $\tilde{\alpha}^j$ that connects
   a left-most point of $V$ to $\infty$, by adding to $\alpha^j$ a 
   curve of diameter at most $C$
   and another curve that does not pass within  distance $2\pi$ of any $\tilde{\gamma}_k$.
 
  By Lemma~\ref{lem:tracts}, for each $k$ each $\tilde{\alpha}^j$ contains a sub-piece
  $\tilde{\alpha}^j_k$ such that
   \[ \distH( \tilde{\alpha}^j_k, \tilde{\gamma}_k)\leq 2\pi.\]
    Moreover, set $\alpha^j_k\defeq \tilde{\alpha}^j_k \cap \alpha^j$ if this
     intersection is non-empty; otherwise let $\alpha^j_k$ consist of just the endpoint of
     $\alpha^j$ at real part $R$. Then 
     $\alpha^j_k$ is a (possibly degenerate) sub-piece of $\alpha^j$ and
     $\distH( \alpha^j_k, \tilde{\alpha}^j_k)\leq C$. Hence
     $\distH( \alpha^j_k, \tilde{\gamma}_k)\leq C+2\pi$.
 
Let $\beta_k^j$ be the image of $\alpha^j_k$ under the branch of
 $F^{-n}$ that maps each $\tilde{\gamma}_k$ to $\gamma_k$. 
  By the expansion assumption on $F$, \eqref{eq_normalized2},
    \[ \distH( \beta_k^j , \gamma_k) \leq (C+2\pi )/2^n < \delta/2.\] In particular, $\beta_k^j$ and $\beta_{k'}^{j'}$ are disjoint for $k\neq k'$, and $\beta_k^1$ and $\beta_k^2$ are 
  disjoint by choice of $\alpha^1$ and $\alpha^2$. This concludes the proof. 
\end{subproof}

Now let $\gamma_0$ be an arc in $J(F)$, and set $\Gamma_0 \defeq \{ \gamma_0\}$ and $\eps_0 \defeq \diam(\gamma_0)/4$. We now inductively define sets $\Gamma_{k+1}$
 of arcs 
  in $J(F)$ and numbers $\eps_{k+1}$ as follows. If $\Gamma_k$ and $\eps_k$
  have been defined, apply the claim to the curves in 
  $\Gamma_k$, with $\eps=\eps_k$. Then choose $\eps_{k+1}<\eps_k/2$ so that
  the Euclidean distance between any two arcs in $\Gamma_{k+1}$ is greater than
  $4\eps_{k+1}$. 
  
  Then, for all $n\geq 0$:
\begin{enumerate}
	\item \label{item:Gamma_n} All arcs in $\Gamma_n$ belong to the same hair of $F$;
	\item\label{item:two_arcs}  For 
	    every arc in $\gamma_{n}\in \Gamma_{n}$, there are two arcs 
	      $\gamma\in\Gamma_{n+1}$ such that the Hausdorff distance
	      between $\gamma_{n}$ and $\gamma$ is at most $\eps_{n+1}$. 
\end{enumerate}

Consider a sequence $(\gamma_n)_{n=0}^{\infty}$, 
where $\gamma_n\in \Gamma_n$, and $\gamma_{n+1}$ 
is one of the two arcs from \eqref{item:two_arcs} for $\gamma_n$. Then
for $m> n$, 
\[ \distH(\gamma_m,\gamma_n) \leq
  \sum_{j=n+1}^{m} \eps_j < 
\sum_{j=n+1}^{\infty} \eps_j < 2\eps_{n+1}. \]
In particular, the sequence $(\gamma_n)_{n=0}^{\infty}$ has an accumulation point
$L$ with respect to the Hausdorff metric, which also satisfies
\[ \distH(L,\gamma_n) < 2\eps_{n+1} < \eps_0 \]
for all $n\geq 0$. In particular, $\diam(L)\geq \diam(\gamma_0)/2>0$. Moreover,
by choice of $\eps_m$, the distance of $L$ to any element of
$\Gamma_n\setminus\{\gamma_n\}$ is greater than 
$2\eps_{n+1}$. Hence, if $\tilde{L}$ is a limit arising from a different sequence, then $L$ 
and $\tilde{L}$ are disjoint. 

Clearly, by \eqref{item:two_arcs} there exist uncountably many eligible
sequences $(\gamma_n)_{n=0}^{\infty}$, and hence uncountably many, 
pairwise disjoint, limit sets $L$, all having diameter at least $\diam(\gamma_0)/2$. By \eqref{item:Gamma_n}, these all must belong to the same hair, but this is impossible, as an arc cannot contain uncountably many pairwise disjoint proper sub-arcs.  %(Alternatively: an arc cannot contain even countably infinitely many sub-arcs whose diameter is bounded from below. 
We have thus obtained the desired contradiction.
\end{proof}

We can now prove Theorem~\ref{thm:criniferousdisjointtype}, restated for the class $\BlogP$ as follows.
\begin{thm}[Characterisation of criniferous disjoint-type maps in $\BlogP$]\label{thm:BlogP-criniferousdisjointtype}
  Let $F\in\BlogP$ be of disjoint type. The following are equivalent.
  \begin{enumerate}[(a)]
    \item $I(F)\cup\{\infty\}$ is path-connected;\label{item:BlogP-Ipathconnected}
    \item $J(F)\cup\{\infty\}$ is path-connected;\label{item:BlogP-Jpathconnected}
    \item every connected component of $J(F)$ is an arc connecting a finite endpoint to $\infty$;\label{item:BlogP-Jarcs}
    \item every $z\in I(F)$ can be connected to $\infty$ by a curve on which $F^n\to\infty$ uniformly.\label{item:BlogP-uniformstrongeremenko}
    \item $F$ is criniferous. \label{item:BlogP-criniferous}
  \end{enumerate}
\end{thm}
\begin{proof}
  The equivalence of~\ref{item:BlogP-Ipathconnected}--\ref{item:BlogP-Jarcs} follows from Proposition~\ref{prop:disjoint_crin}. 
    If $F$ is of disjoint type, then a preimage of any ray tail is again a ray tail, so the equivalence of~\ref{item:BlogP-uniformstrongeremenko} and~\ref{item:BlogP-criniferous}
       is immediate. Clearly~\ref{item:BlogP-uniformstrongeremenko} implies~\ref{item:BlogP-Ipathconnected}. 
       Finally, suppose that~\ref{item:BlogP-Jarcs} holds. Let $z\in I(F)$, set $z_n\defeq F^n(z)$, and let $\gamma_n$ be the arc in $J(F)$ connecting $z_n$ to $\infty$. 
       Then $F(\gamma_n)=\gamma_{n+1}$ for each $n$, and $\gamma_n\to\infty$ uniformly by Theorem~\ref{thm_arcs_infinity}. We have
       established~\ref{item:BlogP-uniformstrongeremenko}, which completes the proof of the theorem. 
\end{proof}

\begin{proof}[Proof of Theorem \ref{thm:criniferousdisjointtype}]
Let $f\in \B$ be of disjoint type. Recall from Proposition \ref{prop_disjoint_type} that $f$ (possibly after applying an affine conjugacy) has a disjoint-type
 logarithmic transform $F\in\BlogP$, and that $\exp(J(F))=J(f)$. Theorem~\ref{thm:criniferousdisjointtype} thus follows from Theorem~\ref{thm:BlogP-criniferousdisjointtype}. 
\end{proof}

%\begin{remark}
 %The proof shows the following stronger result. Suppose that
%  $F$ is of disjoint type, and every connected
 %component of $J(F)$ is an arc to infinity. If
 %$z_n$ is a sequence in $J(F)$ with $z_n\to\infty$, and $\gamma_n$ is the 
 %arc connecting $z_n$ to infinity, then $\gamma_n\to\infty$ uniformly. 
%\end{remark}

\section{Cantor bouquets, arc-smoothness and orderings}\label{sec_dendroids}

Following \cite[Definition 1.2]{aartsoversteegen} and \cite[Definition 2.1]{lasseBrushing}, we define Cantor bouquets as follows.

\begin{defn}[Straight brush, Cantor bouquet] \label{def_brush}
A subset $B$ of $[0,+\infty) \times (\R\setminus\Q)$ is a \emph{straight brush} if the following properties are satisfied:
\begin{itemize}
\item The set $B$ is a closed subset of $\R^2$.
\item For each $y\geq 0$ such that 
   $(\R\times\{y\})\cap B \neq \emptyset$, there is $t_y \geq 0$ so that $(\R\times\{y\})\cap B = [t_y, +\infty)\times\{y\}$. 
The latter set is called the \textit{hair} attached at $y$, and the point $(t_y, y)$ is called its \textit{endpoint}.
\item The set $\{y\colon (\R\times\{y\})\cap B \neq \emptyset\}$ is dense in $\R\setminus\Q$. 
Moreover, for every such $y$, there exist two sequences of hairs attached respectively at $\beta_n, \gamma_n \in \R\setminus\Q$ such that $\beta_n < y < \gamma_n$, $\beta_n, \gamma_n \to y$ and $t_{\beta_n}, t_{\gamma_n} \to t_y$ as $n\to\infty$.
\end{itemize}
\noindent A \emph{Cantor bouquet} is the image of a straight brush under a homeomorphism $\R^2\to\C$. The image of a hair resp.\ endpoint of the
  straight brush under this homeomorphism is called a \textit{hair} (resp. \textit{endpoint}) of the Cantor bouquet. 
\end{defn}

\begin{obs}[Cantor bouquets have accessible endpoints] \label{obs_acc_endp} By definition, any endpoint of a straight brush $B$ is accessible from its complement $\R^2\setminus B$, and so the same applies to Cantor bouquets.
\end{obs}

In order to discuss when the Julia set of a disjoint-type function $F\in\Blog$ is 
a Cantor bouquet, it will be useful to have a characterisation of Cantor bouquets in
terms of their topological properties as subsets of the plane. 
To this end, we recall the following 
continuum-theoretic concepts.
%We shall see below that given $F\in\Blog$ of disjoint type, for $J(F)$ to be a Cantor bouquet it is enough for its closure in $\widehat{\C}$ to be a \textit{smooth dendroid}, defined as follows.

\begin{defn}[Arc-smooth fans and dendroids] Let $X$ be a continuum. 
\begin{enumerate}
\item $X$ is \textit{hereditarily unicoherent} if, for all subcontinua $A, B\subset X$,
    the intersection $A\cap B$ is connected. \label{item:her_uni}
\item $X$ is \emph{uniquely arcwise connected} if for all distinct $x,y\in X$,
  there is a unique arc $[x,y]$ in $X$ connecting $x$ and $y$. When $x=y$, we set
    $[x,y]\defeq \{x\}$. If $X$ is hereditarily unicoherent and arc-connected, then
    it is uniquely arcwise connected. \label{item:arc_con}
\item A continuum satisfying \eqref{item:her_uni} and \eqref{item:arc_con} is
  called a \emph{dendroid}. 
\item \label{item:branch_point}A common endpoint of at least three different arcs in $X$ that are otherwise pairwise disjoint is called a \emph{branch point} of $X$.
\item A dendroid with a unique branch point $x_0$ is called 
 a \emph{fan} (with \emph{top} $x_0$). 
 \item 
A space $X$ is \emph{arc-smooth}\footnote{%
In the topology literature, such spaces are usually called simply ``smooth.'' Here we use the term arc-smooth to avoid confusion with smoothness in the sense of analysis. 
Arc-smoothness is a topological concept introduced in~\cite{arcsmooth}, which reduces to
the usual topological 
concept of smoothness for uniquely arcwise connected spaces, which is the only
context in which we shall apply it.}
 if it is uniquely arcwise connected and there is some $x_0\in X$ such that for any sequence $y_n$ converging to a point $y$, the arcs $[x_0, y_n]$ converge to $[x_0, y]$ in the Hausdorff metric. If we wish to specify the point $x_0$, we say that $X$ is 
   \emph{arc-smooth with respect to $x_0$.}
\item 
 If $x \in X $ is an endpoint of every arc containing it, then $x$ is called an \textit{endpoint} of $X$. 
\item 
 An
  arc-smooth fan $X$ whose endpoints are dense in $X$ is a \textit{Lelek fan}.
  \end{enumerate}
\end{defn}

\begin{rmk}
 In the introduction, we stated that, for a disjoint-type function $f$, the set 
    $J(f)\cup\{\infty\}$ is arc-smooth if and only if there is a homeomorphism of the plane that maps each component of $J(f)$ to a horizontal ray. This
    follows from Proposition~\ref{prop:smoothbouquets} below, together with the definition of a Cantor bouquet. It is true more generally that, if $X\subset\C$ is closed and $X\cup\{\infty\}$ is an arc-smooth fan
    with top infinity, then $X$ is ambiently homeomorphic to a union of horizontal rays. This follows by similar arguments as the proof 
    of~\cite[Theorem~4.1]{aartsoversteegen}, but we do not require this fact. 
\end{rmk}

\begin{rmk}\label{rmk:topinfty}
 If a fan $X$ is arc-smooth with respect to any point $y$, then it is also arc-smooth with respect to its top $x_0$. (This follows easily from the fact that the
   hair containing $y$~-- that is, the connected component of $X\setminus\{x_0\}$ containing $y$~-- cannot be accumulated on by other hairs.) 
\end{rmk}

If $X$ is a Cantor bouquet, then clearly $X\cup\{\infty\}$ is a Lelek fan. Conversely, 
 if $X\cup \{\infty\}$ is a Lelek fan, then $X$ is a Cantor bouquet if and only if every point of $X$ that is accessible from $\C\setminus X$ is an endpoint of $X$ \cite[Theorem 2.8]{nada}.

In order for the Julia set $X$ of a disjoint-type function $F\in\Blog$ to be 
 a Cantor bouquet, clearly it is necessary that $X\cup\{\infty\}$ is arc-smooth. 
 It turns out that this condition is also sufficient.
 
\begin{prop}\label{prop:smoothbouquets}
Suppose that $F\in\Blog$ is of disjoint type and set $\hat{J}(F)\defeq J(F)\cup\{\infty\}$. 
The following are equivalent.
\begin{enumerate}[(a)]
 \item $I(F)\cup \{\infty\}$ is arcwise connected.  \label{item:Iarcwise}
 \item $\hat{J}(F)$ is arcwise connected. \label{item:Jarcwise}
%% \item Every connected component of $J(F)$ is an arc to infinity.\label{item:Jarcs}
 \item $\hat{J}(F)$ is a fan with top $\infty$, and endpoints are dense in $\hat{J}(F)$.
    \label{item:Jfan}
\end{enumerate}
In particular, $J(F)$ is a Cantor bouquet if and only if $\hat{J}(F)$ is arc-smooth. 
\end{prop}
\begin{proof}
 The equivalence between~\ref{item:Iarcwise} and~\ref{item:Jarcwise} is 
  Proposition~\ref{prop:disjoint_crin}. Clearly~\ref{item:Jfan} implies~\ref{item:Jarcwise}.
  Now suppose that~\ref{item:Jarcwise} holds. 
  
  By Proposition~\ref{prop:disjoint_crin}, every connected component of $J(F)$ is an arc. 
    In particular, $\hat{J}(F)$ is arc-wise connected, and 
    no point of $J(F)$ is a branch point of $\hat{J}(F)$. 
    Since $J(F)$ has infinitely many components,
    $\infty$ is a branch point of $\hat{J}(F)$. 
    Any subcontinuum of $\hat{J}(F)$ is either a bounded sub-arc of one of the
    components of $J(F)$, or a union of several arcs to infinity in different
    components; it follows easily that $\hat{J}(F)$ is hereditarily unicoherent.
    In summary, we have shown that $\hat{J}(F)$ 
    a fan with top $\infty$. 
  
% Recall that $\hat{J}(F)$ is the union of the Julia continua of $F$, which are
%  pairwise disjoint except at infinity. Each Julia continuum is hereditarily unicoherent
%  and atriodic (that is, contains no subcontinuum whose complement has at least three
%   components). This follows from \cite[Theorem 5.6]{lasse_arclike}, that states that if $F$ 
%   is of disjoint type, then all its Julia continua have span zero. Span zero
%   continua are atriodic~\cite[p. 213]{Lelek64} and hereditarily unicoherent
%    (indeed, they are tree-like~\cite{Oversteegen_span_84}). 
%   
%  By definition, any connected subset of $J(F)$ is contained in a 
%  single Julia continuum. 
%  It follows easily that $\hat{J}(F)$ is itself hereditarily 
%  unicoherent. By atriodicity
%  of Julia continua, no finite point of $\hat{J}(F)$ is a branch point. If 
%  $\hat{J}(F)$ is arcwise connected, then $\infty$ is the endpoint of infinitely many
%  different arcs (since $J(F)=\hat{J}(F)\setminus\{\infty\}$ has infinitely many connected
%  components), and hence $\infty$ is a branch point of $\hat{J}(F)$. So
%  $\hat{J}(F)$ is a fan with top $\infty$. 
%In particular, every Julia continuum is an atriodic dendrite, and therefore an arc~\cite[Exercise~80.21]{illanesnadler_hyperspaces}. 
 
By \cite[Theorem 2.3]{lasse_arclike}, if $x \in J(F)$ is accessible from $\C\setminus J(F)$, 
  then $x$ is a terminal point of the corresponding Julia continuum, and therefore
  an endpoint of the fan $\hat{J}(F)$. 
   Since $J(F)$ is a closed set with empty interior, points accessible 
   from $\C\setminus J(F)$ are dense in $J(F)$. This establishes~\ref{item:Jfan}. 
  
  The final claim follows from~\cite[Theorem~2.8]{nada}, which was
   mentioned above. Indeed, if $\hat{J}(F)$ is 
  arcwise connected and arc-smooth, then by the above $\hat{J}(F)$ is a Lelek fan. As we just noted, the only 
  accessible points of $J(F)$ are endpoints of $\hat{J}(F)$, and thus 
  $\hat{J}(F)$ is a Cantor bouquet. 
%  , and 
%By definition, any Cantor bouquet together with infinity is clearly a smooth dendroid. For the other direction, suppose that $\hat{J}(F)$ is a smooth dendroid. Note that no point in $J(F)$ can be a branch point, since each Julia continuum is \textit{atriodic}, that is, contains no subcontinuum whose complement has at least three components; this follows from \cite[Theorem 5.6]{lasse_arclike}, that states that if $F$ is of disjoint type, then all its Julia continua have span zero and, consequently, are atriodic \cite[p. 213]{Lelek64}. Moreover, $\infty$ is clearly a branch point. Thus, $\hat{J}(F)$ is a smooth fan with top $\infty$.
%
%By \cite[Theorem 2.3]{lasse_arclike}, if $x \in J(F)$ is accessible from $\C\setminus J(F)$, then $x$ is an endpoint of $J(F)$. Since $J(F)$ is a closed set with empty interior, points accessible from $\C\setminus J(F)$ are dense in $J(F)$, and so $J(F)$ is a Lelek fan. We have shown that $\hat{J}(F)$ is a Lelek fan with top $\infty$ for which points accessible from the complement are endpoints of $J(F)$. By \cite[Theorem 2.8]{nada}, this is equivalent to $J(F)$ being a Cantor bouquet.
\end{proof}
\begin{rmk}\label{rmk:arcconnectedsubsetisfan}
 The proof also shows  the following: if $F$ is of disjoint type and 
    $X\subset \hat{J}(F)$ is a non-degenerate and arcwise connected continuum, then $X$ is either an arc or a fan with top $\infty$. 
\end{rmk}

Suppose now that $F\in\Blog$ is of disjoint type and as in the preceding
proposition; i.e., such that $\hat{J}(F)$ 
is a fan with top $\infty$. Then 
we may introduce
%each connected component of $\hat{J}(F)$ is an arc to infinity; see Proposition~\ref{obs_disjoint_crin}. This allows us to define 
a partial order on $\hat{J}(F)$ as follows: 
\begin{equation}\label{eq_order}
z\preceq w  \quad \text{ if and only if } \quad w\in [z,\infty],
\end{equation}
where $[z,\infty]$ is the unique arc in $\hat{J}(F)$ connecting $z$ and $\infty$.
We now use this order to investigate obstructions to
 the arc-smoothness of $\hat{J}(F)$. % Namely, in a rough sense, the endpoints of an arc in $J(F)$ will constitute a \textit{bad pair} if other arcs in $J(F)$ converge in ``the opposite direction''. We will show in Proposition \ref{prop:badbouquets} that $J(F)$ cannot be a Cantor bouquet in the presence of bad pairs. 
\begin{defn}[Bad pairs and sets] \label{defn_badpair} Let $F\in\Blog$ be of disjoint type and criniferous. We call an ordered pair of $(z,w)$ of points
$z,w\in J(F)$ a \emph{bad pair} if $z\prec w$, but there are sequences $z_n\to z$, $w_n\to w$ such that $w_n\prec z_n$ for all $n\geq 0$. In this case, we call $(z_n)_{n\geq 0}$ and $(w_n)_{n\geq 0}$ \emph{approximating sequences} for the bad pair. Moreover, the union of a bad pair with corresponding approximating sequences is called a \emph{bad set}.
\end{defn}

\begin{prop}[Properties of bad pairs]\label{prop_properties_bad}
Let $F\colon \T\to H$ be of disjoint type and criniferous, and let $(z,w)$ be a bad pair. Then, the following hold.
\begin{enumerate}[(a)]
\item \label{item:approx_arcs} There are approximating sequences $z_n\to z$, 
$w_n\to w$ such that $[w_n,z_n]\to [z,w]$ in the Hausdorff metric as $n\to \infty$. 
Moreover, if $z\preceq \zeta \prec \omega \preceq w$, then $(\zeta,\omega)$ 
is also a bad pair. 
\item \label{item:preimofBad} $(F(z),F(w))$ and $(F_T^{-1}(z),F_T^{-1}(w))$ are also
bad pairs, for any tract $T$ in $\T$. Moreover, if $X\subset J(F)$ is a bad set, then so 
are $F(X)$ and $F_T^{-1}(X)$ for any tract $T$ in $\T$.
\end{enumerate}
\end{prop}
\begin{proof}
To show \ref{item:approx_arcs}, let $z\prec w$ be a bad pair, 
let $z_n\to z$, $w_n\to w$ be approximating sequences, and let 
$z\preceq \zeta \prec \omega \preceq w$. Let $A$ be the Julia continuum containing
$[z,w]$ and recall that $A$ is an arc. Therefore there are open sets 
$U_{\zeta}$ and $U_{\omega}$ with disjoint closures
such that $U_{\zeta}\cup U_{\omega}$ is 
a neighbourhood of $A\setminus [\zeta,\omega]$, and such that 
$\overline{U_{\zeta}}\cap [\zeta,\omega]=\{\zeta\}$ (and similarly for $U_\omega$).
Disregarding finitely many entries, we may assume that 
$z_n\notin U_{\omega}$ and $\omega_n\notin U_{\zeta}$ for all $n$.

 Let $\omega_n$ be the last point on $[w_n,z_n]$ with $w_n\in\overline{U_{\omega}}$, and
 $\omega_n = w_n$ if no such point exists. Similarly, let
 $\zeta_n$ be the first point on $[\omega_n,z_n]$ with $z_n\in\overline{U}_{\zeta}$,
 or $\zeta_n = z_n$ if no such point exists. We claim that 
 $(\omega_n)$ and $(\zeta_n)$ are approximating sequences for $(\zeta,\omega)$ such
 that 
 $[\omega_n,\zeta_n]$ converges to $[\zeta,\omega]$ in the Hausdorff metric.
 
 Indeed, note that we can have $\omega_n = w_n$ for infinitely many $n$ only 
   if $w = \omega$; it follows that $\omega_n\to\omega$ and similarly $\zeta_n\to \zeta$.
   Moreover, any limit point of $[\omega_n,\zeta_n]$ in the Hausdorff metric 
   is contained in $[\zeta,\omega]$ by construction. Hence $[\omega_n,\zeta_n]\to[\zeta,\omega]$, as desired. 
   
%Claim~\ref{item:approx_arcs} thus follows from the following general topological
% fact: if an arc $A$ in a compact metric space $X$ is the Hausdorff limit of arcs $A_n\subset X$, then
% every sub-arc $\tilde{A}\subset A$ is
% the Hausdorff limit of some sequence of subarcs $\tilde{A}_n\subset A_n$. 
% To prove this claim, we may assume that $A\setminus\tilde{A}$ is connected
% (otherwise, apply this fact twice). \dots \textbf{To be completed.}
% 
%Thus, $[z,w]\subset A$. %We claim that we can find sequences $z_n\to z$, $w_n \to w$ with $z'_n\preceq z_n \prec w_n \preceq w'_n$, so that $[w_n,z_n]\to [z,w]$ in the Hausdorff metric.
%Let $Z$ and $W$ be open neighbourhoods, in the induced topology in $A$, of the respective components of $A\setminus [z,w]$ that contain $z$ and $w$ on their boundaries. Then, by the Boundary Bumping Theorem \cite[Theorem 5.6]{nadler_continuum},  for all $n$ large enough, there exist $z_n\in \partial Z \cap [w'_n, z'_n]$ and $w_n\in \partial W \cap [, w'_n, z'_n]$ so that $z_n\to z$, $w_n \to w$ and $[w_n,z_n]\to [z,w]$ in the Hausdorff metric, as desired.

Let $a\prec b$. In particular, $a,b\in J(F)$ and $F(a)$ and $F(b)$ belong to a same tract of $F$. Recall that for each tract $T$, $F_T^{-1}$ is continuous and $F_T^{-1}(\infty)=\infty$. Consequently, $F_T^{-1}([a,\infty])$ is an arc connecting $F_T^{-1}(a)$ to infinity, i.e., $F_T^{-1}([a,\infty]) = [F_T^{-1}(a),\infty]$. This implies that
\begin{equation}\label{eq_preim_order}
F_T^{-1}(a)\prec F_T^{-1}(b) \iff a\prec b \iff F(a)\prec F(b),
\end{equation}
and so, using continuity of $F$ and $F^{-1}_T$, \ref{item:preimofBad} follows.
\end{proof}

\begin{cor} \label{cor:badbouquets}
Suppose that $F\in\Blog$ is of disjoint type and criniferous. Then $J(F)$ is a Cantor bouquet if and only if it contains no bad set.
\end{cor}
\begin{proof}
By Proposition \ref{prop:smoothbouquets}, $\hat{J}(F)$ is a fan with top $\infty$,
 and $J(F)$ is a Cantor bouquet if and only if this fan is arc-smooth. If the fan is not arc-smooth,
  then there exist sequences $z_n\to z$ such that $[z_n,\infty]$ does not converge to
  $[z,\infty]$ in the Hausdorff metric. In particular, there are points $w_n\in [z_n,\infty)$
  that converge to a point $w\notin [z,\infty]$. Since $w$ must belong to the 
  Julia continuum containing $z$, we have that $(w,z)$ is a bad pair. The converse follows from Proposition \ref{prop_properties_bad}\ref{item:approx_arcs}.
%it suffices to show that  $\hat{J}(F)$ being an arc-smooth dendroid is equivalent to $J(F)$ not containing bad sets. Since $F$ is of disjoint type and criniferous, by Proposition~\ref{obs_disjoint_crin}, $\hat{J}(F)$ must be a dendroid, and by Proposition \ref{prop_properties_bad}\ref{item:approx_arcs}, arc-smoothness of the dendroid is clearly equivalent to the non-existence of bad sets.
\end{proof}

Recall from the introduction that, for an entire function $f$, 
a set $A\subset J(f)$ is \emph{absorbing} if it is closed and forward-invariant, the image of any arc to infinity in $A$ is also an arc to infinity, and every point in $I(f)$ is eventually mapped into $A$. The first condition is immediate for maps in $\Blog$, so the 
definition simplifies in this setting. 

\begin{defn}[Absorbing sets]
Let $F\in \Blog$. We say that a subset $X\subset J(F)$ is \textit{absorbing} if $F(X)\subset X$ and
\begin{equation}\label{eq_absorbing}
I(F)\subset \bigcup_{n=0}^{\infty} F^{-n}(X).
\end{equation}
\end{defn}
\begin{prop}\label{prop_abs_crin}
Suppose that $F\in\Blog$ is of disjoint type and $J(F)$ contains an absorbing set $X$ such that $\hat{X}\defeq X\cup \{\infty\}$ is arc-smooth. Then $F$ is criniferous.
\end{prop}

\begin{proof} 
  By Remark~\ref{rmk:arcconnectedsubsetisfan}, $\hat{X}$ is a fan with top $\infty$. (It cannot be an arc because an absorbing set must intersect infinitely many  
    different components of $J(F)$.) So by Remark~\ref{rmk:topinfty}, 
       $\hat{X}$ is arc-smooth with respect to infinity. 
       
If $z\in X\cap I(F)$, then the arcs $[F^n(z), \infty]$ must converge to $[\infty,\infty]=\{\infty\}$ by arc-smoothness. Hence
$[z, \infty]$ is a ray tail. By~ \eqref{eq_absorbing}, this means that every point $z\in I(f)$ is eventually mapped to a ray tail; i.e., $F$ is criniferous.
\end{proof}

\begin{prop} \label{prop_absJR} Let $f\in \B$ and suppose that $J(f)\cup \{\infty\}$ contains an absorbing arc-smooth dendroid $\hat{X}\defeq X\cup \{\infty\}$. Then, for all $R>0$, the set $\hat{J}_R(f)\defeq J_R(f) \cup \{\infty\}$ contains an absorbing arc-smooth dendroid with top $\infty$.
\end{prop}

\begin{proof} 
Let us fix some $R>0$ and let $Y$ be the set of unbounded components of $J_R(f)\cap X$. We claim that $\hat{Y}\defeq Y\cup \{\infty\}$ is an absorbing arc-smooth dendroid. Indeed, since $\hat{X}$ is a dendroid, $\hat{Y}$ is a collection of unbounded curves together with infinity. Moreover, $\hat{Y}$ is a closed subset of $\hat{X}$, as both $\hat{J}_R(f)$ and $\hat{X}$ are closed. Hence, $\hat{Y}$ is an arc-smooth dendroid.

Let $\gamma$ be a component of $Y$. Then, $f(\gamma)\subset J_R(f)$ and $f(\gamma)$ must be an unbounded curve in $X$, by the definition of absorbing set. Hence, $f(\gamma)\subset Y$, and so $f(Y)\subset Y$. If $z\in I(f)$, then there are $N_0, N_1$ such that $f^{n}(z)\in J_R(f)$ for all $n\geq N_0$ and  $f^{n}(z)\in X$ for all $n\geq N_1$. Hence $f^{n}(z)\in J_R(f)\cap X$ for all $n\geq \max\{N_0, N_1\}$. Moreover, by arc-smoothness of $\hat{X}$, if $z\in X\cap I(f)$, then $[z, \infty]$ is a ray tail, as the arcs $[f^n(z), \infty]$ must converge to $[\infty,\infty]$.  Hence, $z$ cannot escape to infinity through bounded components of $J_R(f)\cap X$, and so $I(f)\subset \bigcup_{n=0}^{\infty} f^{-n}(Y)$.
\end{proof}

Finally, the following  characterization of Cantor bouquets will allow us to show in Proposition \ref{prop_UHSC_CB} that disjoint type maps in $\Blog$ satisfying a uniform head-start condition have Cantor bouquet Julia sets.

\begin{prop}\label{prop:bouquetordering}
Let $A\subset\C$ be a closed set such that $\hat{A} \defeq A\cup\{\infty\}$ is connected.
Then $A$ is a Cantor bouquet if and only if there is a partial strict ordering
$\prec$ on $\hat{A}$ with the following properties.
\begin{enumerate}[(a)]
\item \label{item:aprecinfty} $a\prec \infty$ for all $a\in A$.
\item \label{item:comparable} Two points $z,w\in A$ are comparable under $\prec$ if and only if
they belong to the same connected component
of $A$.
\item \label{item:ordercontinuity} Suppose $z,w\in \hat{A}$ with $z\prec w$. Then there
are neighbourhoods $U_z$ and $U_w$ of $z$ and $w$ such that 
$\omega \nprec \zeta$ for all $\zeta \in U_z$ and $\omega \in U_w$.
\item \label{item:accessibility} If $x\in A$ is accessible from $\C\setminus A$, then $x$ is a minimal element of $A$.
%Every connected component of $A$ contains at most one point that is accessible from $\C\setminus A$.  
\end{enumerate}
\end{prop}
\begin{proof}
 Let us fix $A$ as in the statement. Then, \cite[Theorem 2.8]{nada} states that $A$ is a Cantor bouquet if and only if
\begin{enumerate}[(i)]
\item \label{item_Lelek} $\hat{A}$ is a Lelek fan with top $\infty$, and
\item \label{item_accessible_enpoints} if $x\in A$ is accessible from $\C\setminus A$, then $x$ is an endpoint of $A$.
\end{enumerate}
%\ref{item:accessibility} holds. Thus, we need to show that $\hat{A}$ is a smooth fan if and only if \ref{item:aprecinfty}, \ref{item:comparable} and \ref{item:ordercontinuity} hold.
If $\hat{A}$ is a Lelek fan with top $\infty$, then each connected component of $A$ is an arc to infinity. We can then define the partial order ``$\prec$'' in $\hat{A}$ as follows. For $z,w\in \hat{A}$,
\begin{equation}
z\prec w  \quad \text{ if and only if } \quad w\in [z,\infty],
\end{equation}
where $[z,\infty]$ is the unique arc in $\hat{A}$ connecting $z$ and $\infty$. Then, \ref{item:aprecinfty} and \ref{item:comparable} hold by definition of the order, and \ref{item:ordercontinuity} follows from the arc-smoothness of $\hat{A}$. Moreover, minimal points of $\hat{A}$ are the endpoints of the Lelek fan, and so, using \ref{item_accessible_enpoints}, we obtain \ref{item:accessibility}.

The converse proceeds essentially as the proof of \cite[Proposition 4.4]{rrrs}. First we note that if $\gamma$ is a component of $A$, then the topology of $\hat{\gamma}\defeq \gamma \cup \{\infty\}$ as a subset of $\widehat{\C}$ agrees with the order topology induced by $\prec$: the map $\id\colon \hat{\gamma} \to (\hat{\gamma}, \prec)$ is continuous, since, by \ref{item:ordercontinuity}, the sets 
$$ U_a^-\defeq \{w\in \gamma \colon w\prec a\} \quad \text { and } \quad  U_a^+ \defeq \{w\in \gamma \colon a\prec w\},$$
that form a subbasis for the order topology, are open in $\hat{\gamma}$ for every $a\in \hat{\gamma}$. Then, since $\hat{\gamma}$  is compact and the order topology on $\hat{\gamma}$ is Hausdorff, we have that $\id$ is a homeomorphism and both topologies agree. By this and since $\hat{\gamma}$ is a compact connected set, it follows from a well-known characterization of the arc that $\hat{\gamma}$ is either a point or an arc; see \cite[Theorems 6.16, 6.17]{nadler_continuum} and compare to \cite[Theorem~A.5]{rrrs}. Then, $\hat{A}$ is clearly a fan with top $\infty$ and minimal points are endpoints of the fan. In particular, by \ref{item:accessibility}, we have \ref{item_accessible_enpoints}.

To show arc-smoothness of $A$, fix $z_0\in A$ and given a sequence $z_n\to z_0$, denote by $I_n$ the arc connecting $z_n$
to $\infty$ in $A$ for all $n\geq 0$; that is, $I_n \defeq \{ w\colon z_n\prec w \}$. Then, by~\ref{item:ordercontinuity}, the arcs $\{I_n\}_{n=0}^{\infty}$ cannot accumulate at any point $w\prec z_0$, and so $\hat{A}$ is an arc-smooth fan. Since $A$ has empty interior, points accessible from $\C\setminus A$ are dense in $A$, and so by \ref{item_accessible_enpoints}, $\hat{A}$ is a Lelek fan, which shows \ref{item_Lelek} and concludes the proof.
\end{proof}

\begin{obs}\label{obs_arc}
We have shown in the proof of Proposition \ref{prop:bouquetordering} that if $A\subset \C$ is a closed set such that $A\cup 
\{\infty\}$ is connected and for which \ref{item:aprecinfty}-\ref{item:ordercontinuity} hold for a partial strict ordering $\prec$ on $A$, then each connected component of $A$ is an arc to infinity, and minimal points are the endpoints of the arcs.
\end{obs}

\section{Absorbing Cantor bouquet implies global Cantor bouquet}\label{sec_abs}

Our main goal in this section is to prove the following theorem, that is key to the proof of Theorem~\ref{thm:cantor}.

\begin{thm}\label{thm:bad}
Let $F\in\Blog$ of disjoint type. Suppose that $F$ is criniferous, but $J(F)$ is not a Cantor bouquet. Then there exists a closed subset $A\subset I(F)$ such that every relatively open subset of $A$ contains a bad set $X$ such that $\re F^n|_X\to \infty$ uniformly. 
\end{thm} 

\begin{cor}\label{cor:main}
Suppose $F\in\Blog$ is of disjoint type and $\hat{J}(F)$ contains an absorbing arc-smooth dendroid $\hat{X}\defeq X\cup \{\infty\}$. Then $J(F)$ is a Cantor bouquet.
\end{cor}
\begin{proof}[Proof of Corollary~\ref{cor:main}, using Theorem~\ref{thm:bad}] By Proposition \ref{prop_abs_crin}, $F$ is criniferous. Since $\hat{X}$ is arc-smooth, by Proposition~\ref{prop_properties_bad}\ref{item:approx_arcs}, $X$ contains no bad set, and so, by Proposition~\ref{prop_properties_bad}\ref{item:preimofBad}, $F^{-n}(X)$ also contains no bad set, for every $n\geq 0$.
	
Suppose, by contradiction, that $J(F)$ is not a Cantor bouquet. Then, by Theorem~\ref{thm:bad}, there exists a closed subset $A\subset I(F)$ such that every relatively open subset of $A$ has a bad set. This means that the closure of $F^{-n}(X)\cap A$ in $A$ cannot contain any non-empty relatively open subset of $A$; in other words,  $F^{-n}(X)\cap A$ is nowhere dense for all $n\geq 0$. But since $A\subset I(F)$ and $X$ is by assumption an absorbing set,
\[ A = \bigcup_{n=0}^{\infty} F^{-n}(X)\cap A. \] 
However, $A$ being a countable union of nowhere-dense sets is a contradiction to Baire's theorem. 
\end{proof}

\begin{proof}[Proof of Theorem \ref{thm:cantor}, using Corollary \ref{cor:main}]
The implication \ref{item:Jcantor}$\implies$\ref{item:absorbing} follows from \cite[Theorem~1.7]{mio_newCB}, and since by definition any Cantor bouquet together with infinity is clearly an arc-smooth dendroid, \ref{item:absorbing}$\implies$\ref{item:abssmooth} is trivial. Let $f$ be a disjoint type map with an absorbing arc-smooth dendroid  $X\cup \{\infty\}\subset J(f)\cup \{\infty\}$. Then, possibly after applying an affine conjugacy, $f$ has a disjoint type logarithmic transform $F$, see Proposition \ref{prop_disjoint_type}. By \eqref{eq_commute_log} and since $\exp(J(F))=J(f)$, $\exp^{-1}(X)\cup \{\infty\}$ is an absorbing arc-smooth dendroid contained in $J(F) \cup \{\infty\}$. By Corollary~\ref{cor:main}, $J(F)$ is a Cantor bouquet, and so, it is easy to see that $\exp(J(F))=J(f)$ is also a Cantor bouquet; see the proof of \cite[Corollary 6.3]{lasseBrushing} for details. We have shown that \ref{item:abssmooth}~$\implies$~\ref{item:Jcantor}.
\end{proof}

For the remainder of the section we assume the hypotheses of Theorem~\ref{thm:bad}. That is, $F\in\Blog$ is criniferous and of disjoint type, but $J(F)$ is not a Cantor bouquet. 
In particular, by Corollary \ref{cor:badbouquets}, $J(F)$ contains a bad pair.

Roughly speaking, the structure of the proof of Theorem \ref{thm:bad} is as follows: 

\begin{enumerate}[(i)]
\item Following a series of technical observations, we show in Proposition \ref{prop:smallbadsets} that for any bad pair $(z,w)$ so that $[z,w]$ has big enough diameter, we can find bad pairs of points in $[z,w]$ belonging to bad sets $Y_k$ with arbitrarily small diameter that escape to the right faster than some prescribed real sequence $\alpha_n \to \infty$.
\item \label{item_ii} Next, we choose another sequence $a_n \to \infty$ with $a_n<\alpha_n$, and construct bad sets $X_n$ of small diameter that escape to the right much faster than $(a_n)_n$, but on which we have control on the maximum of their real parts. These will be defined as preimages of translates of some of the bad sets $Y_k$.
\item Finally, the set $A$ we are looking for will be the closure of all points that escape to the right much faster than $(a_n)_n$. We will prove that every relatively open subset of $A$ contains an $n$-th preimage of some set $X_n$ from \ref{item_ii}.
\end{enumerate}

We may assume without loss of generality that $F$ is \textit{normalized}, that is, $F\colon \T\to \H$ such that 
\begin{equation}\label{eq_normalized}
\vert F'(z)\vert>2 \quad  \text{ for all }z\in \T.
\end{equation}
We also assume that 
\begin{equation}\label{eq_tractstoright}
\T\subset \{z\in \C \colon \Rea z>2\pi +2\}.
\end{equation}
Otherwise, we can pre-compose $F$ with a translation and restrict to a smaller domain. The resulting map is affine equivalent to $F$, and so they are conjugate on their Julia sets; see \cite[\S 2]{lasseRigidity} for details.

We define
\begin{equation}\label{eq_Et}
L\colon [0,\infty) \to [0,\infty); \qquad t \mapsto \min( t/2, \max(2 , 8\pi \cdot \log( t ) ) ).
\end{equation}
%\begin{equation}\label{eq_Et}
%L\colon [0,\infty) \to [0,\infty); \qquad t \mapsto  \max(2 , 8\pi \cdot \log( t ) ) .
%\end{equation}
Observe that $L$ is strictly increasing, and therefore has an inverse $E\colon [0,\infty)\to [0,\infty)$. In particular, for all sufficiently large $t$, we have
\begin{equation}\label{eq_LE}
L(t) = 8\pi \cdot \log(t) \quad \text{ and } \quad E(t) = \exp(t/(8\pi)).
\end{equation}

\begin{prop}\label{prop_diam}
For any connected set $B \subset \overline{\H}_1$ with $s\leq \diam(B)\leq t$ and any tract $T$, $\diam(F_T^{-1}(B)) \leq \min(\diam(B)/2 ,L(\diam(B)))\leq t/2$. If, in addition, $F(B) \subset \overline{\H}_1$, then $\diam(F(B))\geq \max(2s, E(s)).$
\end{prop}
\begin{proof}
By \cite[Lemma 3.1]{rrrs}, if $A\subset T$ for some tract $T$ and $F(A) \subset  \overline{\H}_1$, then
\begin{equation*}
\diam(A) \leq \max(2 , 8\pi \cdot \log(\diam(F(A)) ).
\end{equation*}
By this and \eqref{eq_normalized} the proposition follows. 
\end{proof}

Next we define the sequence $(\ell_n)_{n\geq 0}$ with  
\[ \ell_n \defeq \left(\frac{3}{2}\right)^n. \]

Let $(\alpha_n)_{n\geq 0}$ be an increasing sequence of positive numbers with $\alpha_n\to\infty$ such that for every tract $T$, 
\begin{equation}\label{eq_alphan}
\diam(\{ z\in \overline{T} \colon \re z \leq \alpha_n \}) \leq \ell_n.
\end{equation}
Clearly such a sequence exists, as for each $r>0$, there is $M\defeq M(r)>0$ such that $\diam(T\setminus \H_r)<M$ for all tracts $T$, and $M(r)\to \infty$ as $r\to \infty$; see \cite[Observation~3.4]{lasse_arclike}. Moreover, $\alpha_0$ can be chosen greater than $2\pi+2$ by the assumption \eqref{eq_tractstoright} on~$\T$.

\begin{lem}\label{lem:sum} For each $n\geq 0$, let
\[ M_n \defeq \sum_{k=1}^{\infty} L^k( \ell_{n+k}). \] 
Then there is a constant $C>2$ such that $M_n < C\cdot (n+1)$ for all $n\geq 0$. 
\end{lem}      
\begin{proof}
By definition of $L$ in \eqref{eq_Et}, for all  $n\geq 0$ we have
\begin{equation}\label{eq_Lelln}
L(\ell_n) \leq 8   \pi   n   \log(3/2) = c\cdot n,
\end{equation}
with $c\defeq  8 \pi  \log(3/2) $. 
Also by definition of $L$, 
\begin{equation}
L^k(\ell_n) \leq c \cdot n \cdot    2^{-(k-1)} \quad\text{ for all } \quad n\geq k \geq 1.
\end{equation}
Now, for $n\geq 1$,
\[  M_n = \sum_{k=1}^{\infty} L^k(\ell_{k+n}) 
\leq 
\sum_{k=1}^{\infty} c\cdot \frac{(k+n)}{2^{k-1}} \leq
c\cdot n \cdot \sum_{k=1}^{\infty} \frac{1+k/n}{2^{k-1}}. \] 
For $k\geq n>1$, we have $1+k/n < (3/2)^{k}$, and so
\[ M_n \leq 2c\cdot n \cdot \sum_{k=1}^{\infty} \left(\frac{3}{4}\right)^{k-1} = 8c\cdot n. \] 
Choosing $C>\max\{M_0,M_1, 8c\}$, the claim follows.
\end{proof}

\begin{lem}\label{lem:uniformescape}
Let $n\geq 0$, and let $A\subset J(F)\cap \H_{\alpha_n}$ be a compact and connected set with $\diam A > M_n$ and such that $F^j(A)$ is totally contained in some tract for each $j\geq 0$. Then there is $z \in A$ such that $ F^j(z) \in \H_{\alpha_{j+n}}$ for all $j\geq 0$. 
\end{lem}
\begin{proof}
For each $j\geq 0$, let $T_j$ be the tract containing $F^j(A)$, let
\begin{equation*}
D_j^j\defeq \{ z\in T_j \colon \re z \leq \alpha_{j+n} \},
\end{equation*}
and note that by \eqref{eq_alphan}, $\diam(D^j_j)\leq \ell_{j+n}.$
For each $0\leq i<j$, we define inductively $D^i_j\defeq F_{T_i}^{-1}(D_j^{i+1})$. By Proposition \ref{prop_diam},
$\diam D_j^i\leq L(\diam D_j^{i+1})$, and, in particular, 
\[ \diam D_j^0 \leq L^j(\ell_{n+j}). \] 
Then, by definition of $M_n$, we have  
\[ \sum_{j=1}^{\infty} \diam D_j^0 \leq M_n < \diam A. \]
Since $A$ is connected, we have 
\[ A \setminus \bigcup_{j=1}^{\infty} D_j^0 \neq \emptyset, \]
which proves the claim. 
\end{proof}

\begin{cor}\label{cor:fastbadpair}
Let $C$ be the constant from Lemma~\ref{lem:sum}, let $n\geq 0$, and suppose that $(z,w)$ is a bad pair such that $\diam([z,w])\geq 3 C\cdot (n+1)$ and $[z,w]\subset \H_{\alpha_n}$. Then there exists a bad set $X$ containing a bad pair $(a,b)$ with $z\preceq a \prec b \preceq w$ and such that $F^j(X)\subset \H_{\alpha_{j+n}}$ for all $j\geq 0$.
\end{cor}
\begin{proof}
We may assume without loss of generality that $\diam([z,w]) = 3C\cdot (n+1)$, since otherwise, by Proposition \ref{prop_properties_bad}\ref{item:approx_arcs}, we can replace $w$ by the smallest point in $[z,w]$ for which this equality holds. Also by Proposition \ref{prop_properties_bad}\ref{item:approx_arcs}, we can choose approximating sequences $z_k \to z$ and $w_k \to w$ such that $[w_k,z_k]\to [z,w]$ in the Hausdorff metric. 

Again by Proposition \ref{prop_properties_bad}\ref{item:approx_arcs}, we can choose a bad pair $(\zeta,\omega)$ with $z\prec \zeta \prec \omega \prec w$ and such that $\diam([z,\zeta]) = \diam([\omega,w]) = C\cdot (n+1)$. In particular, there exist approximating sequences $\zeta_k \to \zeta$ and $\omega_k \to \omega$ with $\zeta_k,\omega_k \in [w_k,z_k]$, and such that $[\zeta_k,z_k]\to [z,\zeta]$ and $[w_k,\omega_k]\to [\omega, w]$. Note that by the choice of $\zeta$ and $\omega$, $[z,\zeta] \cup [\omega,w]\subsetneq [z,w]$, and hence, $[\zeta_k,z_k] \cup [w_k,\omega_k] \subsetneq [w_k,z_k]$ for $k\geq k_0$ with $k_0$ large enough. Consequently, $\omega_k \prec \zeta_k$ for all $k\geq k_0$. 

By Lemma \ref{lem:sum}, $\diam([z,\zeta]) = C\cdot (n+1) > M_n$, and, by assumption, $[z,\zeta]\subset \H_{\alpha_n}$.
Since $[\zeta_k,z_k]\to [z,\zeta]$ and any connected subset of $J(F)$ must be totally contained in some single tract, we can apply Lemma~\ref{lem:uniformescape} to $[\zeta_k,z_k]$ for all $k\geq k_1$, with $k_1\geq k_0$ large enough such that $[\zeta_k,z_k] \subset \H_{\alpha_n}$ and $\diam([\zeta_k,z_k])>M_n$. We obtain
$a_k\in [\zeta_k,z_k]$ with $\re F^j(a_k) \geq \alpha_{j+n}$ for all $j\geq 0$. Likewise, we can find $b_k\in [w_k,\omega_k]$ with analogous properties. Passing to subsequences, we may assume
that $a_k$ and $b_k$ converge to some $a\in [z,\zeta]$ and $b\in [\omega,w]$, respectively. By Proposition \ref{prop_properties_bad}\ref{item:approx_arcs}, $(a,b)$ is a bad pair, and by construction, $a_k\to a$ and $b_k \to b$ are approximating sequences. Thus,\[ X \defeq \{a,b\} \cup \bigcup_{k=k_1}^{\infty} \{a_k,b_k\} \]
is the desired bad set.
\end{proof}

\begin{obs}\label{obs_C2C3}
There are constants $C_2>1$ and $C_3>3C$ such that for all $n\geq 0$, $L(C_2\cdot \ell_{n+1}) < C_3\cdot (n+1)$ and 
$(C_2-1)\cdot \ell_{n} > 2 \cdot C_3 \cdot (n+1)$. 
\end{obs} 
\begin{proof}
Note that for $C_2$ large enough, by \eqref{eq_LE}, $L(C_2\cdot \ell_{n+1})\leq 8\pi (n+1)\log(C_2 \cdot 3/2)$. Thus, for both inequalities to hold, we choose $C_2$ and $C_3$ such that 
\[8\pi\log(C_2\cdot 3/2)<C_3<(3^n(C_2-1))/(2^{n+1}(n+1)) \text{ for all } n\geq 0 . \]
Choosing any $C_2\geq \max\{3C,450\}$ and $C_2=C_3$, the claim follows. 
\end{proof}

\begin{obs}\label{obs:longimagecurve} Let $C_2$ and $C_3$ be the constants from Observation \ref{obs_C2C3} and suppose that $\gamma\subset J(F)$ is a curve of diameter at least $C_2\cdot \ell_n$, for some $n\geq 0$. Then there is a compact subcurve $\gamma_0\subset \gamma\cap \H_{\alpha_n}$ such that $\diam(\gamma_0) = C_3\cdot (n+1)$. 
\end{obs}
\begin{proof}
If $\gamma\subset \H_{\alpha_n}$, then we take $\gamma_0\defeq \gamma$, since $C_2\cdot \ell_n > C_3 \cdot (n+1)$. Otherwise, by definition of $\alpha_n$ in \eqref{eq_alphan}, $\diam(\gamma\setminus \H_{\alpha_n})\leq \ell_n$. Let $\delta$ be the largest diameter of a component of $\gamma\cap \H_{\alpha_n}$, and let $\gamma_0$ be such component. Since $\gamma$ is connected, any point in $\gamma\cap \H_{\alpha_n}$ can be joined to the set $\gamma\setminus\H_{\alpha_n}$ by a subcurve of $\gamma$ of length at most $\delta$, and so
\[ C_2\cdot \ell_n \leq \diam \gamma \leq 2\delta + \ell_n. \]
It follows from Observation \ref{obs_C2C3} that $\delta > C_3\cdot (n+1)$, as required. 
\end{proof} 

\begin{prop}\label{prop:smallbadsets}
Suppose that $(z,w)$ is a bad pair such that $\diam([z,w])\geq C_2$. Then, for every $\eps>0$, there is a bad set $Y$ of diameter less than $\eps$, with a bad pair of points in $[z,w]$ and so that $F^n(Y)\subset \H_{\alpha_n}$ for all $n\geq 0$. 
\end{prop}
\begin{proof}
We claim that there is a sequence of compact curves $\gamma_n \subset F^n([z,w])\cap \H_{\alpha_n}$ such that $\diam \gamma_n = C_3\cdot (n+1)$ and $\gamma_{n+1}\subset F(\gamma_n)$ for all $n\geq 0$. Indeed, if we let $\gamma\defeq [z,w]$, this follows inductively from the assumption $\diam([z,w])\geq C_2$, Observation~\ref{obs:longimagecurve} and the fact that $\diam(F(\gamma_n)) \geq E(\diam \gamma_n) = E(C_3\cdot (n+1)) > C_2 \cdot \ell_{n+1}$, where we have used Proposition \ref{prop_diam} and Observation \ref{obs_C2C3}.

Let us fix $\eps>0$, and some $n>0$ such that
\begin{equation}\label{eq_nforeps}
L^n(C_3\cdot (n+1)) < \eps/3.
\end{equation}
Recall that by Proposition \ref{prop_properties_bad}, $(F^n(z),F^n(w))$ is a bad pair, and moreover, the endpoints of $\gamma_n$ are also a bad pair. Then, Corollary \ref{cor:fastbadpair} provides a bad set $X$ that contains a bad pair $(a,b)$ of points in $\gamma_n$, along with approximating sequences
$a_k\to a$, $b_k\to b$, and so that $F^j(X)\subset \H_{\alpha_{j+n}}$ for all $j\geq 0$. 

Let $\zeta$ and $\omega$ be the respective $n$-th preimages of $a$ and $b$ in $\gamma_0$. In particular, by Proposition \ref{prop_properties_bad}\ref{item:preimofBad}, for each $0\leq j<n$, $(F^j(\zeta),F^j(\omega))$ is a bad pair of points in $\gamma_j$. Let $\zeta_k\to \zeta$ and $\omega_k\to \omega$ be corresponding approximating sequences,
with $F^n(\zeta_k)=a_k$ and $F^n(\omega_k)=b_k$. Note that, by construction, 
\[F^j(\zeta),F^j(\omega)\in \gamma_j\subset \H_{\alpha_j} \quad \text{ for all } \quad 0\leq j\leq n-1.\]
By continuity, we can choose $k_0$ large enough so that all points $(F^j(\zeta_k))_{k\geq k_0}$, $(F^j(\omega_k))_{k\geq k_0}$ are also in $\H_{\alpha_j}$. Moreover, we take $k_0$ sufficiently large so that
\[\diam(\{\zeta_k \colon k\geq k_0\})<\eps/3 \quad \text{and}\quad \diam(\{\omega_k \colon k\geq k_0\})<\eps/3. \]
In addition, by Proposition \ref{prop_diam} and \eqref{eq_nforeps}
\[ |\zeta - \omega| \leq L^n(|a-b|) \leq L^n(\diam \gamma_n) = L^n(C_3\cdot (n+1)) < \eps/3. \]
Consequently, 
\[ Y \defeq \{\zeta,\omega\} \cup \bigcup_{k=k_0}^{\infty} \{\zeta_k,\omega_k\} \]
is the desired bad set.
\end{proof} 

\begin{prop}\label{prop:bad2}
There is an increasing sequence $(a_n)_{n\geq 0}$  of positive numbers, tending to infinity, such that for infinitely many $n\geq 0$, there is a bad set $X_n\subset J(F)$ with the following properties.
\begin{enumerate}[(a)]
\item $\min_{x\in X_n} \re F^j(x) - a_{n+j}$ is positive for all $j$, and tends to infinity as $j\to\infty$;\label{item:fastescape}
\item $X_n\subset \H_{a_n}\setminus \H_{2a_n}$.\label{item:nottoolarge}
\item $\diam X_n \leq 1$. \label{item:smalldiameter}
\end{enumerate}
\end{prop} 
\begin{proof}
Recall that, by assumption, $F$ is criniferous and of disjoint type, and $J(F)$ is not a Cantor bouquet. Then, by Corollary \ref{cor:badbouquets}, there is a bad pair $(z,w)$ in $J(F)$. By passing to a forward iterate and using Proposition \ref{prop_diam}, we may assume that $|z-w|\geq C_2$. 

Let us choose a sequence of positive $\epsilon_k\to 0$, and for each $k\geq 0$, let $Y_k$ be the bad set provided by Proposition~\ref{prop:smallbadsets}, that contains a bad pair in $[z,w]$ and so that  $\diam(Y_k)<\epsilon_k$. Moreover, recall that 
\begin{equation}\label{eq_FjY}
F^j(Y_k)\subset \H_{\alpha_j} \text{ for all }j\geq 0.
\end{equation}
Since we have chosen $\alpha_0>2\pi+2$, we can take an increasing sequence $a_n\to \infty$ such that 
\begin{equation}\label{eq_an}
a_0>\pi+1 \quad \text{ and } \quad  a_{2n} < \alpha_n/2 \quad \text{ for all } n\geq 0.
\end{equation}
Let us fix some $m\geq 0$ sufficiently large so that $[z,w]\subset \C\setminus \H_{a_m-2}$. Moreover, for each $j\geq 0$, let $T_j$ be the tract that contains $F^j([z,w])$. We can then choose $k\defeq k(m)$ large enough so that if $Y\defeq Y_k$, then 
\begin{equation} \label{eq_diamFjY}
\diam(F^j(Y))\leq 1 \quad \text{ and } \quad F^j(Y)\subset T_j \quad\text{ for each } 0
\leq j\leq 2m,
\end{equation}
since by \eqref{eq_FjY}, the sets $F^j(Y)$ do not accumulate to the boundary of tracts. In particular, since $Y \cap [z,w]\neq\emptyset$, $Y\subset \C\setminus \H_{a_m}$. Moreover, $F^j(Y)\subset \H_{\alpha_j}\subset \H_{a_{2j}}\subseteq \H_{a_{m+j}}$ for all $j\geq m$. Consequently, there exists
\[j_0\defeq \max\{0 \leq j< m \colon  F^{j}(Y)\subset \C\setminus \H_{a_{m+j}} \}. \]
We set $n\defeq n(m) = m + j_0$.
\begin{claim}
There is a bad set $X_n\subset \H_{a_n}\setminus \H_{a_n+\pi +1}$ such that $F(X_n)=F^{j_0+1}(Y)+2\pi i K$ for some $K\in \Z$. 
\end{claim}
\begin{subproof} Let $T\defeq T_{j_{0}}$ be the tract that contains $F^{j_0}(Y)$. Since $F\vert_T\colon T\to \H$ is a conformal map, by $T$ being a tract, \eqref{eq_diamFjY} and Proposition \ref{prop_diam}, $F^{-1}_T(\{F^{j_0+1}(Y)+2\pi i K\}_{K\in \Z})$ consists of a collection $\mathcal{C}$ of bad sets of diameter at most one, that accumulate at $+\infty$. In addition, since $F$ is normalized, \eqref{eq_normalized}, for each $K\in \Z$,
\begin{equation}\label{eq_dist_preim}
\dist(F^{-1}_T(F^{j_0+1}(Y)+2\pi i K), F^{-1}_T(F^{j_0+1}(Y)+2\pi i (K-1)))<\pi.
\end{equation}
Recall that $F^{j_0}(Y)\in \mathcal{C}$ and $F^{j_0}(Y)\subset \C \setminus \H_{a_{n}}$. In particular, not all sets in $\mathcal{C}$ are in $\H_{a_{n}}$, and so, there is $K\in \Z$ such that 
$X_n\defeq F^{-1}_T(F^{j_0+1}(Y)+2\pi i K)\subset \H_{a_{n}}$ but $F^{-1}_T(F^{j_0+1}(Y)+2\pi i (K-1))\nsubset \H_{a_{n}}$. Then, by \eqref{eq_dist_preim} and since $\diam(X_n)\leq 1$, the claim follows. 
\end{subproof}

In particular, by $2\pi i $-periodicity of $F$ and the definition of $j_0$, for all $j\geq 1$,
\begin{equation*}
F^j(X_n) =F^{j-1}(F^{j_0+1}(Y)+2\pi i K) \subset \H_{a_{m + j_0 + j}} = \H_{a_{n+j}}.
\end{equation*}
In addition, by \eqref{eq_FjY}, we also have that 
\[F^j(X_n) \subset \H_{\alpha_{j_0+j}}, \]
and when $j \geq m-j_0=n - 2j_0$,
\[ \alpha_{j_0+j} - a_{n+j} > 2 a_{2j_0 + 2j} - a_{n+j} \geq  a_{n+j} \to \infty,  \]
so \ref{item:fastescape} holds for $X_n$. By the claim and \eqref{eq_an}, we also have \ref{item:nottoolarge}. Finally,~\ref{item:smalldiameter} holds as $X_n$ is a preimage of a set of diameter at most $1$, see \eqref{eq_diamFjY} and Proposition \ref{prop_diam}.
\end{proof} 

\begin{proof}[Proof of Theorem~\ref{thm:bad}]
Let $(a_n)_n$ be the sequence from Proposition \ref{prop:bad2}, and set 
\begin{align*} A_1 &\defeq \{z\in J(F)\colon \re F^n(z) > a_n \text{ for all $n$, and } \re F^n(z) - a_n \to\infty\}; \\ 
A &\defeq \overline{A_1}. \end{align*}
Clearly, if $z\in A$, then $\re F^n(z) \geq a_n$ for all $n$, and so $A$ is a closed subset of $I(F)$. Let us fix any $z\in A_1$ and $\epsilon>0$. For each $n\geq 0$, let $T_n$ be the tract containing $z_n \defeq F^n(z)$.
\begin{claim} There is a constant $D$, that does not depend on $z$, such that for infinitely many $n \geq 0$, there is a bad set $\mathcal{Z}_n$ such that 
\begin{enumerate}[(a)]
\item $\min_{x\in \mathcal{Z}_n} \re F^j(x) - a_{n+j}$ is positive for all $j$ and tends to infinity as $j\to\infty,$ \label{item:newa}
\item $\dist(\mathcal{Z}_n, z_n)\leq D$;
\item $\diam(\mathcal{Z}_n)\leq 1$.
\end{enumerate}
\end{claim}
\begin{subproof} %For each $n\geq 0$, let $X_n \subset \H_{a_n}\setminus \H_{2a_n}$ be the bad set provided
By Proposition~\ref{prop:bad2}, there is an infinite set $S\subset \N_0$  such that for all $n\in  S$ there exists a bad set $X_n \subset \H_{a_n}\setminus \H_{2a_n}$ with the properties on the statement of the proposition. By Lemma~\ref{lem_distpreim}, for every $n$ such that $n+1\in S$ and large enough so that $\Rea z_{n+1}\geq 2a_{n+1}$, there is $m\in\Z$ such that the distance between $\mathcal{Z}_n\defeq F_{T_{n}}^{-1}(X_{n+1}+ 2 \pi i m)$ and $z_{n}$ is bounded by a universal constant~$D$. Since $X_{n+1}$ has diameter at most $1$, so does $\mathcal{Z}_n$ by Proposition \ref{prop_diam}. Then, by Proposition~\ref{prop:bad2}\ref{item:fastescape} applied to $X_{n+1}(=F(\mathcal{Z}_n))$, we are only left to check that $\mathcal{Z}_n\subset \H_{a_n}$ to conclude \ref{item:newa}. This holds for all $n$ sufficiently large so that $\Rea z_n-a_n>D+1.$
\end{subproof}

Let $J\defeq \{j\geq 0 \colon \re z_j -a_j<D+1\}$, and set $k\defeq \max J$ if $J\neq\emptyset$, and $k=0$ otherwise, noting that since $z\in A_1$, the maximum exists in the first case. Let us choose any $n$ as in the claim and large
 enough so that 
\begin{equation}\label{eq_choicek}
(D+1)/2^{n-j}<\min\{\epsilon, \re z_j -a_j\} \quad \text{ for all } j\in \{0,\dots, k\}.
\end{equation}
For each $j<n$, let 
$$W_j \defeq F^{-1}_{T_j}\circ F^{-1}_{T_{j+1}} \cdots \circ F^{-1}_{T_{n-1}}( \mathcal{Z}_n),$$
which, by Proposition \ref{prop_properties_bad}\ref{item:preimofBad}, is a bad set. Then, 
by Proposition~\ref{prop_diam}, the previous claim and \eqref{eq_choicek},
\[\diam(W_j\cup \{z_j\})<(D+1)/2^{n-j} \quad \text{ and } \quad  (W_j\cup \{z_j\})\subset \H_{a_j} .\]
Thus, $W_0\subset \{\zeta \in A_1\colon \lvert \zeta -z\rvert < \eps\}$, and the statement follows.
\end{proof}

\section{Proof of Theorem \ref{thm_CB}}\label{sec_CB}
It follows from \cite[Theorem~1.6]{mio_newCB} that if $f\in \CB$, then there exists an absorbing Cantor bouquet $X\subset J(f)$ with $J_{Q}(f)\subset X$ for some $Q\geq 0$. In particular,~\ref{item:absorbingfan} holds. Property~\ref{item:newCB} also follows: let $n_0$ be so large that $f^{n_0}(E)\subset X$, and let 
$\gamma_z$ be the arc connecting $z$ to $\infty$ in $X$.

Now let $f\in\B$, choose $Q$ such that $S(f)\cup\{f(0)\}$ is contained in the disc of radius $e^Q$ centred
  at the origin, and let $F\colon \T\to \H_Q$ be a logarithmic transform of $f$. Also choose $\lambda\in \C^{\ast}$ so
small that $g\colon z\mapsto f(\lambda z)$ is of disjoint type and $0\in F(g)$. By Proposition \ref{prop_disjoint_type},  $g$ has a disjoint-type logarithmic transform $G$. It follows from results in \cite[\S 3]{lasseRigidity}, see \cite[Theorem~4.4]{mio_newCB}, that for all sufficiently large $R$, there is $Q'>0$ and a continuous map $\Theta\colon J_R(G)\rightarrow J(F)$ that is a homeomorphism to its image so that 
\begin{align}\label{eq_commute_cor}
	&\Theta \circ G =F\circ \Theta, \quad \Theta(J_R(G))\subset J_Q(F), \\
	   \notag &\Theta(I_R(G))\subset I(F) \quad \text{ and } \quad  J_{Q'}(F)\subset \Theta(J_R(G)).
\end{align}

Suppose that~\ref{item:absorbingfan} holds. Then, by Proposition~\ref{prop_absJR}, for each $R>0$ the set $J_{e^{Q'}}(f)\cup \{\infty\}$ contains an arc-smooth dendroid $\hat{X}_R\defeq X_R\cup \{\infty\}$.  For all $R\geq Q$,
\begin{equation} \label{eq_Y}
	Y=Y_R\defeq \exp^{-1}(X_R) \subset J_{Q'}(F), 
\end{equation}
see \eqref{eq_JRF}. Note that by \eqref{eq_commute_log}, $\hat{Y}\defeq Y\cup \{\infty\}$ is an absorbing arc-smooth dendroid. 

Fix $R$ large enough such that \eqref{eq_Y} and \eqref{eq_commute_cor} hold, and let $\hat{Y}$ be the corresponding arc-smooth dendroid. Let $Z\defeq \Theta^{-1}(Y)$, and note that $\hat{Z}\defeq Z\cup \{\infty\}$ is an arc-smooth dendroid contained in $J_R(G)\cup \{\infty\}$. In particular, by \eqref{eq_commute_cor} and since $F(Y)\subset Y$,
$$G(Z)=(G\circ \Theta^{-1})(Y)= (\Theta^{-1}\circ F)(Y) \subseteq \Theta^{-1}(Y)=Z. $$

We are left to show that if $z\in I(G)$, then some iterate of $z$ belongs to $Z$. Let us fix $z\in I(G)$, and let $m>0$ such that $G^m(z)\in J_R(G)$. Then, by \eqref{eq_commute_cor}, $w\defeq \Theta(G^m(z))\in I(F)$, and so there exists $p\geq 0$ such that $F^p(w)\in Y$. Again by \eqref{eq_commute_cor}, 
$$(F^p\circ\Theta \circ G^m)(z)=(\Theta \circ G^{p+m})(z) \in Y,$$
and thus $G^{p+m}(z)\in Z$. It follows that $\hat{Z}$ is an absorbing arc-smooth dendroid. By  Corollary~\ref{cor:main}, $J(G)$ is a Cantor bouquet, and hence so is $J(g)=\exp(J(G))$. We have shown that~\ref{item:absorbingfan} implies~$f\in\CB$.

Finally, suppose by way of contradiction that~\ref{item:newCB} holds but $f\notin \CB$. Fix
$R>0$ so 
  large that the function $\Theta$ is defined and satisfies~\eqref{eq_commute_cor}. By~\ref{item:newCB}, the 
    function $f$ is criniferous, and 
 hence so are $g$ and $G$. (Indeed, the pre-image of a ray tail of $F$ in $J_{Q'}(F)$ under $\Theta$ is a ray tail of
   $G$.) For every $z\in J(G)$, denote the unique ray tail connecting $z$ to $\infty$ in $J(G)$ by $\gamma_z^G$. 
   
    By assumption, $J(G)$ is not
   a Cantor bouquet. By Theorem~\ref{thm:bad}, there is a bad set $X^G\subset I(G)$ such that $\re G^n|_{X^G}\to\infty$ uniformly. 
   Let $(z,w)$ be the 
  corresponding bad pair and $(z_k)_{k=0}^{\infty}$ and $(w_k)_{k=0}^{\infty}$ the approximating sequences.
   Passing to a forward iterate, we may assume that $X^{G}\cup \gamma^G_z \subset J_R(G)$. (Recall that the image of
   a bad set is again a bad set by Proposition~\ref{prop_properties_bad}.) 
   Set $X^F \defeq \Theta(X^G)$ and $E\defeq \exp(X^F)$. Then $f^n\to\infty$ 
   uniformly on $E$. %% Our goal is to see that no collection of ray tails as in~\ref{item:newCB} can exist for $E$.  
   We would like to use the map $\Theta$ to relate the ray tails of $G$ to $F$ and $f$, but 
   Theorem~\ref{thm:bad} does not ensure that $\Rea G^n\to\infty$ uniformly on the union of the 
   $\gamma_{w_k}^G$.
   
   Instead, we will apply condition~\ref{item:newCB} to $E$. 
   Let $n_0$ be as in~\ref{item:newCB}; again we may pass to a forward iterate 
   and assume that $n_0=0$. For $\zeta\in E$, let $\gamma^f_{\zeta}$ be the ray tail of $f$ connecting $\zeta$ to $\infty$
   whose existence is ensured by~\ref{item:newCB}.
   By~\ref{item:newCB}, $f^n\to\infty$ uniformly on the union of these
   ray tails, so there is $n\geq 0$ such that $f^n(\gamma^f_{\zeta})\subset J_{e^{Q'}}(f)$ for all $\zeta\in E$.
   
 For each $x\in X^F$, there is a corresponding ray tail $\gamma^F_{F^n(x)}\subset J_{Q'}(F)$ connecting
  $F^n(x)$ to infinity, obtained from $f^n(\gamma^f_{e^{x}})$ via a branch of the logarithm. 
   Applying $\Theta^{-1}$, we obtain a ray tail of $G$ connecting
   $\Theta^{-1}(F^n(x))=G^n(\Theta^{-1}(x))$ to $\infty$ in $J_R(G)$.  
     But $G$ is of disjoint type, so the ray tail is unique and hence agrees with 
  $\gamma^G_{G^n(\Theta^{-1}(x))}$. Recall that $G^n(X^G)$ is a bad set, so $G^n(z)$ is an accumulation point of 
  the curves
  $\gamma^G_{G^n(w_k)}$. Applying $\exp\circ\Theta$, we conclude that, for 
  $\omega_k \defeq \exp(\Theta(w_k))\in E$, the ray tails
  $f^n(\gamma^f_{\omega_k})$ accumulate on $f^n(\zeta)$, where $\zeta=\exp(\Theta(z))$.
  
  On the other hand, since $z\prec w$, we have $G^n(z)\notin \gamma^G_{G^n(w)}$, and therefore
     \[ f^n(\zeta)\notin \exp(\Theta( \gamma^G_{G^n(w)})) = f^n(\gamma^f_{\omega}), \]
     where $\omega\defeq \exp(\Theta(w))$. 
     In summary, we have $\omega_k\to \omega$, but the ray tails $f^n(\gamma^f_{\omega_k})$ accumulate on 
     $f^n(\zeta)$, which is not in $f^n(\gamma^f_{\omega})$. So the 
     ray tails ending at points of $f^n(E)$ do not depend continuously on their endpoints. To complete the proof,
     we should translate this back to the original set $E$, leading to a contradiction with~\ref{item:newCB}.
  
  Recall that $\gamma^G_{z}\subset J_R(G)$. The curve $\exp(\Theta(\gamma^G_{z}))\subset J_{e^Q}(f)$ is a 
    ray tail of $f$ that connects $\zeta$ to infinity, and its image under $f^n$ is $f^n(\gamma^f_{\zeta})$. 
    We conclude that this curve is precisely $\gamma^f_{\zeta}$, which consequently lies in 
    $J_{e^Q}(f)$. Let $\tilde{Q}$ be slightly smaller than $Q$ so that 
	    the disc of radius $e^{\tilde{Q}}$ centred at the origin still contains $S(f)\cup \{f(0)\}$, 
	    and let $U$ be the connected component 
	    of \[ \{\xi\in \C\colon \lvert f^m(\xi)\rvert > e^{\tilde{Q}} \text{ for $0\leq m\leq n+1$} \} \]
    containing $\gamma^f_{\zeta}$. Then $U$ is simply connected and 
    $f^n\colon U\to f^n(U)\eqdef V$ is univalent; let $\phi$ denote its inverse.
    We have $\phi(f^n(\zeta))=\zeta$ and $\phi(f^n(\omega))=\omega$, and 
    similarly for $\omega_k$ and $\zeta_k = \exp(\Theta(z_k))$ when $k$
    is large enough.

  By assumption, 
     \[ \gamma^f_{\omega_k}\to \gamma^f_{\omega}\subset \gamma^f_{\zeta} \]
     in the Hausdorff metric on $\Ch$. Moreover, if $\xi\in \partial V$ is sufficiently large, then 
     $\lvert f(\xi)\rvert = e^{\tilde{Q}}$. (Indeed, otherwise $\xi$ is the image of a point of modulus $e^{\tilde{Q}}$ under 
     $f^m$ for some $m\in\{0,\dots,n\}$; this is not possible of $\xi$ is sufficiently large.)
     Since $f^n(\gamma^f_{\omega_k})\subset J_{e^{Q'}}(f)$, and 
     $f^n(\xi)\to\infty$ as $\xi\to\infty$ in $V$, we conclude that 
     $\gamma^f_{\omega_k}\subset U$ for sufficiently large $k$. 
     But then $\gamma^f_{\omega_k} = \phi(f^n(\gamma^f_{\omega_k}))$ accumulates on
     $\phi(f^n(\zeta)) = \zeta \notin \gamma^f_{\omega}$, which is the desired contradiction.\qed

\section{A criniferous function with non-Cantor bouquet Julia set}\label{sec_noCb}
In this section we prove Theorem \ref{thm_intro_crinNoCantorB}. For clarity of exposition, we first construct a criniferous function $F\in \BlogP$ of disjoint type such that $J(F)$ is not a Cantor bouquet. Then, we will indicate in Subsection  \ref{subsec_51} how this construction is modified to ensure that an endpoint of an arc of $J(F)$ is not accessible from $\C\setminus J(F)$, which by Theorem \ref{thm_Bishop} implies Theorem  \ref{thm_intro_crinNoCantorB}. 

The idea of the construction is as follows: we show that $J(F)$ is not a Cantor bouquet by constructing a bad pair $(a,b)$ in a fixed ray in $\R^+$. To this end, we shall describe a collection of tracts, consisting of a straight half-strip $T_0$ and 
``hooked'' tracts $\{T_n\}_{n=1}^{\infty}$, contained in the intersection of the right half-plane $\H$ and a strip of height $2\pi$. Then we define a disjoint type function $F\in \BlogP$ by considering their $2\pi i$-translates; see Figure~\ref{figure_noCB}. By controlling the thickness of the tracts, we make sure that $F$  satisfies a head-start condition, and so all components of $J(F)$ are arcs to infinity.

More precisely, set
\begin{equation}\label{eq_T0}
	T_0\defeq \{z\in \C \colon \Rea z >4 \text{ and } \vert \Ima z \vert <\pi/2\}, \quad \hat{T}_0\defeq T_0,
\end{equation}
and let $F_0$ be the conformal isomorphism from $T_0$ to $\H$ with
\begin{equation}\label{eq_preim1}
	F_0(5)=5 \quad \text{ and } \quad F_0'(5)>0. 
\end{equation}

In particular, since $T_0$ is symmetric with respect to the real axis, this implies that $F_0([5,\infty))= [5,\infty)$. Moreover, by the expanding property of $F_0$, we have $F_0(t)> t$ for $t > 5$. Let us take some interval $[a,b]\subset \R^+$ with $a\geq 6$ and $b\defeq a +2<F(a)$. We set 
\[ [a_n,b_n] \defeq [F_0^n(a),F_0^n(b)] \quad \text{ for all } \quad n\geq 0.\]
Then $a_n\to\infty$ as $n\to \infty$. 

In order to define the rest of the tracts, we require the following technical result.
\begin{prop} \label{prop_V} Let $(K_n)^\infty_{n=0}$ be a sequence of compact sets, $K_n\subset \overline{\H}\setminus \D$, such that $\lim_{n\to \infty} \min_{z\in K_n}\vert z \vert=\infty$. Let $0<\delta<\pi$. Then there is a continuous non-increasing function $\rho \colon [0,\infty)\to (0,\pi)$ with $\rho(0)\leq \delta$ such
	that the following holds. Let $V$ be the domain
	\[ V \defeq \{ x + iy \colon x > 0, \lvert y \rvert < \rho(x) \}, \]
	and let $\psi\colon V\to\H$ be the conformal isomorphism satisfying $\psi(1)=1$ and $\psi(\infty)=\infty$. Then $\diam(\psi^{-1}(K_n))\leq \delta$ for all $n\geq 0$. 
\end{prop}

\begin{proof} We start with the particular case when $K_n$ are intervals:
	\begin{claim} The proposition holds when $K_n=J_n\subset [1,\infty)$, $n\geq 0$, is a sequence of compact intervals $J_n \defeq[\alpha_n,\beta_n]$ such that $\alpha_n\to\infty$. 
	\end{claim}
	\begin{subproof}    
		First, note that we may suppose without loss of generality that $\alpha_0 = 1$, that  $\alpha_{n+1} = \beta_n$ for all $n$, and that $\beta_0 \geq 2$ and $\beta_{n+1} \geq (\beta_n)^{10}$. To see this, set $\tilde{\alpha}_0\defeq 1$ and $\tilde{\alpha}_1 \defeq \tilde{\beta}_0 \defeq 2$. For $n\geq 1$, we inductively choose $\tilde{\beta}_n = \tilde{\alpha}_{n+1} \geq (\tilde{\beta}_{n-1})^{10}$ such that
		\[ \tilde{\beta}_n \geq \max\{ \beta_j \colon \alpha_j < \tilde{\beta}_{n-1} \}. \]
		Then every one of the original intervals $J_n$ is contained in at most two intervals of the form $[\tilde{\alpha}_j , \tilde{\beta}_j]$,  so if we can prove the claim for these intervals, with $\tilde{\delta} = \delta/2$, we are done. 
		
		Let us assume in the following that the intervals $J_n$ were
		chosen with these additional properties. 
		Define $\ell_n \defeq \diam_{\H}(J_n) = \log \beta_n - \log \alpha_n$. 
		Our assumptions ensure that $(\ell_n)_{n=0}^{\infty}$ is an increasing sequence and 
		$\ell_0 = \log \beta_0 \geq \log 2\geq 1/6$.   
		
		We set 
		$\hat{\delta}\defeq \delta/3$ and $x_j\defeq 1 + \hat{\delta}\cdot j$ for
		$j\geq 0$. Let $\rho \colon [0, \infty)\to (0,\pi]$ be the function such that 
		$\rho(t) = \hat{\delta}/(2\ell_0)$ for $t\leq 1$, 
		$\rho(x_j) = \hat{\delta}/(2\ell_j)$, and such that $\rho$ is defined by linear
		interpolation between $x_j$ and $x_{j+1}$. Then $\rho$ is non-increasing,
		and $\rho(0) \leq \delta$. 
		
		Define $V$ and $\psi$ as in the statement of the proposition. Note that
		$\psi([1,\infty)) = [1,\infty)$ by symmetry. 
		For $x_j \leq x\leq x_{j+1}$, we have
		\[ \frac{\hat{\delta}}{2\ell_j} \geq \dist(x,\partial V) \geq \frac{\hat{\delta}}{4\ell_{j+1}}. \]
		By the standard estimate on the hyperbolic metric, \eqref{eq_standard_est}, we see that
		the hyperbolic length in $V$ of the interval $[x_j,x_{j+1}]$ is 
		at least $\ell_j$ and at most
		$8\ell_{j+1}$. Hence, 
		\[      \log \beta_j = \sum_{k=0}^{j} \ell_k \leq \dist_V(1,x_{j+1}) 
		\leq 8\sum_{k=1}^{j+1} \ell_k 
		\leq 8\sum_{k=0}^{j+1}\ell_k =
		8 \log \beta_{j+1} < \log\beta_{j+2} . \]
		Since $\psi(1)=1$, $\dist_V(1, x_{j+1})= \dist_{\H}(1,\psi(x_{j+1}))=\log \psi(x_{j+1})$. It follows that $\beta_j\leq \psi(x_{j+1})\leq \beta_{j+2}$. Hence,
		\[ \psi^{-1}(J_j) =\psi^{-1}([\alpha_{j}, \beta_j]) \subset [ x_{\max(j-2,0)} , x_{j+1} ], \]
		and so $ \psi^{-1}(J_j)$ has length at most $3\hat{\delta} = \delta$, as claimed. 	
	\end{subproof}
	
	Let $(K_n)^\infty_{n=0}$ be a sequence of compact sets as in the statement. For each $n\geq 0$, set $\alpha_n\defeq \min_{z\in K_n}\vert z \vert$,  $\beta_n\defeq \max_{z\in K_n}\vert z \vert$ and $J_n\defeq [\alpha_n, \beta_n]$. Let us apply the claim with $\hat{\delta}$ and $(J_n)^\infty_{n=0}$, obtaining the domain $V$ and the map $\psi$. We wish to show that the conclusion of the proposition holds if $\tilde{\delta}<\delta/5$ is chosen sufficiently small. To do so, we use considerations concerning harmonic measure, which yield the following:
	
	\begin{claim} There exists a universal constant $C>0$ such that for every $n\geq 0$, there are geodesics $\gamma^{\pm}_{\alpha_n}$ and $\gamma^{\pm}_{\beta_n}$ in $\H$ such that 
		\begin{enumerate}[(a)]
			\item $\gamma^{+}_{\alpha_n}$ connects $\alpha_n$ to $i(0, \alpha_n)\subset \partial \H$ and $\gamma^{-}_{\alpha_n}$ connects $\alpha_n$ to $i(-\alpha_n,0)\subset \partial \H$.
			\item $\gamma^{+}_{\beta_n}$ connects $\alpha_n$ to $i(0, \beta_n)\subset \partial \H$ and $\gamma^{-}_{\beta_n}$ connects $\alpha_n$ to $i(-\beta_n,0)\subset \partial \H$.
			\item $\psi^{-1}(\gamma^{\pm}_{\alpha_n})$ and $\psi^{-1}(\gamma^{\pm}_{\beta_n})$ have length at most $C\cdot \tilde{\delta}$.
		\end{enumerate}
	\end{claim}	
	\begin{subproof}
		We show how to construct $\gamma^{+}_{\alpha_n}$; the construction of the other curves is analogous. 
		Let $M\colon \H \to \D$ be the Möbius transformation that sends $0$ to $1$, $\infty$ to $-1$ and $\alpha_n$ to $0$. That is, $M(z)=\frac{\alpha_n-z}{\alpha_n+z}.$ Then, $M(i(0, \alpha_n))$ is the arc in $\partial \D$ from $1$ to $-i$, and hence has harmonic measure $1/4$ viewed from $0$. By \cite[Corollary 4.18]{pommerenke_boundary}, there is a universal constant $C>0$ and $\theta\in (0, -\pi/2)$ such that the image of the radius $e^{2\pi i\theta}\cdot [0,1)$ under $(M\circ \psi)^{-1}$ has length at most $C\cdot \dist(\psi^{-1}(\alpha_n), \partial V)\leq C\cdot \tilde{\delta}.$ Setting $\gamma_{\alpha_n}\defeq M^{-1}(e^{2\pi i\theta}\cdot [0,1))$, we have the desired conclusion.
	\end{subproof}	
	Let $A^{\pm}$ be the endpoints of $\psi^{-1}(\gamma^{\pm}_{\alpha_n})$ and $B^{\pm}$ be the endpoints of $\psi^{-1}(\gamma^{\pm}_{\beta_n})$. Note that $A^{+}$ and $B^{+}$ bound an arc $\sigma^+$ on $\partial V$ consisting of points of positive imaginary part, and similarly, $A^{-}$ and $B^{-}$ bound an arc $\sigma^-$ on $\partial V$ of points of negative imaginary part. By definition of $V$, $\Rea A^\pm\leq \Rea z \leq \Rea B^\pm$ for $z\in \sigma^\pm$. Consider the subset $V_n$ of $V$ bounded by the arcs $\sigma^\pm, \psi^{-1}(\gamma^{\pm}_{\alpha_n}), \psi^{-1}(\gamma^{\pm}_{\beta_n})$, and observe that its closure contains $\psi^{-1}(K_n)$. If $z,w\in \overline{V}_n$, then $\vert \Ima z - \Ima w \vert\leq 2\tilde{\delta}$ and 
	\[\vert \Rea z - \Rea w \vert \leq 2C\cdot \tilde{\delta} +\vert  \psi^{-1}(\beta_n)- \psi^{-1}(\alpha_n)\vert \leq (2C+1)\cdot \tilde{\delta}.\]
	Hence, $\diam V_n\leq (2C+3)\cdot \tilde{\delta}$. Letting $\tilde{\delta}\defeq \delta/(2C+3)$, the proposition follows.
	%The statement says: For every $\nu>0$, there exists a constant $C>0$ such that the following holds. Suppose $\phi\colon \D \to U$ is a conformal isomorphism and $I \subset S$ is an interval of length at least $\mu$. Then there exists a radius $\gamma$ connecting $0$ to a point of $I$ such that $\ell(\phi(\gamma))\leq C\dist(\phi(0), \partial U)$.
\end{proof}

For each $k\geq 0$ and $m\in \Z$, let
\begin{equation}\label{eq_Ik}
	I_{k,m} \defeq [ \max(5, a_k/10) , 100 a_{k+1}]+2\pi i m \quad \text{ and } \quad J_{k,m}\defeq I_{k,m}/5.
\end{equation}  

Let $V$ be the domain and $\psi\colon V\to \H$ be the conformal isomorphism provided by Proposition \ref{prop_V}, with $\delta = 1$ and the collection of intervals $\{J_{k,m} \colon k\geq 0 \text{ and } m\in \Z\}$.

We are now ready to define our tracts $\{T_n\}_{n\geq 1}$. For each $n\geq 1$, let 
\begin{align*} \hat{T_n} \defeq &\{ x + iy \colon x > a_n \text{ and } 1/(3n+1) < y/\pi-1 < 1/(3n) \}  \; \cup \\
	&\{ a_n<x<a_n +1\text{ and } 1/(3n+2)  \leq y/\pi-1 \leq 1/(3n+1) \}  \; \cup \\
	&  \{ a_n<x<b_n \text{ and } 1/(3n+3) < y/\pi-1 < 1/(3n+2) \}. 
\end{align*}
%\begin{align*} \hat{T_n} \defeq &\{ x + iy \colon x > a_n \text{ and } 1/(2n+2) < \frac{y}{\pi}-1 < 1/(2n) \} \\
%&\setminus ( \{ x + \pi(1 + 1/(2n+1))i \colon a_n+1\leq x \leq b_n \} \; \cup \\
%&  \{ b_n + y i\colon 1 + 1/(2n+2) \leq y/\pi \leq 1/(2n+1) \}). 
%\end{align*}
Let $\Sigma$ denote the strip $\Sigma \defeq \{ x + iy \colon \lvert y\rvert < \pi \}$ and let $\phi_n\colon \Sigma\to \hat{T}_n$ be the conformal isomorphism with $\phi_n(-\infty)=b_n+\pi(1+\frac{1}{3n+3})i$, $\phi_n(\infty)=\infty$ and $\phi_n(1)=\zeta_n$, where $\zeta_n$ is a point in $\hat{T}_n$ at real part $b_n-1$ and imaginary part less than $\pi(1+1/(3n+2))$ that sits in the geodesic that joins $\phi_n(-\infty)$ to $\phi_n(\infty)$; see Figure~\ref{figure_noCB}. 

We define 
\begin{equation}\label{eq_Tn_Fn}
	T_n \defeq \phi_n(V) \quad \text{ and } \quad F_n\colon T_n\to \H; \quad F_n(z) \defeq 5\cdot \psi(\phi_n^{-1}(z)).
\end{equation}
Note that by the choice of $V$, the closures of all tracts $\{T_n\}_{n\geq 0}$ are pairwise disjoint and
\begin{equation}\label{eq_preim2}
	F_n(\zeta_n)= 5\cdot\psi(1)=5.
\end{equation}

\begin{figure}[htb]
\begingroup%
\makeatletter%
\providecommand\color[2][]{%
	\errmessage{(Inkscape) Color is used for the text in Inkscape, but the package 'color.sty' is not loaded}%
	\renewcommand\color[2][]{}%
}%
\providecommand\transparent[1]{%
	\errmessage{(Inkscape) Transparency is used (non-zero) for the text in Inkscape, but the package 'transparent.sty' is not loaded}%
	\renewcommand\transparent[1]{}%
}%
\providecommand\rotatebox[2]{#2}%
\newcommand*\fsize{\dimexpr\f@size pt\relax}%
\newcommand*\lineheight[1]{\fontsize{\fsize}{#1\fsize}\selectfont}%
\ifx\svgwidth\undefined%
\setlength{\unitlength}{394.01574803bp}%
\ifx\svgscale\undefined%
\relax%
\else%
\setlength{\unitlength}{\unitlength * \real{\svgscale}}%
\fi%
\else%
\setlength{\unitlength}{\svgwidth}%
\fi%
\global\let\svgwidth\undefined%
\global\let\svgscale\undefined%
\makeatother%
\begin{picture}(1,0.49640288)%
	\lineheight{1}%
	\setlength\tabcolsep{0pt}%
	\put(0.02331849,0.2069055){\color[rgb]{0,0,0}\makebox(0,0)[lt]{\lineheight{1.25}\smash{\begin{tabular}[t]{l}$T_0$\end{tabular}}}}%
	\put(0.04094679,0.07143817){\color[rgb]{0,0,0}\makebox(0,0)[lt]{\lineheight{1.25}\smash{\begin{tabular}[t]{l}$a$\end{tabular}}}}%
	\put(0.16208333,0.07239958){\color[rgb]{0,0,0}\makebox(0,0)[lt]{\lineheight{1.25}\smash{\begin{tabular}[t]{l}$b$\end{tabular}}}}%
	\put(0.91824044,0.47547945){\color[rgb]{0,0,0}\makebox(0,0)[lt]{\lineheight{1.25}\smash{\begin{tabular}[t]{l}$\hat{T}_n$\end{tabular}}}}%
	\put(0.57957248,0.36615838){\color[rgb]{0,0,0}\makebox(0,0)[lt]{\lineheight{1.25}\smash{\begin{tabular}[t]{l}$T_n$\end{tabular}}}}%
	\put(0.20051995,0.35482423){\color[rgb]{0,0,0}\makebox(0,0)[lt]{\lineheight{1.25}\smash{\begin{tabular}[t]{l}\fontsize{9pt}{1em}$\pi(1+\frac{1}{3n+3})$\end{tabular}}}}%
	\put(0.4829252,0.33506097){\color[rgb]{0,0,0}\makebox(0,0)[lt]{\lineheight{1.25}\smash{\begin{tabular}[t]{l}$\zeta_n$\end{tabular}}}}%
	\put(0.84653374,0.22591543){\color[rgb]{0,0,0}\makebox(0,0)[lt]{\lineheight{1.25}\smash{\begin{tabular}[t]{l}$\zeta_{n+1}$\end{tabular}}}}%
	\put(0.22113484,0.46076637){\color[rgb]{0,0,0}\makebox(0,0)[lt]{\lineheight{1.25}\smash{\begin{tabular}[t]{l}\fontsize{9pt}{1em}$\pi(1+\frac{1}{3n})$\end{tabular}}}}%
	\put(0,0){\includegraphics[width=\unitlength,page=1]{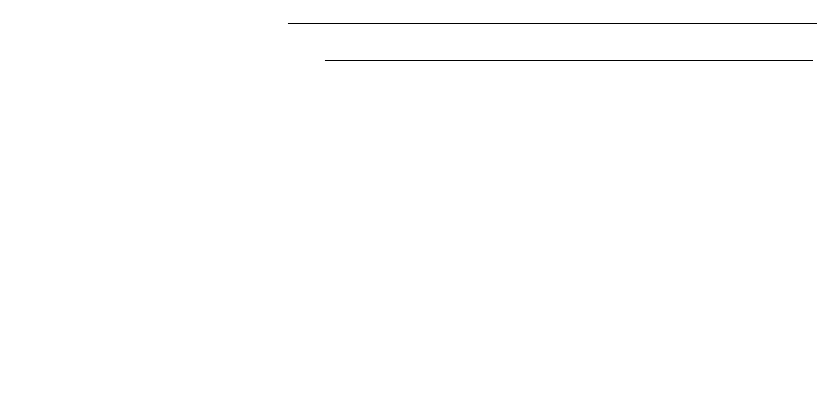}}%
	\put(0.92978224,0.36738433){\color[rgb]{0,0,0}\makebox(0,0)[lt]{\lineheight{1.25}\smash{\begin{tabular}[t]{l}$\hat{T}_{n+1}$\end{tabular}}}}%
	\put(0.94025108,0.2598378){\color[rgb]{0,0,0}\makebox(0,0)[lt]{\lineheight{1.25}\smash{\begin{tabular}[t]{l}$T_{n+1}$\end{tabular}}}}%
	\put(0.5259,0.24){\color[rgb]{0,0,0}\makebox(0,0)[lt]{\lineheight{1.25}\smash{\begin{tabular}[t]{l}\fontsize{9pt}{1em}$\pi(1+\frac{1}{3(n+2)})$\end{tabular}}}}%
	\put(0,0){\includegraphics[width=\unitlength,page=2]{NewFigure3NoCB.pdf}}%
	\put(0.33405183,0.07275786){\color[rgb]{0,0,0}\makebox(0,0)[lt]{\lineheight{1.25}\smash{\begin{tabular}[t]{l}$a_n$\end{tabular}}}}%
	\put(0.5465976,0.07250801){\color[rgb]{0,0,0}\makebox(0,0)[lt]{\lineheight{1.25}\smash{\begin{tabular}[t]{l}$b_n$\end{tabular}}}}%
	\put(0.66737432,0.07498542){\color[rgb]{0,0,0}\makebox(0,0)[lt]{\lineheight{1.25}\smash{\begin{tabular}[t]{l}$a_{n+1}$\end{tabular}}}}%
	\put(0.91286287,0.07773618){\color[rgb]{0,0,0}\makebox(0,0)[lt]{\lineheight{1.25}\smash{\begin{tabular}[t]{l}$b_{n+1}$\end{tabular}}}}%
	\put(0,0){\includegraphics[width=\unitlength,page=3]{NewFigure3NoCB.pdf}}%
\end{picture}%
\endgroup%
\caption{The tracts $\{T_n\}_{n\geq 0}$ as defined in \eqref{eq_T0} and \eqref{eq_Tn_Fn}.}
	\label{figure_noCB}
\end{figure}

The following proposition will aid us when showing in Proposition \ref{prop_HSC} that, in a rough sense, preimages of the ``hooks'' of the tracts have uniformly bounded diameter, and so the components of $J(F)$ will be arcs. We formalize this using the intervals $\{I_{k,m}\}$ defined in \eqref{eq_Ik}.
\begin{prop} \label{prop_preimIk} For each $n\geq 1$, $k\geq 0$ and $m\in \Z$,  $F_n^{-1}(I_{k,m})$ has diameter at most~$1/2$.
\end{prop}
\begin{proof}
	The hyperbolic density of $\Sigma$ is given by
	\[ \rho_{\Sigma}(z) = \frac{1}{2\cos(\im z/2)}, \]
	see \cite[Example 7.9]{beardon_minda}. Since $\hat{T}_n$ is contained in a horizontal strip of height $\pi/3$, by  monotonicity of the hyperbolic metric, the density of the hyperbolic metric in $\hat{T}_n$ satisfies $\rho_{\hat{T}_n}(z) \geq 2$ for all $z\in \hat{T}_n$. Using Pick's Theorem, \cite[Theorem 6.4]{beardon_minda}, it follows that for all $n\geq 1$ and all $z\in \Sigma$ with $\lvert \im z\rvert \leq \pi/2$, and so for all $z\in V$,
	\[ \lvert \phi_n'(z)\rvert = \frac{\rho_{\Sigma}(z)}{\rho_{\hat{T}_n}(\phi_n(z))} 
	\leq \frac{1}{4 \cos(\pi/4)} = \frac{1}{2\sqrt{2}} \leq \frac{1}{2}. \] 
	By this and Proposition \ref{prop_V}, the claim follows.
\end{proof}

Let $\T$ be the union of all $T_n$, $n\geq 0$, together with their $2\pi i $-translates, and define $F\colon \T\to \H$ such that if $z\in (T_n+2\pi im)$ for some $m\in \Z$ and $n\geq 0$, then  $F(z)\defeq F_n(z+2\pi i m)$. Since $\overline{\T}\subset \H$, $F$ is by definition of disjoint type, and $F\in \BlogP$. Moreover, by our choice of tracts, for any $z\in \T$ such that $F(z)\in \T$,
\begin{equation}\label{eq_expansion2}
	\vert F'(z) \vert \geq \frac{\Rea F(z)}{2}\geq 2;
\end{equation}
see \cite[Theorem 1.1]{lasse_constant}.

%\begin{prop} For each $T\in \T$, there exists a constant $K$ such that 
%\begin{equation*}
%\diam(F^{-1}_T(\{\zeta \in \H \colon \vert \zeta\vert=\rho\}))\leq K
%\end{equation*}
%for all sufficiently large $\rho>0$.
%\end{prop}

Recall that, in Definition~\ref{UHSCentire}, we defined what it means for a disjoint-type entire function to satisfy a \emph{uniform} head-start condition
 on its Julia set. This condition implies that~$J(f)$ is a Cantor bouquet (see Proposition~\ref{prop_UHSC_CB}). To show that our function $F$, whose Julia set turns out \emph{not} to be a Cantor bouquet,
   is criniferous, we shall use 
  a \emph{non-uniform} notion of head-start conditions, as was previously defined in 
  \cite[Definition 4.1]{rrrs}. 

\begin{defn}[Head-start condition for addresses] \label{defn_HSC_addr} Let $F \in \BlogP$, let $T, T'$ be tracts of $F$ and let $\phi \colon [0,\infty)\to [0,\infty)$ be a monotonically increasing  upper semicontinuous function with $\phi(x)>x$. We say that the pair $(T,T')$ satisfies the \textit{head-start condition for} $\phi$, if, for all $z,w \in \overline{T}$ with $F(z),F(w)\in \overline{T'}$, 
\begin{equation}\label{eq_hs1}
\Rea(w)> \phi(\Rea(z)) \text{ implies }
\Rea(F(w)) > \phi( \Rea(F(z)). 
\end{equation}
An external address $\s$ satisfies the head-start condition for $\phi$ if all consecutive pairs of tracts $(T_k, T_{k+1})$ satisfy the head-start condition for $\phi$, and if for all distinct $z, w \in J_{\s}$ there is $n \geq 0$ such that either $\Rea( F^n(w) ) > \phi( \Rea( F^n(z)))$ or $\Rea( F^n(z) ) > \phi( \Rea(F^n(w)))$. 
\end{defn}

\begin{obs} In fact, \cite[Definition 4.1]{rrrs} required the function $\phi$ to be continuous, rather than upper semicontinuous. However, 
  the results of that paper also hold with the above weaker assumption. Indeed, the assumption of continuity is used only in 
  \cite[Proposition 4.4]{rrrs}, to show that the unbounded connected component of 
  $J_{\s}$ is an arc to infinity whenever $\s$ satisfies a head-start condition. The same conclusion holds 
  when $\phi$ is upper semicontinuous (see the argument used in the proof of Proposition~\ref{prop:bouquetordering} below). 
\end{obs}

\begin{obs} Observe that~\eqref{eq_hs1} is required to hold for any pair of points whose images are in the same tract, in contrast to Definition~\ref{UHSCentire}, where we
 asked only that the conditions holds for points in the Julia set. 
% In~\cite{rrrs}, the function $F$ was said to satisfy a \textit{head-start condition} if every external address of
% $F$ satisfies some head-start condition; this condition is said to be uniform if $\phi$ can be chosen independently of $\s$. 
% Note that here~\eqref{eq_hs1} is required to hold for any pair of points whose images are in the same tract, in contrast to Definitions~\ref{UHSCentire} and~\ref{def_UHSClog},
% where such a condition was only required for points in the Julia set. 
\end{obs}

\begin{thm}[{\cite[Theorem 4.2]{rrrs}}] \label{thm_HS_then_crin}If every external address of $F\in \BlogP$ satisfies a head-start condition, then $F$ is criniferous.
\end{thm}

 To establish a head-start condition for $F$, 
   we use the following version of the Ahlfors distortion theorem, see for example \cite[Corollary to Theorem~4.8]{Ahlfors_conf_inv73}.%, and compare to \cite[Theorem 3.4]{lasse_full}.

\begin{thm}\label{thm_ahlfors}
	Let $V \subset \mathbb{C}$ be a simply connected domain, and
	let $\gamma_t, \gamma_{t'} \subset V$ be two maximal vertical line segments at real parts $t<t'$ respectively. Set $\Sigma \defeq \{ x + iy \in \C \colon |y| < \pi \},$	and let $\phi \colon V \rightarrow \Sigma$ be a conformal isomorphism such that $\phi(\gamma_t)$
	and $\phi(\gamma_{t'})$ both connect the upper and lower boundaries of the strip $\Sigma$, and such that $\phi(\gamma_t)$ separates $\phi(\gamma_{t'})$ from $-\infty$ in $\Sigma$. For $t\leq s \leq t'$, let $\theta(s)$ denote the shortest length of a vertical line segment at real part $s$ that separates $\gamma_t$ from $\gamma_{t'}$ in $V$. If $\int_{t}^{t'}ds/\theta(s) \geq 1/2$, then 
	\begin{equation}
		\min_{z \in \gamma_{t'}} \Rea \phi(z)- \max_{w \in \gamma_{t}} \Rea \phi(w)\geq 2\pi\int_{t}^{t'}\frac{ds}{\theta(s)}-2\ln 32.
	\end{equation}
\end{thm}

\begin{prop} \label{prop_HSC}The map $F$ satisfies a head-start condition. 
\end{prop}
\begin{proof}
	We must show that each admissible external address of $F$ satisfies a head-start condition for some (not necessarily the same) map $\phi$.
	
	First, suppose that $\s=s_0s_1s_2\ldots$ is an admissible external address such that $s_j$ is a translate of $T_0$ for all but finitely many $j$. Then, the restriction $\tilde{F}$ of $F$ to 
	$$\bigcup_{j=0}^\infty \bigcup_{m=-\infty}^\infty s_j +2\pi i m$$
	has bounded slope and uniformly bounded wiggling in the sense of \cite[Definition~5.3]{rrrs}. By \cite[Proposition~5.4]{rrrs}, $\tilde{F}$ satisfies a head-start condition, and consequently, so does $\s$.
	
	Suppose now that $\s=s_0s_1s_2\ldots $ is an admissible external address such that for infinitely many $k\geq 0$, $s_k\neq T_0+2\pi i m$ for all $m\in \Z$. Let us define $\phi \colon [0,\infty)\to [0,\infty)$ as
	\begin{equation}
		\label{eq_phi_1}
		\phi(x) \defeq y_{n_x} \quad \text{ for } \quad n_x\defeq \max\{n \geq 0 \colon x\geq \alpha_n\}, 
	\end{equation}
	where $\{y_n\}_{n=0}^{\infty}$ and $\{\alpha_n\}_{n=0}^{\infty}$ are two monotonically increasing sequences. In this case, we choose
	\begin{equation}\label{eq_param1}
		\alpha_n\defeq a_n/2 \quad \text{ and } \quad y_n\defeq 100 a_{n+1}, \quad \text{ for each } n\geq 0.
	\end{equation}
	We aim to show that $\s$ satisfies the head-start condition for $\phi$. 
	
	Let us fix $z,w\in \overline{s_k}$ with $F(z),F(w)\in s_{k+1}$ for some $k\geq 0$.  In particular, $s_k=T_{j}+2\pi i {m_1}$ and $s_{k+1}=T_{\ell}+2\pi i {m_2}$ for some $j,\ell \geq 0$ and $m_1,m_2\in \Z$.  Suppose that $\Rea(w) > \phi( \Rea(z))$. Let
	\[x\defeq \Rea(z) \quad \text{ and } \quad y\ \defeq \Rea(w).\]
	In particular, 
	\begin{equation}\label{eq_bounds_xy}
		x<\alpha_{n_x+1}, \quad \quad y\geq y_{n_x} \quad \text{ and } \quad n_x\geq j.
	\end{equation}

		\begin{claim}[Claim 1] If $s_k= T_0+2\pi i m$ for some $m\in \Z$, then $\Rea F(z)<\alpha_{n_x+2}$.
	\end{claim}
	\begin{subproof}
		By \eqref{eq_standard_est} and using that $F(5)=5$, we have
		\begin{equation*} 
			\log\frac{F(x)}{5}=\dist_\H(5, F(x))=\dist_T(5,x)\leq 2(x-5), 
		\end{equation*}
		and so 
		\begin{equation} \label{eq_upperFx}
			F(x) \leq 5 \exp(2(x-5)). 
		\end{equation}
		Observe that $F$ restricted to $s_k$ is the map $\zeta\mapsto 5\cosh(\zeta-4)/\cosh(1)$, which sends vertical lines to ellipses with real major axis, and so, we have that $\Rea F(z)<F(x)$. Consequently,
		\begin{equation} 
			\Rea F(z)<F(x)<F(a_{n_x+1}/2)<a_{n_x+2}/2=\alpha_{n_x+2},
		\end{equation}
		as needed. \qedhere 
	\end{subproof}	
	
	\begin{claim}[Claim 2]
		We have that $\Rea(F(w))\geq \exp(5a_{n_x+1})\Rea (F(z))\geq y_{n_x+1}.$
	\end{claim}
	\begin{subproof}
		Observe that any vertical segment that separates $z$ from $w$ in $\hat{T}_j$ has length at most $\pi$, and that $y\geq b_j$. Let $\Sigma = \{\zeta\in\C\colon \lvert \im\zeta\rvert \leq \pi \}$, and let $\phi_j^{-1}\colon \hat{T_j}\to \Sigma$, where the map $\phi_j$ is defined after \eqref{eq_Ik} if $j\neq 0$, and $\phi^{-1}_0\defeq 2\cdot\id$. By the geometry of $\hat{T_j}$ and the definition of $\phi_j$, there exist vertical segments $\gamma_x, \gamma_{y}$ at real parts $x+1$ and $y-1$ respectively such that $\gamma_x$ separates $z$ from all points at real parts at least $a_{n_{x+1}}$,  $\gamma_y$ separates $w$ from all points at real parts at most $a_{n_{x+1}}/2$ and $\phi_j^{-1}(\gamma_x)$ separates $\phi_j^{-1}(\gamma_{y})$ from $-\infty$ in $\Sigma$.  Let
		%\[\alpha(F_{T'_j}^{-1}(5))=\exp^{-1}(5).\]
		\begin{equation*}
			\beta \defeq \min_{z \in \gamma_{y}} \Rea \phi_j^{-1}(z) \quad \text{ and } \quad \alpha\defeq \max_{w \in \gamma_{x}} \Rea \phi_j^{-1}(w),
		\end{equation*}
		and note that  $\beta,\alpha>1$. Then,  since any vertical line segment in $\hat{T}_j$ has length at most $\pi$, $a_{n_x+1}\geq 6$ and by \eqref{eq_bounds_xy}, $y-x-2>100a_{n_x+1}-a_{n_x+1}/2-2\geq 99 a_{n_x+1}$, by Theorem~\ref{thm_ahlfors},
		\begin{equation}\label{eq_Ahl1}
			\beta-\alpha\geq2\cdot99a_{n_x+1}-2\ln32 \geq2\cdot 99a_{n_x+1}-2a_{n_x+1}\geq  196 a_{n_x+1}.
		\end{equation}
		If $j=0$, we have that 
		\begin{equation}\label{eq_j0}
			\vert F(w)\vert=\vert\exp(\phi^{-1}_0(w)/2)\vert \geq \exp(50 a_{n_x+1})\vert F(z)\vert.
		\end{equation}
		If $j\neq 0$, let $\ell_\beta$ and $\ell_\alpha$ be the vertical segments in $V$ at real parts $\beta$ and $\alpha$ respectively, consider $2\cdot\Log (5\cdot\psi)\colon V\to \Sigma$ and recall that by definition of $V$, the length of $\ell_\beta$ and $\ell_\alpha$ is bounded by a constant $\delta<\pi$. Let
		\begin{equation*}
			M^+ \defeq \min_{z \in \ell_\beta} \Rea(\Log (5\cdot\psi)(z)) \quad \text{ and } \quad M^- \defeq \max_{w \in \ell_\alpha} \Rea(\Log (5\cdot\psi)(w)).
		\end{equation*}
		Then, by \eqref{eq_Ahl1} and Theorem \ref{thm_ahlfors},
		\begin{equation*}
			M^+-M^-\geq98 a_{n_x+1}-\ln32 \geq  98a_{n_x+1}-a_{n_x+1}\geq  50 a_{n_x+1}.
		\end{equation*}
		Hence, since $\exp\circ \Log (5\cdot\psi) \circ \phi^{-1}_j \vert_{T_j} \equiv F \vert_{T_j}, $ we have that
		\begin{equation}\label{eq_jneq0}
			\vert F(w)\vert \geq \exp(50 a_{n_x+1})\vert F(z)\vert.
		\end{equation}
		
		Recall that since $F(z),F(w) \in s_{j+1}$, the imaginary parts of $F(z)$ and $F(w)$ differ by at most $\pi$, and so for any $j$, by \eqref{eq_j0} and \eqref{eq_jneq0},
		\begin{equation*}\label{eq_real_parts}
			\begin{split}
				\Rea (F(w))&>\vert F(w)\vert-\vert \Ima(F(z))\vert-\pi \\
				&>(\exp(50 a_{n_x+1})-1)\vert F(z)\vert -\pi \\
				&\geq (\exp(50 a_{n_x+1})-1)\Rea(F(z))-\pi.
			\end{split}
		\end{equation*}
		Recall that the set of tracts of $F$ satisfies $\overline{\T}\subset \H_4$. Hence, since $F(z)\in s_{j+1}$, $\Rea F(z) >4$. Moreover, by assumption, $a_n\geq 6$ for all $n\geq 0$. Thus, we have that
		\begin{equation*}\label{eq_real_parts2}
			\exp(50 a_{n_x+1}-1)\Rea (F(z))-\pi>\exp(5a_{n_x+1})\Rea (F(z)).
		\end{equation*}
		Using \eqref{eq_upperFx},
		\begin{equation*} \label{eq_phiReaFz}
			\exp(5a_{n_x+1})\geq \exp(3\cdot 6)\exp(2a_{n_x+1})\geq 100 F(a_{n_x+1}) \geq 100 a_{n_x+2}=y_{n_x+1}.\qedhere
		\end{equation*}
	\end{subproof}
	
It follows from Claims 1 and 2 that if $s_j= T_0+2\pi i m$ for some $m\in \Z$, by definition of $\phi$,
\begin{equation}
		\Rea(F(w))>y_{n_x+1}=\phi(\alpha_{n_x+1})\geq \phi(\Rea F(z)), 
	\end{equation} 
	as we wanted to show.
	
	We are left to study the case $s_j\neq T_0+2\pi i m$ for all $m\in \Z$. First, note that by \eqref{eq_bounds_xy}, $\Rea(w)-\Rea(z)>y_{n_x}-\alpha_{n_x+1}>10$, and that Claim 2 implies that $\Rea(F(w)) > \Rea (F(z))$. Moreover, by definition of $\phi$ and the intervals $\{I_{k,m}\}$ in \eqref{eq_Ik}, $\Rea(F(z))\in I_{k,m}$ for some $k\geq 0$, $m\in \Z$, and if $\Rea(F(w))\in I_{k,m}$, Proposition \ref{prop_preimIk} would lead to a contradiction. Thus, we have that $\Rea(F(w)) > \phi( \Rea (F(z)))$.
	
	Next, consider any two distinct $z,w \in J_{\s}$. We want to show that either $\Rea( F^n(w) ) > \phi( \Rea( F^n(z)))$ or $\Rea( F^n(z) ) > \phi( \Rea(F^n(w)))$ for all $n\geq 0$ sufficiently large. Suppose, by contradiction, that this is not the case, that is, that for all $n\geq 0$, 
	\begin{equation}\label{eq_hsc}
		\Rea(F^n(z)) \leq  \phi(\Rea(F^n(w))) \quad \text{ and } \quad \Rea(F^n(w)) \leq  \phi(\Rea(F^n(z))).
	\end{equation}
	In other words, either $\Rea(F^n(w)) < \Rea(F^n(z)) \leq  \phi(\Rea(F^n(w)))$ or $\Rea(F^n(z)) < \Rea(F^n(w)) \leq  \phi(\Rea(F^n(z)))$, and so, by definition of $\phi$ and the choice of the intervals $I_{k,m}$ in \eqref{eq_Ik}, $\Rea(F^n(w)), \Rea(F^n(z))\in I_{k,m}$ for some $k\geq 0$ and $m\in \Z$. In particular, by our assumption on $\s$, \eqref{eq_hsc} holds for a subsequence of $n_j\to \infty$ such that $s_{n_j}\neq T_0+2\pi i m$ for all $m\in \Z$. Moreover, since both $F^{n_j}(w)$ and $F^{n_j}(z)$ belong to the same tract, their imaginary parts differ at most by $1$. But note that then, by Proposition \ref{prop_preimIk} and \eqref{eq_expansion2}, $\vert w-z \vert<3/2^{n_j}$ for each $n_j$, contradicting $z\neq w$.
\end{proof}

\begin{remark}We have shown in the proof of the previous proposition that all external addresses $\s=s_0s_1s_2\ldots $ such that for infinitely many $j\geq 0$, $s_j\neq T_0+2\pi i m$ for all $m\in \Z$ satisfy the head-start condition for the same non-linear function $\phi$, as defined in \eqref{eq_phi_1}. In particular, the restriction $\tilde{F}$ of $F$ to all tracts that are $2\pi i \Z$ translates of $\{T_n\}_{n\geq 1}$ is an example of a function in $\BlogP$ whose Julia set is a Cantor bouquet (by \cite[Corollary~6.3]{lasseBrushing}, or 
Proposition~\ref{prop_UHSC_CB} below), but does not satisfy a linear uniform head-start condition in the sense of \cite[Section~5]{rrrs}.	
\end{remark}

\begin{cor} \label{cor_crin} The function $F$ is criniferous and each connected component of $J(F)$ is a dynamic ray together with its endpoint. 
\end{cor}
\begin{proof}
	By Proposition \ref{prop_HSC}, that $F$ is criniferous follows from Theorem \ref{thm_HS_then_crin}. Since $F$ is of disjoint type, the rest of the statement is a consequence of Proposition~\ref{prop:disjoint_crin}.
\end{proof}

We have shown that $J(F)$ is a collection of unbounded disjoint curves, and so we can endow $\hat{J}(F)$  with the partial order defined in \eqref{eq_order}; namely, if $z,w\in J(F)$,
\begin{equation*}
	z\preceq w  \quad \text{ if and only if } \quad w\in [z,\infty],
\end{equation*}
where $[z,\infty]$ is the unique arc in $\hat{J}(F)$ connecting $z$ and $\infty$.

\begin{prop} \label{prop_noCB} $J(F)$ is not a Cantor bouquet.
\end{prop}
\begin{proof}
	We will show that $J(F)$ is not a Cantor bouquet by showing that $(a,b)$ constitutes a bad pair in the sense of Definition \ref{defn_badpair}. For each $k\geq 0$, consider the external address 
	$$\s_k\defeq \underbrace{ T_0\ldots T_0}_{ k \text{ times }} T_k T_0 T_0 T_0\ldots. $$
	Since, by definition of $F$,  $F([5,\infty))=[5,\infty)$, we have that $a\prec b$ and that for each $k$,
	$$J_{\s_k}=F_{T_0}^{-k}(F^{-1}_{T_k}([5, \infty))).$$ 
	Since $F(5)=5$, see \eqref{eq_preim1}, and by the geometry of tracts,  and in particular \eqref{eq_preim2}, there exist $z_k,w_k\in F^{k}(J_{\s_k})$ such that $w_k\prec z_k$,
	\[\vert w_k-b_k\vert<1/k \quad \text{ and }\quad \vert z_k-a_k\vert<1/k.\]
	Let $\zeta_k\defeq F_{T_0}^{-k}(z_k)$ and $\omega_k\defeq F_{T_0}^{-k}(w_k)$. Then, using \eqref{eq_expansion2}, we have that $\zeta_k\to a$, $\omega_k\to b$ as $k\to \infty$, and that $\omega_k\prec \zeta_k \in J_{\s_k}$ for all $k\geq 0$, see Proposition \ref{prop_properties_bad}. Hence, by constructing approximating sequences, we have shown that $(a,b)$ is a bad pair. The result then follows from Corollary \ref{cor:badbouquets}.
\end{proof}

\subsection{Variation: non-accessible endpoint}\label{subsec_51}\hfill

Note that in the previous construction, the hair that we show to contain a bad pair is invariant. By \cite[Theorem B]{BarKarp_codingtrees}, see also \cite[Proposition 3.10]{lasse_arclike}, if an external address is \textit{bounded}, that is, contains only finitely many different tracts, then its endpoint is accessible from $\C\setminus J(F)$. To prove Theorem \ref{thm_intro_crinNoCantorB}, we must therefore modify the construction. We do so by adding tracts, in such a way that the endpoint of the curve associated to a non-bounded external address has a non-accessible endpoint.

Namely, let $\{T_n\}_{n\geq 0}$ be the collection of tracts described in  \eqref{eq_T0} and \eqref{eq_Tn_Fn}, and let $F_n\colon T_n\to \H$ be the corresponding isomorphisms specified in  \eqref{eq_preim1} and \eqref{eq_preim2}. For each $n\geq 0$ consider the half-strip of width $1/(6n(n+1))$ ``surrounded'' by $T_n$; that is, 
\begin{equation*}
	S_{n}\defeq \{x+iy \colon x>a_n +1 \text{ and } 1/(3n+2)  < y/\pi-1 < 1/(3n+1)\},
\end{equation*}
see Figure \ref{figure_tractsS}, and let $G_{n}$ be the conformal isomorphism from $S_{n}$ to $\H$ such that for $p_n\defeq a_n+1+ (2n+1)i/(6n(n+1))$, 
\begin{equation}
	G_{n}(p_n)=p_{n+1} \quad \text{ and } \quad G_{n}'(p_n)>0. 
\end{equation}

\begin{figure}[htb]
	\begingroup%
	\makeatletter%
	\providecommand\color[2][]{%
		\errmessage{(Inkscape) Color is used for the text in Inkscape, but the package 'color.sty' is not loaded}%
		\renewcommand\color[2][]{}%
	}%
	\providecommand\transparent[1]{%
		\errmessage{(Inkscape) Transparency is used (non-zero) for the text in Inkscape, but the package 'transparent.sty' is not loaded}%
		\renewcommand\transparent[1]{}%
	}%
	\providecommand\rotatebox[2]{#2}%
	\newcommand*\fsize{\dimexpr\f@size pt\relax}%
	\newcommand*\lineheight[1]{\fontsize{\fsize}{#1\fsize}\selectfont}%
	\ifx\svgwidth\undefined%
	\setlength{\unitlength}{394.01574803bp}%
	\ifx\svgscale\undefined%
	\relax%
	\else%
	\setlength{\unitlength}{\unitlength * \real{\svgscale}}%
	\fi%
	\else%
	\setlength{\unitlength}{\svgwidth}%
	\fi%
	\global\let\svgwidth\undefined%
	\global\let\svgscale\undefined%
	\makeatother%
	\begin{picture}(1,0.47482014)%
		\lineheight{1}%
		\setlength\tabcolsep{0pt}%
		\put(0.04934914,0.06014485){\color[rgb]{0,0,0}\makebox(0,0)[lt]{\lineheight{1.25}\smash{\begin{tabular}[t]{l}$a$\end{tabular}}}}%
		\put(0.1704857,0.06110626){\color[rgb]{0,0,0}\makebox(0,0)[lt]{\lineheight{1.25}\smash{\begin{tabular}[t]{l}$b$\end{tabular}}}}%
		\put(0.20892231,0.34353092){\color[rgb]{0,0,0}\makebox(0,0)[lt]{\lineheight{1.25}\smash{\begin{tabular}[t]{l}\fontsize{9pt}{1em}$\pi(1+\frac{1}{3n+3})$\end{tabular}}}}%
		\put(0.92637261,0.3925475){\color[rgb]{0,0,0}\makebox(0,0)[lt]{\lineheight{1.25}\smash{\begin{tabular}[t]{l}\fontsize{9pt}{1em}$S_{n}$\end{tabular}}}}%
		\put(0.92764074,0.28273105){\color[rgb]{0,0,0}\makebox(0,0)[lt]{\lineheight{1.25}\smash{\begin{tabular}[t]{l}\fontsize{9pt}{1em}$S_{n+1}$\end{tabular}}}}%
		\put(0.22953721,0.44947306){\color[rgb]{0,0,0}\makebox(0,0)[lt]{\lineheight{1.25}\smash{\begin{tabular}[t]{l}\fontsize{9pt}{1em}$\pi(1+\frac{1}{3n})$\end{tabular}}}}%
		\put(0,0){\includegraphics[width=\unitlength,page=1]{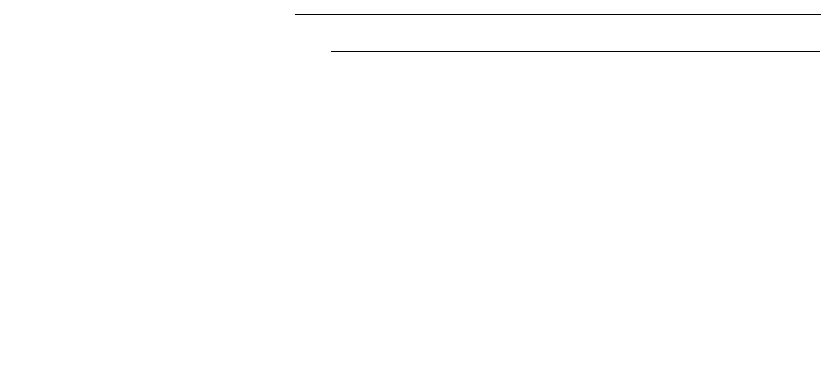}}%
		\put(0.53727177,0.23007006){\color[rgb]{0,0,0}\makebox(0,0)[lt]{\lineheight{1.25}\smash{\begin{tabular}[t]{l}\fontsize{9pt}{1em}$\pi(1+\frac{1}{3(n+2)})$\end{tabular}}}}%
		\put(0,0){\includegraphics[width=\unitlength,page=2]{TractsS.pdf}}%
		\put(0.34245418,0.06146438){\color[rgb]{0,0,0}\makebox(0,0)[lt]{\lineheight{1.25}\smash{\begin{tabular}[t]{l}$a_n$\end{tabular}}}}%
		\put(0.55119296,0.06121453){\color[rgb]{0,0,0}\makebox(0,0)[lt]{\lineheight{1.25}\smash{\begin{tabular}[t]{l}$b_n$\end{tabular}}}}%
		\put(0.67577664,0.06369194){\color[rgb]{0,0,0}\makebox(0,0)[lt]{\lineheight{1.25}\smash{\begin{tabular}[t]{l}$a_{n+1}$\end{tabular}}}}%
		\put(0.9212652,0.0664427){\color[rgb]{0,0,0}\makebox(0,0)[lt]{\lineheight{1.25}\smash{\begin{tabular}[t]{l}$b_{n+1}$\end{tabular}}}}%
		\put(0,0){\includegraphics[width=\unitlength,page=3]{TractsS.pdf}}%
		\put(0.94108033,0.24012196){\color[rgb]{0,0,0}\makebox(0,0)[lt]{\lineheight{1.25}\smash{\begin{tabular}[t]{l}\fontsize{9pt}{1em}$T_{n+1}$\end{tabular}}}}%
		\put(0.58465785,0.35564384){\color[rgb]{0,0,0}\makebox(0,0)[lt]{\lineheight{1.25}\smash{\begin{tabular}[t]{l}\fontsize{9pt}{1em}$T_n$\end{tabular}}}}%
		\put(0.00485849,0.18784052){\color[rgb]{0,0,0}\makebox(0,0)[lt]{\lineheight{1.25}\smash{\begin{tabular}[t]{l}\fontsize{9pt}{1em}$T_0$\end{tabular}}}}%
		\put(0.85951974,0.22037841){\color[rgb]{0,0,0}\makebox(0,0)[lt]{\lineheight{1.25}\smash{\begin{tabular}[t]{l}\fontsize{9pt}{1em}$\zeta_{n+1}$\end{tabular}}}}%
		\put(0.49684966,0.32412687){\color[rgb]{0,0,0}\makebox(0,0)[lt]{\lineheight{1.25}\smash{\begin{tabular}[t]{l}\fontsize{9pt}{1em}$\zeta_{n}$\end{tabular}}}}%
		\put(0.43169186,0.39444573){\color[rgb]{0,0,0}\makebox(0,0)[lt]{\lineheight{1.25}\smash{\begin{tabular}[t]{l}\fontsize{9pt}{1em}$p_{n}$\end{tabular}}}}%
		\put(0,0){\includegraphics[width=\unitlength,page=4]{TractsS.pdf}}%
		\put(0.77402178,0.28340401){\color[rgb]{0,0,0}\makebox(0,0)[lt]{\lineheight{1.25}\smash{\begin{tabular}[t]{l}\fontsize{9pt}{1em}$p_{n+1}$\end{tabular}}}}%
		\put(0,0){\includegraphics[width=\unitlength,page=5]{TractsS.pdf}}%
	\end{picture}%
	\endgroup%
	\caption{The tracts $\{S_{n}\}_{n\geq 0}$.}
	\label{figure_tractsS}
\end{figure}

Let $\T$ be the collection of all tracts $T_n$ and $S_{n}$ together with their $2\pi i $-translates, and consider the corresponding $2\pi i$-periodic logarithmic transform $F\colon \T\to \H$. Since $\overline{\T}\subset \H$, $F$ is by definition of disjoint type.

We claim that $F$ is criniferous and that for the external address 
$\s \defeq S_{1} S_{2}S_{3}S_{4}\ldots,$ the endpoint of $J_{\s}$ is not accessible from $\C\setminus J(F)$, implying that $J(F)$ is not a Cantor bouquet; see Observation \ref{obs_acc_endp} . Since the proofs of these facts are closely related to those in the previous construction, here we will just indicate the changes needed for them to carry over to our new setting.
\begin{prop}\label{prop_newHSC} The map $F$ satisfies a head-start condition. 
\end{prop}
\begin{proof}
We only need to argue for the admissible external addresses of the form $\underline{\tau}=\tau_0\tau_1\tau_2\ldots $ such that for infinitely many $k\geq 0$, $\tau_k=S_j+2\pi i m$ for some $j\geq 0$ and $m\in \Z$, since all other cases are already covered or follow from arguments in the proof of Proposition~\ref{prop_HSC}.

We claim that any such $\underline{\tau}$ satisfies the head-start condition for the map $\phi$ defined in \eqref{eq_phi_1}, with the choice of parameters 
\begin{equation}\label{eq_param2}
	\alpha_n\defeq \frac{a_n}{n^3} \quad \text{ and } \quad y_n\defeq 100a_{n+1}, \quad \text{ for each } n\geq 0.
\end{equation}
Indeed, let us fix $z,w\in \overline{\tau_k}$ with $F(z),F(w)\in \tau_{k+1}$ for some $k\geq 0$.  Assume that $\Rea(w) > \phi( \Rea(z))$. Then, in place of \textit{Claim 1} in the proof of Proposition~\ref{prop_HSC}, we must show that if $\tau_k= S_j+2\pi i m$ for some $j\geq 0$ and $m\in \Z$, then $\Rea F(z)<\alpha_{n_x+2}$. For that, note that any $S_n$ equals $T_0$ up to an appropriate translation and scaling by a factor of $1/n^2$. Thus, a calculation shows that whenever $z\in S_j$ for some $j\leq n$ and $\Rea z>2a_j,$ then
\begin{equation}\label{eq_boundFexp}
	\vert F(z) \vert <\exp(C n^2\Rea z)
\end{equation} 
for some constant $C>0$. Then, since $j\leq n_x$ and $\Rea z<\alpha_{n_x+1}$, see \eqref{eq_bounds_xy}, 
$$\Rea F(z)\leq c_1\exp(n_x^2\alpha_{n_x+1})\leq c_1\exp( a_{n_x+1}/n_x)= c_1a_{n_x+2}^{1/n_x}<a_{n_x+2}/(n_x+2)^3=\alpha_{n_x+2},$$
as desired. \textit{Claim 1} in the rest of the proof of Proposition~\ref{prop_HSC} now holds verbatim, as the same calculations apply to the new tracts.
\end{proof}

\begin{prop}\label{prop_nonaccess}Let  $\s \defeq S_{1} S_{2}S_{3}S_{4}\ldots$. Then, the endpoint $e_{\s}$ of $J_{\s}$ is not accessible from $\C\setminus J(F)$. In particular, $J(F)$ is not a Cantor bouquet.
\end{prop}
\begin{proof}
To prove that the endpoint of $J_{\s}$ is not accessible from $\mathbb{C}\setminus J(F)$, for each $n\geq 0$, consider the external address
$$\underline{\tau}_n\defeq S_{1} S_{2}S_{3}\ldots S_{n-1}T_nT_0T_0\ldots,$$
and let $\hat{\gamma}_n$ be the union of $F^n(J_{\underline{\tau}_n})=J_{\sigma^n(\underline{\tau}_n)}$ with the vertical segment $\hat{L}_n$ that joins the endpoint of $J_{\sigma^n(\underline{\tau}_n)}$ to a point of $J_{\sigma^n(\underline{\tau}_n)}$ of imaginary part  greater than $\pi(1+1/(3n+1))$. Here, $\sigma$ denotes the shift map on sequences, that is, $\sigma(\tau_0\tau_1\tau_2\ldots)=\tau_1\tau_2\ldots$. Note that $\overline{T_n}\cap \hat{\gamma}_n$ is connected and separates $F^n(e_{\s})$ from $0\in \C\setminus J(F)$. Then, by Janiszewski's Theorem, see  \cite[p. 110]{Newman}, so does $\hat{\gamma}_n$. Next, let $\gamma_n\defeq F_{\s}^{-n}(\hat{\gamma}_n)\subset S_1$ and $L_n\defeq F_{\s}^{-n}(\hat{L}_n)$, and note that $\gamma_n$ separates $e_{\s}$ from $0$. In particular, any curve connecting $e_{\s}$ to $0$ in the complement of $J(F)$ would have to intersect $L_n$ for all $n\geq 0$. However, by \eqref{eq_expansion2}, the length of $L_n$ tends to $0$ as $n\to \infty$, and by construction and \eqref{eq_JF_disj}, any limiting point of   the sequence $(L_n)_n$ must belong to $J(F) \cap J_{\s}.$ Nevertheless, note that if we choose some $n_0\geq 0$ such that $a_1<b_{n_0}/n^3_0$, then, using \eqref{eq_boundFexp}, we have that $F^k(b_{n_0}/n^3_0)\leq b_{n_0+k}/(n_0+k)^3<b_{n_0+k}$ for all $k\geq 0$.
Since $\hat{L}_n$ is a vertical segment at real part $b_n$, we have that the distance between any segment $L_n$ and $e_{\s}$ is uniformly bounded from below, leading to a contradiction. 

We have shown that the endpoint $e_{\s}$ is not accessible from $\C\setminus J(F)$. Since all endpoints of a Cantor bouquet are accessible from its complement, see Observation \ref{obs_acc_endp}, we conclude that $J(F)$ is not a Cantor bouquet. 
%Separation theorem due to Janiszewski :if $K_1$ and $K_2$ are compact subsets of the sphere whose intersection is connected, then a pair of points that is not separated by either $K_1$ or $K_2$ is also not separated by the union $K_1\cup K_2$. \cite[p. 110]{Newman}.
\end{proof}

\begin{proof}[Proof of Theorem \ref{thm_intro_crinNoCantorB}] We have constructed $F\in \BlogP$ of disjoint type such that each component of $J(F)$ is an arc connecting a finite endpoint to infinity, Proposition~\ref{prop_newHSC} together with Theorem \ref{thm_HS_then_crin}, one these arcs $\gamma$  has a non-accessible endpoint and $J(F)$ is not a Cantor bouquet, Proposition \ref{prop_nonaccess}. Since by \cite[Theorem~2.3]{lasse_arclike} only endpoints might be accessible from $\C\setminus J(F)$, no point in $\gamma$ is accessible. The result now follows from Theorem \ref{thm_Bishop}, using \cite{lasseRigidity} and noting that all these properties are preserved under homeomorphisms. 
\end{proof}

\subsection{Variation 2: uniformly bounded addresses }
Let $F$ be either of the two functions that we constructed in this section. If we restrict $F$
  to finitely many of its tracts, 
  then these tracts have uniformly bounded slope
  and bounded wiggling in the sense of \cite[Definition~5.3]{rrrs}, and hence satisfy a uniform head-start condition. So the corresponding subsets
  of $J(F)$ are Cantor bouquets. In particular, if we consider a collection of hairs 
  with \emph{uniformly bounded addresses} (all hairs remain in some finite collection of logarithmic tracts under
  iteration), then
  the union of these hairs is homeomorphic to the product of a totally disconnected set with an interval. Sets of hairs at 
  uniformly bounded addresses play an important role in the study of functions with bounded postsingular set
  (see e.g.~\cite{lasse_dreadlocks}), and it is therefore interesting to ask whether the above phenomenon, 
  namely the continuous dependence of ray tails at uniformly bounded addresses on their endpoints, holds for all criniferous
  functions.
  
  This is not the case, as we can show by a modification of our example. Here the addresses of 
    the accumulating hairs involve just two different tracts: $T_0$ is the same as above,
    while $\hat{T}_1$ is a tract that wiggles exactly above $a_n$ and $b_n$ for every $n$;
    see Figure~\ref{fig:var2}. We define
    $T_1$ as before; that is, $T_1 = \phi_1(V)$, where $\phi_1\colon \Sigma\to\hat{T}_1$ is 
    a conformal isomorphism and $V$ is as in Proposition~\ref{prop_V}. The point $\zeta_1\in \hat{T}_1$ may be chosen 
    to have real part $5$. As before, we obtain a conformal isomorphism $F_1\colon T_1\to \HH$ with 
    $F_1(\zeta_1) = 5$. 
    
    Now consider the function
    $F\in\BlogP$ obtained by $2\pi i\Z$-periodic extension of $F_0$ and $F_1$. 
    The proof that $F$ is criniferous goes through verbatim. In the proof of Proposition~\ref{prop_noCB} 
    the address $\s_k$ now has an entry of $T_1$ at time $k$, and
    all other entries are equal to $T_0$. In particular, this collection of addresses is uniformly bounded in the sense discussed above.
    
    The curve $F_{T_1}^{-1}([5,\infty))$ crosses all of the 
    wiggles of $T_1$. In particular, as we traverse this curve from $\zeta_1$ to $\infty$, for every $k\geq 0$ 
    there is a sub-arc that passes
    from real part $b_{k}-1$ to real part $a_{k}+1$. The image of this arc under 
    $F_{T_0}^{-k}$ connects two points $w_k$ and $z_k$ in $J_{\s_k}(F)$, with $w_k\prec z_k$ and 
    $w_k\to b_0$, $z_k\to a_0$ as $k\to\infty$. Hence we have obtained a bad set with uniformly bounded addresses, as claimed.
    (Since the addresses are bounded, 
     the endpoints of $J_{\s_k}(F)$ are all accessible from $\C\setminus J(F)$, in contrast to Variation 1.)

  \begin{figure}[tb]
	\centering \def\svgwidth{\linewidth}
	%% Creator: Inkscape 1.3.2 (091e20e, 2023-11-25), www.inkscape.org
%% Accompanies image file 'TractsSvariation.pdf' (pdf, eps, ps)
\begingroup%
  \makeatletter%
  \providecommand\color[2][]{%
    \errmessage{(Inkscape) Color is used for the text in Inkscape, but the package 'color.sty' is not loaded}%
    \renewcommand\color[2][]{}%
  }%
  \providecommand\transparent[1]{%
    \errmessage{(Inkscape) Transparency is used (non-zero) for the text in Inkscape, but the package 'transparent.sty' is not loaded}%
    \renewcommand\transparent[1]{}%
  }%
  \providecommand\rotatebox[2]{#2}%
  \newcommand*\fsize{\dimexpr\f@size pt\relax}%
  \newcommand*\lineheight[1]{\fontsize{\fsize}{#1\fsize}\selectfont}%
  \ifx\svgwidth\undefined%
    \setlength{\unitlength}{476.22047244bp}%
    \ifx\svgscale\undefined%
      \relax%
    \else%
      \setlength{\unitlength}{\unitlength * \real{\svgscale}}%
    \fi%
  \else%
    \setlength{\unitlength}{\svgwidth}%
  \fi%
  \global\let\svgwidth\undefined%
  \global\let\svgscale\undefined%
  \makeatother%
  \begin{picture}(1,0.39285714)%
    \lineheight{1}%
    \setlength\tabcolsep{0pt}%
    \put(0,0){\includegraphics[width=\unitlength,page=1]{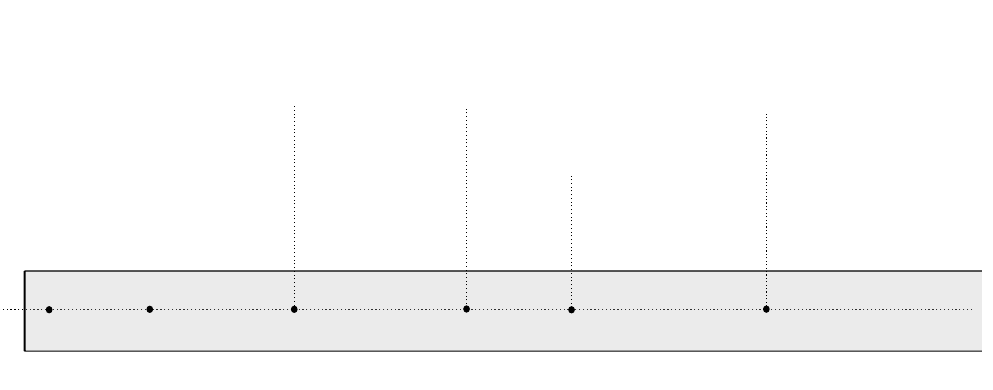}}%
    \put(0.28334007,0.05085446){\color[rgb]{0,0,0}\makebox(0,0)[lt]{\lineheight{1.25}\smash{\begin{tabular}[t]{l}$a_n$\end{tabular}}}}%
    \put(0.45604656,0.05064774){\color[rgb]{0,0,0}\makebox(0,0)[lt]{\lineheight{1.25}\smash{\begin{tabular}[t]{l}$b_n$\end{tabular}}}}%
    \put(0.55912472,0.0526975){\color[rgb]{0,0,0}\makebox(0,0)[lt]{\lineheight{1.25}\smash{\begin{tabular}[t]{l}$a_{n+1}$\end{tabular}}}}%
    \put(0.76223728,0.05497342){\color[rgb]{0,0,0}\makebox(0,0)[lt]{\lineheight{1.25}\smash{\begin{tabular}[t]{l}$b_{n+1}$\end{tabular}}}}%
    \put(0.02605459,0.13677088){\color[rgb]{0,0,0}\makebox(0,0)[lt]{\lineheight{1.25}\smash{\begin{tabular}[t]{l}\fontsize{9pt}{1em}$T_0$\end{tabular}}}}%
    \put(0,0){\includegraphics[width=\unitlength,page=2]{TractsSvariation.pdf}}%
    \put(0.02098988,0.31361889){\color[rgb]{0,0,0}\makebox(0,0)[lt]{\lineheight{1.25}\smash{\begin{tabular}[t]{l}\fontsize{9pt}{1em}$T_1$\end{tabular}}}}%
  \end{picture}%
\endgroup%

	\caption{Sketch of tracts that generate a bad pair in a function with uniformly bounded addresses.}\label{fig:var2}
\end{figure}

\section{Cantor bouquets imply uniform head-start}\label{sec_CB_then_UHSC}

The goal of this section is to prove a more general version of Theorem \ref{thm:headstart}.
 Recall that we defined in Definition~\ref{UHSCentire} what it means for a disjoint-type entire function to satisfy a uniform head-start condition on its Julia set. 
 We now formulate a more general version of this definition, translated to the setting of disjoint-type functions in $\BlogP$.

\begin{defn}[Uniform head-start condition]\label{def_UHSClog} Let $F\colon \T\to H$ in $\BlogP$ of disjoint type, and let $\rho\colon \overline{\T}\to [0,\infty)$ be a continuous function such that $\rho(z)\to+\infty$ as $z\to \infty$. We say that $F$ satisfies a \textit{uniform head-start condition on $J(F)$ with respect to $\rho$} if there is a monotonically increasing
upper semicontinuous function  $\phi \colon [0,\infty)\to [0,\infty)$ with the following properties for all points $z$ and $w$ belonging to the same component of $J(F)$. 
\begin{enumerate}[(i)]
\item If $\rho(w)> \phi(\rho(z))$, then
$\rho(F(w)) > \phi( \rho(F(z))$. 
\item If $z\neq w$, then there is $n \geq 0$ such that either 
$\rho( F^n(w) ) > \phi( \rho( F^n(z)))$ or
$\rho( F^n(z) ) > \phi( \rho(F^n(w)))$. 
\end{enumerate}
\end{defn}

\begin{obs}\label{obs_UHSC} A function $f\in \B$ of disjoint type satisfies a uniform head-start condition  in the sense of Definition~\ref{UHSCentire} if and only if a corresponding disjoint-type logarithmic transform $F$ (as in Proposition~\ref{prop_disjoint_type}) satisfies a uniform head-start condition with respect to $\rho(z)\defeq \Rea z$,  as defined above.
\end{obs}

\begin{prop}\label{prop_UHSC_CB}
Suppose that $F\in \Blog$ is of disjoint type and satisfies a uniform head-start condition on $J(F)$ with respect to $\rho$ for some $\phi$. Then $J(F)$ is a Cantor bouquet.
\end{prop}
\begin{remark}
 This was proved in \cite[Corollary~6.3]{lasseBrushing}, under a more restrictive notion of head-start condition
  (where, in particular, $\rho(z) = \re (z)$). The proof goes through unchanged
  using our definition, but for completeness we include an argument that uses the  
  characterisation in Proposition~\ref{prop:bouquetordering}. 
\end{remark}
\begin{proof}
Let $F$ be as in the statement. Suppose that $z$ and $w$ are in the same connected component of $J(F)$. Then we define the order
\begin{equation}
z\prec w  \quad \text{ if and only if } \quad \rho(F^n(w)) > \phi(\rho(F^n(z))) \text { for some } n\geq 0.
\end{equation}
We also define $z\prec \infty$ for all $z\in J(F)$. By definition of the uniform head-start condition, this is a well-defined transitive relation, and total on each component of $J(F)$; compare to \cite[Definition and Lemma~4.3]{rrrs}. If we show that this order satisfies properties \ref{item:aprecinfty}-\ref{item:accessibility} in Proposition \ref{prop:bouquetordering}, it follows that $J(F)$ is a Cantor bouquet.

Properties \ref{item:aprecinfty} and \ref{item:comparable} are a direct consequence of the definition of the order. Let us fix $z\prec w$ such that $\rho(F^n(w)) > \phi(\rho(F^n(z)))$ for all $n$ large enough. Moreover, let $z_j\to z$ and $w_j\to w$ be sequences such that for each $j \geq 0$, both $z_j$ and $w_j$ belong to the same component of $J(F)$. Then, by upper semicontinuity of $\phi$ and continuity of $\rho$,
$$\limsup_{j\to\infty} \phi(\rho(F^n(z_j))) \leq \phi(\rho( F^n(z)) ).$$ 
This implies that for large $n$, we have $\rho(F^n(w_j)) > \phi(\rho(F^n(z_j)))$, and hence $z_j \prec w_j$. In particular, $w_j \nprec z_j$, and so \ref{item:ordercontinuity} follows. Thus, see Observation \ref{obs_arc}, 
each connected component of $J(F)$ is an arc whose endpoint is a minimal point. Since, by \cite[Theorem 2.3]{lasse_arclike}, only endpoints of $J(F)$ are accessible from $\C\setminus J(F)$, we deduce \ref{item:accessibility}. 
\end{proof}

 We also introduce the following geometric property of logarithmic tracts, which generalises the bounded-slope condition and is closely related to \cite[Definition~7.1]{lasse_arclike}.
\begin{defn}[Uniformly anguine] \label{defn_anguine} Let $F\colon \mathcal{T}\to H$ in $\BlogP$ be of disjoint type, let $\rho\colon \overline{\T}\to [0,\infty)$ be a continuous function such that $\rho(z)\to+\infty$ as $\re z\to \infty$ and let $\Gamma$ be a curve connecting the boundary of $H$ to $+\infty$ without intersecting $\overline{\mathcal{T}}$. We say that $F$ is \emph{uniformly anguine} with respect to $\rho$ if, for all $t\geq 0$, the set of points of $\rho^{-1}(t)$ that sits between two given adjacent $2\pi i\Z$-translates of $\Gamma$ has uniformly bounded hyperbolic diameter in $H$.
\end{defn}

\begin{remark}
 By $2\pi i$-periodicity of $F$ and $\mathcal{T}$, the definition is independent of the curve $\Gamma$ chosen.
\end{remark}

\begin{obs}[Bounded slope implies uniformly anguine] \label{obs_bslope} In \cite[Definition~7.2]{lasse_arclike}, a disjoint-type $F\colon \mathcal{T}\to H$ in $\Blog$ is defined to have \emph{bounded slope} if there exists a curve $\Gamma \colon [0,\infty) \to \mathcal{T}$ and a constant $K>0$ such that $\Rea \Gamma(t) \to +\infty$ as $t \to \infty$ and
\[\vert \Ima \Gamma(t) \vert \leq K \cdot \Rea \Gamma(t)\] for all $t\geq 0$. In other words, all tracts of $F$ eventually lie within some fixed sector of opening angle less than $\pi$ around the positive real axis. In particular, $F$ is uniformly anguine with respect to $\rho(z)\defeq \Rea z$; see \cite[Remark 7.5]{lasse_arclike} for details.
\end{obs}

With this terminology in place, we can state our strengthened version of Theorem~\ref{thm:headstart}, for the class $\BlogP$. 

\begin{thm} \label{thm_CBthenHS} Suppose that $F\colon \T\to H$ in $\Blog$ is of disjoint type, uniformly anguine with respect to some $\rho$, and that $J(F)$ is a Cantor bouquet. Then, $F$ satisfies a uniform head-start condition on $J(F)$ with respect to $\rho$. 
\end{thm}

\begin{proof}[Proof of Theorem \ref{thm:headstart} using Theorem \ref{thm_CBthenHS}]
The fact that the Julia set of any $f\in \B$ of disjoint type that satisfies a uniform head-start condition for $z \mapsto \vert z \vert$ on $J(f)$ is a Cantor bouquet is \cite[Corollary~6.3]{lasseBrushing}; see also Proposition \ref{prop_UHSC_CB}. Next, suppose that $f\in \B$ is of disjoint type and $J(f)$ is a Cantor bouquet. Then, by Proposition \ref{prop_disjoint_type}, possibly after applying an affine conjugacy, $f$ has a disjoint type logarithmic transform $F$. Note that  $J(F)$ is a Cantor bouquet; see the proof of \cite[Corollary~6.3]{lasseBrushing}. Assume additionally that $F$ has bounded slope as defined in Observation~\ref{obs_bslope}. This implies that $F$ is uniformly anguine with respect to  $\rho(z)\defeq \Rea z$. Then, by Theorem \ref{thm_CBthenHS}, $F$ satisfies a uniform head-start condition on $J(F)$ with respect to $\rho$, and so $f$ satisfies a uniform head-start condition with respect to $\lvert z\rvert$, as desired. 
\end{proof}

For the remainder of the section we assume the hypotheses of Theorem~\ref{thm_CBthenHS}. Since $J(F)$ is by assumption a Cantor bouquet, we can induce $\hat{J}(F)\defeq J(F)\cup \{\infty\}$ with the partial order ``$\prec$'' provided by Proposition \ref{prop:bouquetordering}. Since $F$ is of disjoint type, it follows from \eqref{eq_standard_est} that there exists a universal constant $\Lambda>1$ such that for any tract $T$ in $\T$ and $z,w\in T$,
\begin{equation}\label{eq_expansion}
\dist_H(F(z),F(w)) \geq \Lambda \dist_H(z,w), 
\end{equation} 
where $\dist_H$ denotes the distance on the hyperbolic metric in $H$; see \cite[Lemma~2.1]{rrrs} or \cite[Lemma 3.3]{BarKarp_codingtrees} for details. 

Since $F$ is uniformly anguine with respect to $\rho$, there exists a constant $B>0$ such that for all $t\geq 0$, the set of points of $\rho^{-1}(t)$ that sits between two given adjacent $2\pi i\Z$-translates of $\Gamma$ has hyperbolic diameter in $H$ at most $B$. We set
\begin{equation*}
x_0 \defeq \min_{z\in \overline{\T}} \Rea(z), \qquad R_0\defeq \min_{z\in J(F)} \rho(z) \qquad  \text{ and }\qquad \alpha \defeq  \frac{B}{\Lambda-1}.
\end{equation*}
In particular, since $\rho$ moves continuously from any point $w\in J(F)$ with $\rho(w)=R_0$ to infinity, $J(F)$ contains a point $z$ with $\rho(z) = R$ for every $R\geq R_0$. 

For each $R\geq R_0$, we define $\phi(R)$ as the largest number $t$ for which there exist points $z,z',w',w$, all belonging to the closure $\overline{T}$ of the same tract $T$, and such that
      \begin{enumerate}[(a)]
         \item $\rho(z) = R$ and $\rho(w) = t$; \label{item:phi_a}
         \item $z$ and $z'$ can be connected by a curve contained in $\overline{T}$
           whose diameter with respect to the hyperbolic metric of $H$ is at most 
            $\alpha$, as can $w$ and $w'$;\label{item:phi_b}
          \item  $w'\preceq z'$. \label{item:phi_c}
      \end{enumerate}

\begin{remark}In the definition of $\phi(R)$, the points $w',z'$ belong to 
$J(F)$, but $z, w$ need not.
\end{remark}
\begin{lem}
  For all $R\geq R_0$, $\phi(R)$ is defined. Moreover, the function $ \phi\colon R\mapsto \phi(R)$ is non-decreasing and upper
    semicontinuous.        
\end{lem}
\begin{proof}
Let us fix $R\geq R_0$. Note that item \ref{item:accumulatingatinfty} in the definition of the class $\Blog$ implies that for each $R > 0$, there are only finitely many tracts in $\T$, up to translates by integer multiples of $2\pi i$, that intersect the vertical line at real part $R$. In particular, since $\rho$ is continuous and $\rho(z)\to \infty$ as $z\to \infty$, there are only finitely many tracts whose closure contains a point $z$ with $\rho(z)\leq R$. For each such tract $T$, the set $X$ of points in $\overline{T}$ that can be connected to a point $z$ with $\rho(z)\leq R$ by a curve of hyperbolic diameter at most $\alpha$ is compact. Since $J(F)$ is closed, it follows
 that the intersection $A\defeq X\cap J(F)$ is a compact set. Moreover, since $J(F)$ is a Cantor bouquet, and so homeomorphic to a straight brush, by the properties of the latter, the set
$$C\defeq \{b\in J(F) \colon b\preceq a \text{ for some } a\in A \} $$
is compact. Finally, the set $D$ of points in $\overline{T}$ that can be connected to a point of $C$ by a curve of hyperbolic diameter at most $\alpha$ is compact. Since we have deduced that the points $z,w,z',w'$ in the definition of $\phi(R)$ are chosen among those in finitely many compact sets, the maximum in the definition of $\phi(R)$ does indeed exist.
    
To see that $\phi$ is non-decreasing, observe that \ref{item:phi_a}, \ref{item:phi_b} and \ref{item:phi_c} in the definition of $\phi$ hold taking $z=w=z'=w'$, and so $\phi(R)\geq R$ for all $R$. 
    Now let $R\geq R_0$, and let $R'\in [R,\phi(R)]$. Let $z,z',w',w$ be as in the definition of $\phi$, and
    let $\gamma_z, \gamma_w\subset \overline{T}$ be the curves connecting 
    $z$ to $z'$ and $w$ to $w'$. By continuity of $\rho$,
    the curve $\Gamma \defeq  \gamma_w \cup [w',z'] \cup \gamma_z$ contains a point $\zeta$ with $\rho(\zeta)=R'$. It follows that $\phi(R')\geq \phi(R)$:
    \begin{itemize}
      \item if $\zeta \in \gamma_z$, use $\zeta,z',w',w$;
      \item if $\zeta \in [w',z']$, use $\zeta, \zeta, w',w$; 
      \item if $\zeta \in \gamma_w$, use $\zeta,w',w',w$; 
    \end{itemize}
 these choices of points show that $\phi(R)$ is a lower bound for $\phi(R')$.
On the other hand, when $R'\geq \phi(R)$, then $\phi(R')\geq R'\geq \phi(R)$. 
 We conclude that $\phi(R')\geq \phi(R)$ for all $R'\geq R$.
  
To see that $\phi$ is upper semicontinuous, let $R_k\to R$ be a sequence of positive real numbers; we must show that $\phi(R) \geq \rho_0 \defeq \limsup \phi(R_k)$. To see this, for each $k>0$, let us take points $z_k,z_k', w_k',w_k$ as in the definition of $\phi(R_k)$; by passing to a subsequence we may assume that $\phi(R_k)=\rho(w_k) \to \rho_0$ and that each of the above sequences of points converges, with limits $z,z',w,w'$. Then, by Proposition \ref{prop:bouquetordering}\ref{item:ordercontinuity}, $w'\preceq z'$, and so we see that $z,z',w,w'$ satisfy the assumptions on the definition of $\phi(R)$. Hence, $\phi(R)\geq \rho(w) = \rho_0$, as required.
\end{proof}
      
Our goal is to show that $F$ satisfies the uniform head-start condition on $J(F)$ for $\phi$. We begin with the following.

   \begin{obs}\label{obs:secondpiece}
    Suppose that $w\prec z$ and $\rho(z)<\rho(w)$. Then there is $w'$ such that $\rho(w')=\rho(w)$ and $z\prec w'$. In particular, $\dist_H(w,w')\leq B$. 
   \end{obs}
\begin{proof}
 The existence of $w'$ is an immediate consequence of the fact that $\rho(\zeta)$ moves continuously from $\rho(z)$ to
 $\infty$ as $\zeta$ moves along $[z,\infty)$. The bound on the hyperbolic distance between $w$ and $w'$ is a direct consequence of $F$ being uniformly anguine.
\end{proof}
 For each $z\in J(F)$, we denote by $J_z$ the hair of $J(F)$ containing $z$, and recall from \eqref{eq_JF_disj} that $J_z\subset T$ for some tract $T$ in $\T$.
\begin{lem}\label{lem:hookbetweenphi}
 Suppose that $z\in J(F)$, $w\in J_z$ and $\rho(z) \leq \rho(w) \leq \phi(\rho(z))$. Then there are
   $z',w'\in J(F)$ such that $\dist_H(z,z'), \dist_H(w,w') \leq \alpha + B$ and $w'\preceq z'$. Moreover, if $\dist_H(z,w)>3B+2\alpha$, then $\rho(z')< \rho(w')$.
\end{lem}
\begin{proof}
      By definition of $\phi(\rho(z))$, there are $\zeta,\omega,\zeta',\omega'$, all belonging to the closure of the same tract $T$, but not necessarily the tract $z$ belongs to, such that
        \begin{itemize}
           \item
            $\rho(\zeta) = \rho(z)$;
            \item $\rho(\omega)=\phi(\rho(z)) \geq \rho(w)$;
            \item $\zeta$ and $\zeta'$ can be joined by a curve 
               $\gamma_{\zeta}$ whose 
                hyperbolic diameter with respect to the hyperbolic metric of $H$
                is at most $\alpha$;
            \item $\omega$ and $\omega'$ can be joined by a curve $\gamma_{\omega}\subset \overline{T}$ whose diameter % \textbf{(or really length??)}
              with respect to the hyperbolic metric of $H$ is at most $\alpha$;
            \item $\omega'\preceq \zeta'$.
          \end{itemize}  

Let $\Gamma$ be a curve connecting the boundary of $H$ to $+\infty$ without intersecting $\overline{\mathcal{T}}$. Since everything is invariant under translation by $2\pi i$, we may assume that $T$ lies between the same two translates of $\Gamma$ as $J_z$ does. Note that $\gamma_{\omega}\cup [\omega',\zeta'] \cup \gamma_{\zeta}$ is a curve connecting $\omega$ and $\zeta$ and, by assumption, $\rho(\zeta)=\rho(z)\leq \rho(w)\leq \rho(\omega)$. Hence, since $\rho$ is continuous, this curve contains a point $x$ with $\rho(x)=\rho(w)$. By assumption, $x$ is at hyperbolic distance at most $\alpha$ from some point $w'\in [\omega',\zeta']$. Set $z' \defeq \zeta'$. By definition, since $F$ is uniformly anguine with respect to $\rho$, the points $z,\zeta, x$ and $w$ sit between two given adjacent $2\pi i\Z$-translates of $\Gamma$ and $\rho(z)= \rho(\zeta)$, $\rho(x)= \rho(w)$, we have that 
\begin{equation}\label{eq_maxdist}
\max\{\dist_H(z,\zeta),\dist_H(w,x)\}\leq B.
\end{equation}
Hence, $\dist_H(z,z'),\dist_H(w,w')\leq B + \alpha$, as claimed. 

If $\dist_H(z,w)>3B+2\alpha$, then, using \eqref{eq_maxdist}, we have $\dist_H(x,\zeta)>B+2\alpha$. Let $\gamma_x$ be the subcurve in $\gamma_{\omega}\cup [\omega',\zeta'] \cup \gamma_{\zeta}$ joining $x$ and $w'$, and note that $\gamma_x$ might be a singleton. Then, as $F$ is uniformly anguine, we have that $\rho(\gamma_{x})\cap \rho(\gamma_{\zeta})=\emptyset$. Since $\rho(\zeta)<\rho(x)$, $z'\in \gamma_{\zeta}$, $w'\in \gamma_{x}$ and the function $\rho$ is continuous in the connected sets $\gamma_{x}$ and $\gamma_{\zeta}$, we have that $\rho(z')<\rho(w')$, as we wanted to show.
 \end{proof}

 \begin{cor}\label{cor1}
   Let $z\in J(F)$ and $w\in J_z$. If $\rho(w) > \phi(\rho(z))$, then $\rho(F(w)) > \phi(\rho(F(z)))$.
 \end{cor}
 \begin{proof}
   We show the contrapositive. Suppose that $\rho(F(w))\leq \phi(\rho(F(z)))$ and let $T$ be the tract containing $z$ and $w$.
    If $\rho(F(z))\leq\rho(F(w))$, since $F(w)\in J_{F(z)}$, we can apply Lemma~\ref{lem:hookbetweenphi} to $F(z)$ and $F(w)$ and obtain $z',w'$ as in the statement of the lemma. Consider the points $\zeta \defeq F_T^{-1}(z')$ and
   $\omega \defeq F_T^{-1}(w')$. Then, since $w' \prec z'$, $\omega \preceq \zeta$. Furthermore, by \eqref{eq_expansion}, $z$ and $\zeta$ can be connected by a curve in $T$ of hyperbolic diameter (in $H$) at most
      \[ \frac{\alpha + B}{\Lambda} =  \frac{B}{\Lambda}\cdot \left(1+\frac{1}{\Lambda-1}\right) = \frac{B}{\Lambda-1} = \alpha. \]
   The same holds for $w$ and $\omega$. Hence, $z,\zeta,\omega,w$ satisfy the conditions on the definition of $\phi(\rho(z))$, and so we have that $\rho(w) \leq \phi(\rho(z))$.
   
We are left to consider the case $\rho(F(w))<\rho(F(z))$. Note that if $w\prec z$, then the quadruple of points $z,z,w,w$ satisfies the conditions on the definition of $\phi(\rho(z))$, and so $\rho(w) \leq \phi(\rho(z))$. Suppose then that $z\prec w$, and so $F(z)\prec F(w)$. Then, by Observation \ref{obs:secondpiece}, there exists $z'$ such that $F(z)\prec F(w)\prec z'$, $\rho(F(z))=\rho(z')$ and $\dist_H(F(z),z')\leq B$. Hence, if $\zeta \defeq F_T^{-1}(z')$, then the hyperbolic length (in the metric of $H$) of the 
  geodesic of $T$ connecting $z$ and $\zeta$ is bounded by $B/ \Lambda< \alpha$, see \eqref{eq_expansion}. Therefore, the points $z,\zeta,w,w$ satisfy the conditions in the definition of $\phi(\rho(z))$, and so $\rho(w) \leq \phi(\rho(z))$.
\end{proof}

\begin{cor}\label{cor2}
Let $z,w\in J(F)$, where $w\in J_z$ and $z\neq w$. Then there is $n\geq 0$ such that $\rho(F^n(z)) > \phi(\rho(F^n(w)))$, or vice versa. 
\end{cor}
\begin{proof}
Suppose, by contradiction, that this is not the case, that is, that for all $n\geq 0$,
\begin{align}
& \rho(F^n(z)) \leq  \phi(\rho(F^n(w)) \quad  \text { and } \label{eq_l1} \\
& \rho(F^n(w)) \leq  \phi(\rho(F^n(z)). \label{eq_l2}
\end{align}
We may assume without loss of generality that $z\prec w$. By \eqref{eq_expansion}, there is $N\geq 0$ such that $\dist_H(F^n(z), F^n(w))>3B+2\alpha$ for all $n\geq N$. For each $n\geq N$, if $\rho(F^n(z)) <\rho(F^n(w))$, then we use \eqref{eq_l2} and apply Lemma~\ref{lem:hookbetweenphi} to obtain points $\zeta_n$ and $\omega_n$ with 
\begin{equation}\label{eq_nbound}
\dist_H(\zeta_n,F^n(z)), \dist_H(\omega_n,F^n(w))\leq 2B + \alpha \quad \text{ and } \quad \omega_n \preceq \zeta_n.
\end{equation}

Otherwise, if $\rho(F^n(w)) <\rho(F^n(z))$, then we use \eqref{eq_l1} and apply Lemma~\ref{lem:hookbetweenphi}. This time, we obtain a pair of points $z_n'$ and $w_n'$ such that $z_n'\prec w_n'$, $\rho(w_n')<\rho(z_n')$ and 
\begin{equation*}
\dist_H(z_n',F^n(z)), \dist_H(w'_n,F^n(w))\leq B + \alpha.
\end{equation*}
Then, by Observation \ref{obs:secondpiece}, there exists $\zeta_n$ such that $w_n' \prec \zeta_n$ and so that \eqref{eq_nbound} holds with $\omega_n\defeq w'_n$. 

Let $z_n$ and $w_n$ be the pullback of $\zeta_n$ and $\omega_n$ along the external address $\s$ of $z$; that is, $z_n \defeq F_{\s}^{-n}(\zeta_n)$ and $w_n \defeq F_{\s}^{-n}(\omega_n)$. Then, by \eqref{eq_nbound} combined with \eqref{eq_expansion}, we have that $z_n\to z$, $w_n\to w$ as $n\to \infty$. Moreover, $w_n\preceq z_n$ for all $n$, while $z \prec w$; see \eqref{eq_preim_order}. This contradicts Proposition~\ref{prop:bouquetordering}\ref{item:ordercontinuity}.
\end{proof}
 
 Combining Corollaries \ref{cor1} and \ref{cor2}, we conclude that $F$ satisfies a uniform head-start condition on $J(F)$ with respect to $\rho$, and so we have proved Theorem \ref{thm_CBthenHS}.

\bibliographystyle{alpha}
\bibliography{biblioComplex}

\newcommand{\etalchar}[1]{$^{#1}$}
\begin{thebibliography}{{Par}22b}

\bibitem[Ahl73]{Ahlfors_conf_inv73}
L.~V. Ahlfors.
\newblock {\em Conformal invariants: topics in geometric function theory}.
\newblock McGraw-Hill Series in Higher Mathematics. McGraw-Hill Book Co., New
  York-D\"{u}sseldorf-Johannesburg, 1973.

\bibitem[AO93]{aartsoversteegen}
Jan~M. Aarts and Lex~G. Oversteegen.
\newblock {The geometry of Julia sets}.
\newblock {\em Trans. Amer. Math. Soc.}, 338(2):897--918, 1993.

\bibitem[AR17]{nada}
N.~Alhabib and L.~{Rempe-Gillen}.
\newblock Escaping endpoints explode.
\newblock {\em Comput. Methods Funct. Theory}, 17(1):65--100, 2017.

\bibitem[Bar07]{Baranski_Trees}
K.~Bara\'{n}ski.
\newblock Trees and hairs for some hyperbolic entire maps of finite order.
\newblock {\em Math. Z.}, 257(1):33--59, 2007.

\bibitem[BDH{\etalchar{+}}99]{dghnew1}
Clara Bodel{\'o}n, Robert~L. Devaney, Michael Hayes, Gareth Roberts, Lisa~R.
  Goldberg, and John~H. Hubbard.
\newblock Hairs for the complex exponential family.
\newblock {\em Internat. J. Bifur. Chaos Appl. Sci. Engrg.}, 9(8):1517--1534,
  1999.

\bibitem[BDH{\etalchar{+}}00]{dghnew2}
Clara Bodel{\'o}n, Robert~L. Devaney, Michael Hayes, Gareth Roberts, Lisa~R.
  Goldberg, and John~H. Hubbard.
\newblock Dynamical convergence of polynomials to the exponential.
\newblock {\em J. Differ. Equations Appl.}, 6(3):275--307, 2000.

\bibitem[Bis15]{bishopClassB}
Christopher~J. Bishop.
\newblock {Models for the Eremenko–Lyubich class}.
\newblock {\em Journal of the London Mathematical Society}, 92(1):202--221,
  2015.

\bibitem[BJR12]{lasseBrushing}
K.~Bara{\'{n}}ski, X.~Jarque, and L.~Rempe.
\newblock Brushing the hairs of transcendental entire functions.
\newblock {\em Topology and its Applications}, 159(8):2102--2114, 2012.

\bibitem[BK07]{BarKarp_codingtrees}
K.~Bara{\'{n}}ski and B.~Karpi{\'{n}}ska.
\newblock Coding trees and boundaries of attracting basins for some entire
  maps.
\newblock {\em Nonlinearity}, 20(2):391--415, 2007.

\bibitem[BM07]{beardon_minda}
A.F. Beardon and D.~Minda.
\newblock {\em The hyperbolic metric and geometric function theory}.
\newblock In: Quasiconformal Mappings and their Applications. Narosa, New
  Delhi, 2007.

\bibitem[BR20]{lasse_dreadlocks}
A.M. Benini and L.~{Rempe}.
\newblock A landing theorem for entire functions with bounded post-singular
  sets.
\newblock {\em Geom. Funct. Anal.}, 30:1465–1530, 2020.

\bibitem[BR21]{pseudoarcs}
Tania~Gricel Benitez and Lasse Rempe.
\newblock A bouquet of pseudo-arcs.
\newblock {\em Preprint, arXiv:2105.10391}, 2021.

\bibitem[Cur91]{curry_continua}
Stephen~B. Curry.
\newblock One-dimensional nonseparating plane continua with disjoint
  {$\epsilon$}-dense subcontinua.
\newblock {\em Topology Appl.}, 39(2):145--151, 1991.

\bibitem[DG87]{devaneygoldberg}
Robert~L. Devaney and Lisa~R. Goldberg.
\newblock {Uniformization of attracting basins for exponential maps}.
\newblock {\em Duke Math. J.}, 55(2):253--266, 1987.

\bibitem[DGH86]{dgh}
Robert~L. Devaney, Lisa~R. Goldberg, and John~H. Hubbard.
\newblock A dynamical approximation to the exponential map by polynomials.
\newblock Preprint, MSRI Berkeley, 1986.
\newblock published as \cite{dghnew1,dghnew2}.

\bibitem[DK84]{devaney_Krych}
R.~L. Devaney and M.~Krych.
\newblock {Dynamics of $\exp(z)$}.
\newblock {\em Ergodic Theory and Dynamical Systems}, 4(1):35–52, 1984.

\bibitem[EL92]{eremenkoclassB}
A.~{Erë}menko and M.~Lyubich.
\newblock Dynamical properties of some classes of entire functions.
\newblock {\em Ann. Inst. Fourier (Grenoble)}, 42(4):989--1020, 1992.

\bibitem[Ere89]{eremenko1989}
Alexandre~{\`E}. Eremenko.
\newblock On the iteration of entire functions.
\newblock In {\em Dynamical systems and ergodic theory (Warsaw, 1986)},
  volume~23 of {\em Banach Center Publ.}, pages 339--345. PWN, Warsaw, 1989.

\bibitem[Fat26]{Fatou}
P.~Fatou.
\newblock Sur l'it\'{e}ration des fonctions transcendantes enti\`eres.
\newblock {\em Acta Math.}, 47(4):337--370, 1926.

\bibitem[FGL81]{arcsmooth}
J.~B. Fugate, G.~R. Gordh, Jr., and Lewis Lum.
\newblock Arc-smooth continua.
\newblock {\em Trans. Amer. Math. Soc.}, 265(2):545--561, 1981.

\bibitem[Kar03]{karpinskaaccessible}
Bogus\l{}awa Karpińska.
\newblock On the accessible points in the {J}ulia sets of some entire
  functions.
\newblock {\em Fund. Math.}, 180:89--98, 2003.

\bibitem[Mil00]{milnorrays}
John Milnor.
\newblock Periodic orbits, externals rays and the {M}andelbrot set: an
  expository account.
\newblock {\em Ast\'erisque}, (261):xiii, 277--333, 2000.
\newblock G\'eom\'etrie complexe et syst\`emes dynamiques (Orsay, 1995).

\bibitem[Mil06]{milnor_book}
J.~Milnor.
\newblock {\em Dynamics in One Complex Variable. (AM-160) - Third Edition}.
\newblock Annals of Mathematics Studies. Princeton University Press, 2006.

\bibitem[Nad92]{nadler_continuum}
S.~Nadler.
\newblock {\em Continuum Theory: An Introduction}.
\newblock Chapman \& Hall/CRC Pure and Applied Mathematics. Taylor \& Francis,
  1992.

\bibitem[New51]{Newman}
M.~H.~A. Newman.
\newblock {\em Elements of the topology of plane sets of points}.
\newblock Cambridge University Press, Cambridge, 1951.
\newblock 2nd ed.

\bibitem[{Par}21]{mio_signed_addr}
L.~{Pardo-Sim\'on}.
\newblock Combinatorics of criniferous entire maps with escaping critical
  values.
\newblock {\em Conform. Geom. Dyn.}, 25:51--78, 2021.

\bibitem[{Par}22a]{mio_newCB}
L.~{Pardo-Sim\'on}.
\newblock Criniferous entire maps with absorbing {C}antor bouquets.
\newblock {\em Discrete Contin. Dyn. Syst.}, 42(2):989--1010, 2022.

\bibitem[{Par}22b]{mio_splitting}
L.~{Pardo-Sim\'on}.
\newblock Splitting hairs with transcendental entire functions.
\newblock {\em Int. Math. Res. Not.}, 2022.

\bibitem[Pom92]{pommerenke_boundary}
Ch. Pommerenke.
\newblock {\em Boundary behaviour of conformal maps}, volume 299 of {\em
  Grundlehren der Mathematischen Wissenschaften [Fundamental Principles of
  Mathematical Sciences]}.
\newblock Springer-Verlag, Berlin, 1992.

\bibitem[Rem07]{lasse_questionErem}
Lasse Rempe.
\newblock {On a question of Eremenko concerning escaping components of entire
  functions}.
\newblock {\em Bulletin of the London Mathematical Society}, 39(4):661--666,
  2007.

\bibitem[Rem09]{lasseRigidity}
L.~Rempe.
\newblock Rigidity of escaping dynamics for transcendental entire functions.
\newblock {\em Acta Math.}, 203(2):235--267, 2009.

\bibitem[{Rem}16]{lasse_arclike}
L.~{Rempe-Gillen}.
\newblock Arc-like continua, {J}ulia sets of entire functions, and {E}remenko's
  {C}onjecture.
\newblock {\em Preprint, arXiv:1610.06278v4}, 2016.

\bibitem[Rem23]{lasse_constant}
Lasse Rempe.
\newblock The {E}remenko-{L}yubich constant.
\newblock {\em Bull. Lond. Math. Soc.}, 55(1):113--118, 2023.

\bibitem[RRRS11]{rrrs}
G.~Rottenfu{\ss}er, J.~R\"{u}ckert, L.~Rempe, and D.~Schleicher.
\newblock Dynamic rays of bounded-type entire functions.
\newblock {\em Annals of Mathematics (2)}, 173(1):77--125, 2011.

\bibitem[RS09]{bifurcations}
Lasse Rempe and Dierk Schleicher.
\newblock Bifurcations in the space of exponential maps.
\newblock {\em Invent. Math.}, 175(1):103--135, 2009.

\bibitem[Six18]{Dave_survey}
D.~J. Sixsmith.
\newblock Dynamics in the {E}remenko-{L}yubich class.
\newblock {\em Conform. Geom. Dyn.}, 22:185--224, 2018.

\end{thebibliography}

\end{document}